\documentclass[a4paper,11pt]{article}

\usepackage{microtype}
\usepackage[plainpages=false,colorlinks=true,
            linkcolor=black,urlcolor=black,citecolor=black]{hyperref}
\usepackage[margin=2.5cm]{geometry}
\usepackage{amsmath,amssymb,amsthm, amsfonts,mathrsfs}
\usepackage{color}
\usepackage{booktabs}
\usepackage{graphicx}
\graphicspath{{./}{./fig/}{./progs/}}

\usepackage{array, tabularx}
\usepackage{paralist} \usepackage{verbatim} \usepackage{subfig} \usepackage{xcolor,marginnote,enumitem}
\usepackage[numbers,sort]{natbib}

\newtheorem{lemma}{Lemma}
\newtheorem{theorem}{Theorem}
\newtheorem{proposition}{Proposition}
\newtheorem{corollary}{Corollary}

\theoremstyle{remark}
\newtheorem{remark}{Remark}

\theoremstyle{definition}
\newtheorem{definition}{Definition}

\DeclareMathOperator{\dom}{dom}
\DeclareMathOperator{\id}{id}
\DeclareMathOperator*{\wlim}{w-lim}
\DeclareMathOperator{\erg}{erg}
\DeclareMathOperator{\diag}{diag}
\DeclareMathOperator{\SGS}{SGS}
\DeclareMathOperator{\SSOR}{SSOR}
\DeclareMathOperator{\TV}{TV}
\DeclareMathOperator{\Div}{div}

\DeclareMathOperator{\res}{res}
\DeclareMathOperator{\est}{est}

\newcommand{\grad}{\nabla}

\newcommand{\lookUp}[1]{}

\newcommand{\seq}[1]{\{{#1}\}}
\newcommand{\sett}[1]{\{{#1}\}}
\newcommand{\abs}[2][]{|{#2}|_{#1}}
\newcommand{\Bigabs}[2][]{\Bigl|{#2}\Bigr|_{#1}}
\newcommand{\norm}[2][]{\|{#2}\|_{#1}}
\newcommand{\bignorm}[2][]{\bigl\|{#2}\bigr\|_{#1}}
\newcommand{\scp}[3][]{\langle{#2},{#3}\rangle_{#1}}
\newcommand{\Bigscp}[3][]{\Bigl\langle{#2},{#3}\Bigr\rangle_{#1}}
\newcommand{\weakto}{\rightharpoonup}

\newcommand{\subgrad}{\partial}

\newcommand{\RR}{\mathbf{R}}
\newcommand{\NN}{\mathbf{N}}

\newcommand{\mF}{\mathcal{F}}
\newcommand{\mG}{\mathcal{G}}
\newcommand{\mK}{\mathcal{K}}
\newcommand{\mH}{\mathcal{H}}
\newcommand{\mM}{\mathcal{M}}
\newcommand{\mA}{\mathcal{A}}
\newcommand{\mL}{\mathcal{L}}
\newcommand{\mO}{\mathcal{O}}
\newcommand{\mI}{\mathcal{I}}

\newcommand{\gap}{\mathfrak{G}}
\newcommand{\err}{\mathfrak{E}}

\DeclareMathAlphabet{\mathbfit}{OML}{cmm}{b}{it}
\newcommand{\mo}{\mathbfit{o}}

\begin{document}

\title{Accelerated Douglas--Rachford methods
for the solution
of convex-concave saddle-point problems}
\author{Kristian Bredies\thanks{Institute for Mathematics and
Scientific Computing, University of Graz, Heinrichstra\ss{}e 36, A-8010 Graz,
Austria. Email: \href{mailto:kristian.bredies@uni-graz.at}{kristian.bredies@uni-graz.at}. The Institute for Mathematics and Scientific Computing is
a member of NAWI Graz (\url{http://www.nawigraz.at/}).} 
\and Hongpeng Sun\thanks{Institute for Mathematical Sciences,
Renmin University of China, No.~59, Zhongguancun Street, Haidian District,
Beijing, 100872 Beijing, People's Republic of China.
Email: \href{mailto:hpsun@amss.ac.cn}{hpsun@amss.ac.cn}.}}

\maketitle

\begin{abstract}
  We study acceleration and preconditioning strategies for
  a class of Douglas--Rachford methods aiming at the solution
  of convex-concave saddle-point problems associated with
  Fenchel--Rockafellar duality. While the basic iteration converges
  weakly in Hilbert space with $\mO(1/k)$ ergodic convergence of
  restricted primal-dual gaps, acceleration can be achieved under
  strong-convexity assumptions. Namely, if either the primal or dual
  functional in the saddle-point formulation
  is strongly convex, then the method can be modified to
  yield $\mO(1/k^2)$   ergodic convergence. 
  In case of both functionals
  being strongly convex, similar modifications lead to an asymptotic
  convergence of $\mO(\vartheta^k)$ for some $0 < \vartheta < 1$.
  All methods allow in particular for preconditioning, i.e., the
  inexact solution of the implicit linear step in terms of linear
  splitting methods with all convergence rates being maintained.
  The efficiency of the proposed methods is verified and compared
  numerically, especially showing competitiveness with
  respect to state-of-the-art accelerated algorithms.
\end{abstract}



\section{Introduction}

In this work we are concerned with Douglas--Rachford-type algorithms
for solving the saddle-point problem
\begin{equation}
  \label{eq:saddle-point-prob}
  \min_{x \in \dom \mF} \max_{y \in \dom \mG} \
  \scp{\mK x}{y} + \mF(x) - \mG(y)
\end{equation}
where $\mF: X \to {]{-\infty,\infty}]}$, $\mG: Y \to {]{-\infty,\infty}]}$
are proper, convex and lower semi-continuous functionals on the
real Hilbert spaces $X$ and $Y$, respectively, and $\mK: X \to Y$ is a
linear and continuous mapping. Problems of this type commonly occur as
primal-dual formulations in the context of Fenchel--Rockafellar duality.
In case the latter holds, solutions of~\eqref{eq:saddle-point-prob} are
exactly solutions of the primal-dual problem
\begin{equation}
  \label{eq:primal-dual-prob}
  \min_{x \in X} \ \mF(x) + \mG^*(\mK x), \qquad
  \max_{y \in Y} \ -\mF^*(-\mK^* y) - \mG(y)
\end{equation}
with $\mF^*$ and $\mG^*$ denoting the Fenchel conjugate functionals of
$\mF$ and $\mG$, respectively.
We propose and study a class of Douglas--Rachford iterative algorithms for
the solution of~\eqref{eq:saddle-point-prob} with a focus on acceleration and
preconditioning strategies. The main aspects are threefold: First, the
algorithms only utilize specific implicit steps, the proximal mappings of
$\mF$ and $\mG$ as well as the solution of linear equations. In particular,
they converge independently from the choice of step-sizes associated with
these operations. Here, convergence of the iterates is in the weak sense
and we have a $\mO(1/k)$ rate for restricted primal-dual gaps
evaluated at associated ergodic sequences.
Second, in case of strong convexity of $\mF$ or $\mG$,
the iteration can be modified to yield accelerated convergence.
Specifically, in case of strongly convex $\mF$, the associated primal sequence
$\seq{x^k}$ converges strongly at rate $\mO(1/k^2)$ in terms of the
quadratic norm distance while there is weak convergence for the associated
dual sequence $\seq{y^k}$. The rate $\mO(1/k^2)$ also follows
for restricted primal-dual gaps evaluated at associated ergodic sequences.
Moreover, in case of $\mF$ and $\mG$ both being strongly convex, another
accelerating modification yields strongly convergent primal and dual
sequences with a rate of $\mo(\vartheta^k)$ in terms of squared norms
for some $0 < \vartheta < 1$. Analogously, the full primal-dual gap converges
with the rate $\mO(\vartheta^k)$ when evaluated at associated ergodic sequences.
Third, both the basic and accelerated versions of the Douglas--Rachford
iteration may be amended by preconditioning allowing for inexact solution
of the implicit linear step. Here, preconditioning is incorporated by
performing one or finitely many steps of linear operator splitting method
instead of solving the full implicit equation in each iteration step.
For appropriate preconditioners, the Douglas--Rachford method as well as
its accelerated variants converge with the same asymptotic rates.

To the best knowledge of the authors, this is the first time that
accelerated Douglas--Rachford-type methods for the solution
of~\eqref{eq:saddle-point-prob} are proposed for which the optimal rate
of $\mO(1/k^2)$ can be proven. Nevertheless, a lot of related research
has already been carried out regarding the solution of
\eqref{eq:saddle-point-prob} or \eqref{eq:primal-dual-prob}, the design and convergence
analysis of Douglas--Rachford-type methods as well as acceleration
and preconditioning strategies. For the solution
of~\eqref{eq:saddle-point-prob},
first-order primal-dual methods are commonly used, such as the
popular primal-dual iteration of \cite{CP1} which performs explicit
steps involving the operator $\mK$ and its adjoint as well as proximal steps
with respect to $\mF$ and $\mG$. Likewise, a commonly-used approach
for solving the primal problem in~\eqref{eq:primal-dual-prob} is the
\emph{alternating direction method of multipliers} (ADMM) in which
the arguments of $\mF$ and $\mG$ are treated independently and coupled by
linear constraints involving the operator $\mK$, see \cite{BPC}.
The latter is often identified as a Douglas--Rachford iteration on
the dual problem in~\eqref{eq:primal-dual-prob}, see \cite{DGBA, EP}, however,
as this iteration works, in its most general form, for
maximally monotone operators (see \cite{EP}), it may also be applied directly
on the optimality conditions for~\eqref{eq:saddle-point-prob}, see \cite{BLC},
as it is also done in this paper. Besides this application,
Douglas--Rachford-type algorithms are applied for the minimization
of sums of convex functionals (see \cite{CoPe}, for instance), potentially
with additional structure (see \cite{BH,D}, for instance).

Regarding acceleration strategies, the work \cite{YU} firstly introduced
approaches to minimize a class of possibly non-smooth convex functional
with optimal complexity $\mO(1/k^2)$ on the functional error. Since then,
several accelerated algorithms have been designed,
most notably FISTA \cite{BT} which is applicable for the
primal problem in~\eqref{eq:primal-dual-prob} when $\mG^*$ is smooth, and
again, the acceleration strategy presented in \cite{CP1,CP}, which is
applicable when $\mF$ is strongly convex and yields, as in this paper,
optimal rates for restricted primal-dual gaps evaluated at ergodic
sequences. Recently, an accelerated Douglas--Rachford method for the
minimization of the primal problem in~\eqref{eq:primal-dual-prob} was
proposed in \cite{PSB}, which, however, requires $\mG$ to be quadratic and
involves step-size constraints. In contrast, the acceleration strategies
for the solution of~\eqref{eq:saddle-point-prob}
in the present paper only require strong convexity of the primal
(or dual) functional, which are the same assumptions as in \cite{CP1}.
We also refer to \cite{OCLP,GDSB} for
accelerated ADMM-type algorithms as well as to \cite{Nbook,T} for a
further overview on acceleration strategies.

Finally, preconditioning techniques for the solution
of~\eqref{eq:saddle-point-prob} or~\eqref{eq:primal-dual-prob} can,
for instance, be found in \cite{POC} and base on modifying the
Hilbert-space norms leading to different proximal mappings. In
particular, for non-linear proximal mappings, this approach can be quite
limiting with respect to the choice of preconditioners and often, only
diagonal preconditioners are applicable in practice. The situation is
different for linear mappings where a variety of preconditioners is
applicable. This circumstance has been exploited in \cite{BS1,BS2} where
Douglas--Rachford methods for the solution~\eqref{eq:saddle-point-prob} have
been introduced that realize one or finitely iteration steps of
classical linear operator splitting methods, as it is also done in the
present work. The latter allows in particular for the inexact solution of the
implicit linear step without losing convergence properties or introducing any
kind of error control. Besides these approaches, the Douglas--Rachford
iteration and related methods may be preconditioned by so-called
\emph{metric selection}, see \cite{GB}.
This paper focuses, however, on obtaining linear
convergence only and needs restrictive strong convexity and
smoothness assumptions on the involved functionals. In contrast, the
preconditioning approaches in the present work are applicable for the
proposed Douglas--Rachford methods for the solution
of~\eqref{eq:saddle-point-prob} and, in particular, for the accelerated
variants without any additional assumptions.

The outline of the paper is as follows. First, in Section~\ref{sec:basic-dr},
the Douglas--Rachford-type algorithm is derived from the optimality
conditions for~\eqref{eq:saddle-point-prob} which serves as a basis for
acceleration. We prove weak convergence to saddle-points, an asymptotic
ergodic rate of $\mO(1/k)$ for restricted primal-dual gaps and shortly
discuss preconditioning. This iteration scheme is then modified, in
Section~\ref{sec:acceleration}, to yield accelerated iterations in case
of strong convexity of the involved functionals. In particular, we
obtain the ergodic rates $\mO(1/k^2)$ and $\mO(\vartheta^k)$, respectively,
in case that one or both functionals are strongly convex. For these
accelerated schemes, preconditioning is again discussed. Numerical
experiments and comparisons are then shown in
Section~\ref{sec:numerics} for the accelerated schemes and variational
image denoising with quadratic discrepancy and total variation ($\TV$)
as well as Huber-regularized total variation. We conclude with a summary
of the convergence results and an outlook in Section~\ref{sec:summary}.

\section{The basic Douglas--Rachford iteration}
\label{sec:basic-dr}

To set up for the accelerated Douglas--Rachford methods, let us first
discuss the basic solution strategy for
problem~\eqref{eq:saddle-point-prob}. We proceed by writing the iteration
scheme as a proximal point algorithm associated with as degenerate metric,
similar to \cite{BS1,BS2}. Here, however, the reformulation differs in the
way how the linear and non-linear terms are treated and by the fact that
step-size operators are introduced. The latter allow, in particular,
for different primal and dual step-sizes, a feature that is crucial
for acceleration, as it will turn out in Section~\ref{sec:acceleration}.

The proposed Douglas--Rachford iteration is derived as follows.
First observe that in terms of subdifferentials, saddle-point
pairs $(x,y) \in X \times Y$ can be characterized by the inclusion relation
\begin{equation}
  \label{eq:saddle-point-optimality}
  \begin{pmatrix}
    0 \\ 0
  \end{pmatrix}
  \in
  \begin{pmatrix}
    0 & \mK^* \\ -\mK & 0
  \end{pmatrix}
  \begin{pmatrix}
    x \\ y
  \end{pmatrix}
  +
  \begin{pmatrix}
    \partial \mF & 0 \\ 0 & \partial \mG
  \end{pmatrix}
  \begin{pmatrix}
    x \\ y
  \end{pmatrix}.
\end{equation}
This may, in turn, be written as the finding the root of the sum of
two monotone operators:
Denoting, $Z = X \times Y$, $z = (x,y)$,
we find the equivalent representation
\[
0 \in Az + Bz
\]
with the data
\begin{equation}
  \label{eq:saddle-point-data}
  Az = \begin{pmatrix}
    0 & \mK^* \\ -\mK & 0
  \end{pmatrix}
  \begin{pmatrix}
    x \\ y
  \end{pmatrix},
  \quad
  Bz =
  \begin{pmatrix}
    \partial \mF & 0 \\ 0 & \partial \mG
  \end{pmatrix}
  \begin{pmatrix}
    x \\ y
\end{pmatrix}.
\end{equation}
Both $A$ and $B$
are maximally monotone operators on $Z$. Since $A$ has full domain,
also $A + B$ is maximally monotone \cite{BCP}.

Among the many possibilities for solving the root-finding problem
$0 \in Az + Bz$, the Douglas--Rachford method is fully implicit, i.e.,
a reformulation in terms of the proximal point iteration suggests itself.
This can indeed be achieved with a degenerate metric on $Z \times Z$.
Introduce a step-size operator $\Sigma: Z \to Z$
which is assumed to be a continuous, symmetric, positive operator.
Then, setting $\Sigma^{-1} \tilde z \in Az$, we see that
\[
0 \in Az+Bz \quad \Leftrightarrow
\quad
\left\{
\begin{aligned}
0 &\in Bz + \Sigma^{-1} \tilde z, \\
\Sigma^{-1}\tilde z &\in Az
\end{aligned}
\right.
\quad
\Leftrightarrow
\quad
\left\{
\begin{aligned}
0 &\in Bz + \Sigma^{-1} \tilde z, \\
0 &\in -\Sigma^{-1} z + \Sigma^{-1}A^{-1}\Sigma^{-1} \tilde z.
\end{aligned}
\right.
\]
Hence, one arrives at the equivalent problem
\begin{equation}
  \label{eq:saddle_point_inclusion_prob}
  0 \in \mA
  \begin{pmatrix}
    z \\ \tilde z
  \end{pmatrix},
  \qquad
  \mA =
  \begin{pmatrix}
    B & \Sigma^{-1} \\
    -\Sigma^{-1} & \Sigma^{-1}A^{-1}\Sigma^{-1}
  \end{pmatrix}.
\end{equation}
We employ the proximal point algorithm with respect
to the degenerate metric
\[
\mM =
\begin{pmatrix}
  \Sigma^{-1} & -\Sigma^{-1} \\
  -\Sigma^{-1} & \Sigma^{-1}
\end{pmatrix},
\]
which corresponds to
\begin{equation}
  \label{eq:general_proximal_point}
  0 \in \mM
  \begin{pmatrix}
    z^{k+1} - z^k\\ \tilde z^{k+1} - \tilde z^k
  \end{pmatrix}
  +
  \mA
  \begin{pmatrix}
    z^{k+1}\\ \tilde z^{k+1}
  \end{pmatrix},
\end{equation}
or, equivalently,
\[
\left\{
\begin{aligned}
z^{k+1} &= (\id  + \Sigma B)^{-1}(z^k - \tilde z^k), \\
\tilde z^{k+1} &= (\id  + A^{-1}\Sigma^{-1})^{-1}(2z^{k+1}- z^k + \tilde z^k). \\
\end{aligned}
\right.
\]
Using Moreau's identity and substituting $\bar z^k = z^k - \tilde z^k$,
one indeed arrives at
the Douglas--Rachford iteration
\begin{equation}\label{eq:dr:old}
\left\{
\begin{aligned}
z^{k+1} &= (\id  + \Sigma B)^{-1}(\bar z^k), \\
\bar z^{k+1} &= \bar z^k +
(\id  + \Sigma A)^{-1}(2 z^{k+1} - \bar z^k) - z^{k+1}. \\
\end{aligned}
\right.
\end{equation}
Plugging in the data~\eqref{eq:saddle-point-data} as well as
setting, for $\sigma, \tau > 0$,
\begin{equation}
  \label{eq:step_size_operator}
  \Sigma =
  \begin{pmatrix}
    \sigma \id & 0 \\ 0 & \tau \id
  \end{pmatrix},
\end{equation}
the iterative scheme~\eqref{eq:dr:old} turns out to be equivalent to
\begin{equation}
  \label{eq:douglas-rachford}
  \left\{
    \begin{aligned}
      x^{k+1} &= (\id + \sigma \subgrad \mF)^{-1}(\bar x^k), \\
      y^{k+1} &= (\id + \tau \subgrad \mG)^{-1}(\bar y^k), \\
      \bar x^{k+1} &= x^{k+1} - \sigma \mK^* (y^{k+1} + \bar y^{k+1} - \bar y^k), \\
      \bar y^{k+1} &= y^{k+1} + \tau \mK (x^{k+1} + \bar x^{k+1} - \bar x^k).
    \end{aligned}
\right.
\end{equation}
The latter two equations are coupled but can easily be decoupled,
resulting, e.g., with denoting $d^{k+1} = x^{k+1} + \bar{x}^{k+1} -\bar{x}^{k}$,
in an update scheme as shown in
Table~\ref{tab:douglas-rachford}.
With the knowledge of the resolvents $(\id + \sigma \subgrad \mF)^{-1}$,
$(\id + \tau \subgrad \mG)^{-1}$, $(\id + \sigma \tau \mK^*\mK)^{-1}$,
the iteration can be computed.

\begin{table}
  \centering
  \begin{tabular}{p{0.15\linewidth}p{0.75\linewidth}}
    \toprule

    \multicolumn{2}{l}{\textbf{DR}\ \ \ \ \textbf{Objective:}
    \hfill    Solve  \ \ $\min_{x\in \dom \mF} \max_{y \in  \dom \mG}
    \ \langle \mK x,y\rangle + \mF(x) - \mG(y)$
    \hfill\mbox{} }
    \\
    \midrule

    Initialization:
 &
   $(\bar x^0, \bar y^0) \in X \times Y$
   initial guess, $\sigma > 0, \tau > 0$ step sizes
    \\[\medskipamount]
    Iteration:
 &
   \raisebox{1em}[0mm][7.4\baselineskip]{\begin{minipage}[t]{1.0\linewidth}
       {       \renewcommand{\theHequation}{dr}       \begin{equation}\label{eq:douglas-rachford_}\tag{DR}
         \left\{
           \begin{aligned}
             x^{k+1} &= (\id + \sigma \subgrad \mF)^{-1}(\bar x^k) \\
             y^{k+1} &= (\id + \tau \subgrad \mG)^{-1}(\bar y^k) \\
             b^{k+1} &= (2x^{k+1} - \bar x^k) - \sigma \mK^*(2y^{k+1} - \bar y^k) \\
             d^{k+1} &= (\id + \sigma \tau \mK^*\mK)^{-1}b^{k+1} \\
             \bar x^{k+1} &= \bar x^k - x^{k+1} + d^{k+1} \\
             \bar y^{k+1} &= y^{k+1} + \tau \mK d^{k+1}
           \end{aligned}
         \right.
       \end{equation}}
     \end{minipage}}
    \\
        Output:
 &
      $\seq{(x^k,y^k)}$  primal-dual sequence  \\
    \bottomrule
  \end{tabular}
  \caption{The basic Douglas--Rachford iteration for the solution of
    convex-concave saddle-point problems of
    type~\eqref{eq:saddle-point-prob}.}
  \label{tab:douglas-rachford}
\end{table}

\subsection{Convergence analysis}

Regarding convergence of the iteration~\eqref{eq:douglas-rachford},
introduce the Lagrangian $\mL: \dom \mF \times \dom \mG \to \RR$ associated
with the saddle-point problem~\eqref{eq:saddle-point-prob}:
\[
\mL(x,y) = \scp{\mK x}{y} + \mF(x) - \mG(y).
\]
If Fenchel--Rockafellar duality holds, then we may consider
restricted primal-dual
gaps $\gap_{X_0 \times Y_0}: X \times Y \to {]{-\infty,\infty}]}$
associated with $X_0 \times Y_0 \subset \dom \mF \times \dom \mG$
given by
\[
\gap_{X_0 \times Y_0}(x,y)
= \sup_{(x',y') \in X_0 \times Y_0} \ \mL(x,y') - \mL(x',y).
\]
Letting $(x^*,y^*)$ be a primal-dual solution pair
of~\eqref{eq:saddle-point-prob}, we
furthermore consider the following
restricted primal and dual error functionals
\[
\begin{aligned}
  \err^p_{Y_0}(x) &=  \bigl[ \mF(x)
  + \sup_{y' \in Y_0} \scp{\mK x}{y'} - \mG(y') \bigr]
  - \bigl[ \mF(x^*) + \sup_{y' \in Y_0} \scp{\mK x^*}{y'} - \mG(y') \bigr], \\
  \err^d_{X_0}(y) &=
  \bigl[ \mG(y) + \sup_{x' \in X_0} \scp{-\mK^* y}{x'} - \mF(x') \bigr]
  - \bigl[ \mG(y^*) + \sup_{x' \in X_0} \scp{-\mK^*y^*}{x'} - \mF(x') \bigr].
\end{aligned}
\]
The special case $X_0 = \dom \mF$, $Y_0 = \dom \mG$ gives the well-known
primal-dual gap as well as the primal and dual error, i.e.,
\begin{align*}
  \gap(x,y) &= \mF(x) + \mG^*(\mK x) + \mG(y) + \mF^*(- \mK^* y),
  \\
  \err^p(x) &= \mF(x) + \mG^*(\mK x) - \bigl[\mF(x^*) + \mG^*(\mK x^*)
  \bigr],
  \\
  \err^d(y) &= \mG(y) + \mF^*(-\mK^*y) -
              \bigl[\mG(y^*) + \mF^*(-\mK^*y^*) \bigr],
\end{align*}
where $\mF^*$ and $\mG^*$ are again the Fenchel conjugates.
For this reason, the study of the difference of Lagrangian is
appropriate for convergence analysis.
We will later find situations where $\gap_{X_0 \times Y_0}$,
$\err^p_{Y_0}$ and $\err^d_{X_0}$ coincide,
for the iterates of the proposed methods, with $\gap$, $\err^p$ and
$\err^d$, respectively,
for $X_0$ and $Y_0$ not corresponding to the full domains of $\mF$
and $\mG$.
However, the restricted variants are already suitable for measuring
optimality, as the following proposition shows.

\begin{proposition}
  \label{prop:restricted-error-func}
  Let $X_0 \times Y_0 \subset \dom \mF \times \dom \mG$ contain a
  saddle-point $(x^*,y^*)$ of~\eqref{eq:saddle-point-prob}. Then,
  $\gap_{X_0 \times Y_0}$, $\err^p_{Y_0}$ and $\err^d_{X_0}$ are non-negative
  and we have
  \[
  \gap_{X_0 \times Y_0}(x,y) = \err^p_{X_0}(x) + \err^d_{Y_0}(y),
  \qquad
  \err^p_{Y_0}(x) \leq \gap_{\sett{x^*} \times Y_0} (x,y),
  \qquad
  \err^d_{X_0}(y) \leq \gap_{X_0 \times \sett{y^*}} (x,y)
  \]
  for each $(x,y) \in X \times Y$.
\end{proposition}
\begin{proof}
  We first show the statements for $\gap_{X_0 \times Y_0}$. As $(x^*,y^*)$
  is a saddle-point, $\mL(x,y^*) - \mL(x^*,y) \geq 0$ for any $(x,y)
  \in X \times Y$, hence $\gap_{X_0 \times Y_0}(x,y) \geq 0$.
  By optimality of saddle-points, it follows that
  \[
  \mG^*(\mK x^*) =
  \scp{\mK x^*}{y^*} - \mG(y^*) \geq \sup_{y' \in Y_0} \scp{\mK x^*}{y'}
  - \mG(y') \geq \scp{\mK x^*}{y^*} - \mG(y^*),
  \]
  and analogously, $\mF^*(-\mK^* y^*) = \sup_{x' \in X_0}
  \scp{-\mK^* y^*}{x'} - \mF(x')$. By Fenchel--Rockafellar duality,
  $\mF(x^*) + \mG^*(\mK x^*) + \mG(y^*) + \mF^*(-\mK^* y^*) = 0$, hence
  \begin{align*}
    \err^p_{Y_0}(x) + \err^d_{X_0}(y)
    &=
      \mF(x) + \sup_{y' \in Y_0} \scp{\mK x}{y'} - \mG(y')
      + \mG(y) + \sup_{x' \in X_0} \scp{-\mK^*y^*}{x'} - \mF(x')
    \\
    & = \sup_{(x',y') \in X_0 \times Y_0} \mL(x,y') - \mL(x',y)
      = \gap_{X_0 \times Y_0}(x,y).
  \end{align*}
  Next, observe that optimality of $(x^*,y^*)$ and the subgradient
  inequality implies
  \begin{multline*}
    \err^p_{Y_0}(x) \geq \mF(x) + \scp{\mK x}{y^*} - \mG(y^*) - \mF(x^*)
    - \scp{\mK x^*}{y^*} + \mG(y^*) \\
    = \mF(x) - \mF(x^*)
    - \scp{- \mK^* y^*}{x - x^*} \geq 0.
  \end{multline*}
  Analogously, one sees that $\err^d_{X_0} \geq 0$.
  Using the non-negativity and the already-proven identity for
  $\gap_{X_0 \times Y_0}$ yields
  $\err^p_{Y_0}(x) \leq \err^p_{Y_0}(x) + \err^d_{\sett{x^*}}(y) =
  \gap_{\sett{x^*} \times Y_0}(x,y)$ and the
  analogous estimate $\err^d_{X_0}(x) \leq \gap_{X_0 \times \sett{y^*}}(x,y)$.
                              \end{proof}
We are particularly interested in the restricted error functionals
associated with bounded sets containing a saddle-point. These can,
however, coincide with the ``full'' functionals under mild assumptions.
\begin{lemma}
  \label{lem:full-gap-errors}
  Suppose a saddle point $(x^*,y^*)$ of $\mL$ exists.
  \begin{enumerate}
  \item
    If $\mF$ is strongly coercive, i.e., $\mF(x)/\norm{x} \to \infty$
    if $\norm{x} \to \infty$, then
    for each bounded $Y_0 \subset \dom \mG$ there is
    a bounded $X_0 \subset \dom \mF$
    such that $\err^d = \err^d_{X_0}$ on $Y_0$.
  \item
    If $\mG$ is strongly coercive, then for each
    bounded $X_0 \subset \dom \mF$ there is a bounded $Y_0 \subset
    \dom \mG$
    such that
    $\err^p = \err^p_{Y_0}$ on $X_0$.
  \item
    If both $\mF$ and $\mG$ are strongly coercive, then
    for bounded $X_0 \times Y_0 \subset \dom \mF \times \dom \mG$
    there exist bounded $X_0' \times Y_0' \subset \dom \mF \times
    \dom \mG$ such that
    $\gap = \gap_{X_0' \times Y_0'}$ on $X_0 \times Y_0$.
  \end{enumerate}
\end{lemma}
\begin{proof}
  Clearly, it suffices to show the first statement
  as the second follows by analogy and the third by combining the
  first and the second.
    For this purpose, let $Y_0 \subset \dom \mG$ be bounded, i.e.,
  $\norm{y} \leq C$ for each $y \in Y_0 \cup \sett{y^*}$
  and some $C > 0$ independent of
  $y$. Pick $x_0 \in \dom \mF$ and choose $M \geq \norm{x_0}$ such that
  $\mF(x') \geq C \norm{\mK} (\norm{x'} + \norm{x_0}) + \mF(x_0)$
  for each $\norm{x'} > M$ which is possible by strong
  coercivity of $\mF$. Then, for all $y \in Y_0 \cup \sett{y^*}$
  and $\norm{x'} > M$,
  \[
  \scp{-\mK^* y}{x'} - \mF(x') \leq \norm{\mK} \norm{x'} \norm{y} -
  C \norm{\mK} \norm{x'} - C \norm{\mK}\norm{x_0} - \mF(x_0) \leq
  \scp{-\mK^*y}{x_0} -  \mF(x_0),
  \]
  hence the supremum $\mF^*(-\mK^* y) = \sup_{x' \in \dom \mF} \
  \scp{-\mK^*y}{x'} - \mF(x')$ is
  attained on the bounded set
  $X_0 = \sett{\norm{x'} \leq M}$.
  This shows $\err^d = \err^d_{X_0}$ on $Y_0$.
            \end{proof}

As the next step preparing the convergence analysis,
let us denote the symmetric bilinear form associated
with the positive semi-definite
operator $\mM$:
\[
\scp[\mM]{(z_1,\tilde z_1)}{(z_2, \tilde z_2)} =
\Bigscp{\mM
  \begin{bmatrix}
    z_1 \\ \tilde z_1
  \end{bmatrix}}
{\begin{bmatrix}
    z_2 \\ \tilde z_2
  \end{bmatrix}}
= \scp{\Sigma^{-1}(z_1 - \tilde z_1)}{z_2 - \tilde z_2}.
\]
With the choice~\eqref{eq:step_size_operator} and the substitution
$\bar z = z - \tilde z = (\bar x, \bar y)$ for $(z, \tilde z) = (x,y,\tilde x,\tilde y)$,
this becomes, with a slight abuse of notation,
\[
\scp[\mM]{(z_1, \tilde z_1)}{(z_2, \tilde z_2)} =
\scp[\mM]{\bar z_1}{\bar z_2} =
\frac1\sigma \scp{\bar x_1}{\bar x_2} + \frac1\tau \scp{\bar y_1}{\bar y_2}.
\]
Naturally, the induced squared norm will be denoted by
$\norm[\mM]{(z,\tilde z)}^2 =
\norm[\mM]{\bar z}^2 = \frac1\sigma \norm{\bar x}^2 + \frac1\tau \norm{\bar y}^2$.

\begin{lemma}
  \label{lem:douglas-rachford-lagrangian}
  For each $k \in \NN$ and $z = (x,y) \in \dom \mF \times \dom \mG$,
  iteration~\eqref{eq:douglas-rachford} satisfies
  \begin{equation}
    \label{eq:douglas-rachford-lagrangian-estimate}
    \begin{aligned}
      \mL(x^{k+1},y) - \mL(x,y^{k+1})
      & \leq \frac{\norm[\mM]{\bar z^{k} - (\id - \Sigma A) z}^2}2 -
      \frac{\norm[\mM]{\bar z^{k+1} - (\id - \Sigma A) z}^2}2
      - \frac{\norm[\mM]{\bar z^{k} - \bar z^{k+1}}^2}2
      \\
      & = \frac1{\sigma}
      \Bigl(\frac{\norm{\bar x^k - x + \sigma \mK^* y}^2}{2} -
      \frac{\norm{\bar x^{k+1} - x + \sigma \mK^* y}^2}{2} -
      \frac{\norm{\bar x^k - \bar x^{k+1}}^2}{2}
      \Bigr)  \\
      & \quad
      + \frac1{\tau}
      \Bigl(\frac{\norm{\bar y^k - y - \tau \mK x}^2}{2} -
      \frac{\norm{\bar y^{k+1} - y - \tau \mK x}^2}{2} -
      \frac{\norm{\bar y^k - \bar y^{k+1}}^2}{2}
      \Bigr).
    \end{aligned}
  \end{equation}
\end{lemma}

\begin{proof}
  The iteration~\eqref{eq:douglas-rachford} tells that
  $\frac{1}{\sigma}(\bar x^k - x^{k+1}) \in \subgrad \mF(x^{k+1})$
  and $\frac{1}{\tau}(\bar y^k - y^{k+1}) \in \subgrad \mG(y^{k+1})$,
  meaning, in particular, that
  \[
  \mF(x^{k+1}) - \mF(x) \leq
  \frac1\sigma \scp{\bar x^k - x^{k+1}}{x^{k+1} - x},
  \qquad
  \mG(y^{k+1}) - \mG(y) \leq
  \frac1\tau \scp{\bar y^k - y^{k+1}}{y^{k+1} - y}.
  \]
  With the identity
  \[
  \frac1\sigma
  \scp{\bar x^k - x^{k+1}}{x^{k+1} - x} =
  \frac1\sigma \scp{\bar x^k - \bar x^{k+1}}{\bar x^{k+1} - x}
  + \frac1\sigma\scp{\bar x^{k+1} - x^{k+1}}{x^{k+1} + \bar x^{k+1} - \bar x^k - x}
  \]
  and the update rule for $\bar x^{k+1}$ in~\eqref{eq:douglas-rachford}
  in the difference $\bar x^{k+1} - x^{k+1}$,
  we
  get
  \[
  \begin{aligned}
    \frac1\sigma
    &\langle \bar x^k - x^{k+1}, x^{k+1} - x \rangle
    -\scp{\mK x}{y^{k+1}} \\
    &
    = \frac1\sigma \scp{\bar x^k - \bar x^{k+1}}{\bar x^{k+1} - x}
    - \scp{\mK(x^{k+1} + \bar x^{k+1} - \bar x^k - x)}
    {y^{k+1} + \bar y^{k+1} - \bar y^k} + \scp{-\mK x}{y^{k+1}} \\
    &= \frac1\sigma \scp{\bar x^k - \bar x^{k+1}}{\bar x^{k+1} - x}
    - \scp{\mK x}{\bar y^k - \bar y^{k+1}}
    - \scp{\mK(x^{k+1} + \bar x^{k+1} - \bar x^k)}
    {y^{k+1} + \bar y^{k+1} - \bar y^k}.
  \end{aligned}
  \]
  Analogously, we may arrive at
  \[
  \begin{aligned}
    \frac1\tau
    &  \scp{\bar y^k -  y^{k+1}}{y^{k+1} - y}
    +\scp{\mK x^{k+1}}{y} \\
    & = \frac1\tau \scp{\bar y^k - \bar y^{k+1}}{\bar y^{k+1} - y}
    + \scp{\bar x^k - \bar x^{k+1}}{\mK^* y} +
    \scp{\mK(x^{k+1} + \bar x^{k+1} - \bar x^k)}
    {y^{k+1} + \bar y^{k+1} - \bar y^k}.
  \end{aligned}
  \]
  Consequently, the estimate
  \begin{align}
    \label{eq:dr-lagrangian-est0}
    \notag
    \mL(x^{k+1},y)
    &- \mL(x,y^{k+1})  \\
    \notag
    &\leq
      \frac1\sigma \langle \bar x^k - x^{k+1}, x^{k+1} - x \rangle
      -\scp{\mK x}{y^{k+1}} +
      \frac1\tau \langle\bar y^k - y^{k+1}, y^{k+1} - y \rangle
      +\scp{\mK x^{k+1}}{y}  \\
    \notag
    &= \frac1\sigma \scp{\bar x^k - \bar x^{k+1}}{\bar x^{k+1} - x + \sigma \mK^* y}
      + \frac1\tau \scp{\bar y^k - \bar y^{k+1}}{\bar y^{k+1} - y - \tau \mK x} \\
    &= \scp[\mM]{\bar z^k - \bar z^{k+1}}{\bar z^{k+1} -(\id - \Sigma A) z},
  \end{align}
  remembering the definitions~\eqref{eq:saddle-point-data}
  and~\eqref{eq:step_size_operator}.
  Employing the Hilbert space identity $\scp{u}{v} = \frac12
  \norm{u+v}^2 - \frac12 \norm{u}^2 - \frac12 \norm{v}^2$ leads us
  to~\eqref{eq:douglas-rachford-lagrangian-estimate}.
\end{proof}

\begin{lemma}
  \label{lem:douglas-rachford-fixed-points}
  Iteration~\eqref{eq:douglas-rachford} possesses the following
  properties:
  \begin{enumerate}
  \item
    A point $(x^*,y^*,\bar x^*,\bar y^*)$ is a fixed point if and only
    if $(x^*,y^*)$ is a saddle point for~\eqref{eq:saddle-point-prob} and
    $x^* - \sigma \mK^* y^* = \bar x^*$,
    $y^* + \tau \mK x^* = \bar y^*$.
  \item
    \sloppy
    If $\wlim_{i \to \infty} (\bar x^{k_i}, \bar y^{k_i}) = (\bar x^*, \bar y^*)$ and
    $\lim_{i \to \infty} (\bar x^{k_i} - \bar x^{k_i+1}, \bar y^{k_i} - \bar y^{k_i+1})
    = (0,0)$, then
    $(x^{k_i+1}, y^{k_i+1}, \bar x^{k_i+1}, \bar y^{k_i+1})$ converges weakly to
    a fixed point.
  \end{enumerate}
\end{lemma}

\begin{proof}
  Regarding the first point, observe that as~\eqref{eq:douglas-rachford}
  is equivalent to~\eqref{eq:general_proximal_point} with the
  choice~\eqref{eq:step_size_operator} and $\tilde z^k = z^k - \bar z^k$ for
  all $k$, fixed-points $(z^*,\bar z^*)$ are exactly the solutions
  to~\eqref{eq:saddle_point_inclusion_prob} of the form
  $(z^*, \tilde z^*)$, $\tilde z^* = z^* - \bar z^*$.
  With $A$ and $B$ according to~\eqref{eq:saddle-point-data},
  fixed-points are equivalent to solutions
  of~\eqref{eq:saddle_point_inclusion_prob} which obey
  $0 \in Az^* + Bz^*$ and $z^* - \Sigma A z^* = \bar z^*$ which means that
  $z^* = (x^*,y^*)$ is a solution of~\eqref{eq:saddle-point-prob} and
  $\bar z^*= (\bar x^*, \bar y^*)$ satisfies $\bar x^* = x^* - \sigma \mK^* y^*$
  and $\bar y^* = y^* + \tau \mK x^*$.

  \sloppy
  Next, suppose that some subsequence $(\bar x^{k_i}, \bar y^{k_i})$ converges weakly
  to $(\bar x^*, \bar y^*)$ and
  $\lim_{i \to \infty} (\bar x^{k_i+1} - \bar x^{k_i}, \bar y^{k_i+1} - \bar y^{k_i})
  = (0,0)$
  as $i \to \infty$.
  The iteration~\eqref{eq:douglas-rachford} then gives
  \[
  \begin{bmatrix}
    x^{k_i+1} \\ y^{k_i+1}
  \end{bmatrix}
  =
  \begin{bmatrix}
    \id & - \sigma \mK^* \\
    \tau \mK & \id
  \end{bmatrix}^{-1}
  \begin{bmatrix}
    \bar x^{k_i+1} + \sigma \mK^*(\bar y^{k_i+1} - \bar y^{k_i}) \\
    \bar y^{k_i+1} - \tau \mK(\bar x^{k_i+1} - \bar x^{k_i})
  \end{bmatrix},
  \]
  so the sequence $(x^{k_i+1},y^{k_i+1})$ converges weakly to
  $(x^*,y^*)$ with $x^* - \sigma \mK y^* = \bar x^*$ and
  $y^* + \tau \mK x^* = \bar y^*$ by weak sequential continuity of continuous
  linear operators.
  Now, iteration~\eqref{eq:douglas-rachford} can also be written as
  \[
  \left\{
    \begin{aligned}
      \frac1\sigma (\bar x^{k_i} - \bar x^{k_i+1})
      - \mK^*(\bar y^{k_i+1} - \bar y^{k_i})
      &\in \mK^* y^{k_i+1} + \subgrad \mF(x^{k_i+1}) \\
      \frac1\tau (\bar y^{k_i} - \bar y^{k_i+1})
      + \mK(\bar x^{k_i+1} - \bar x^{k_i})
      &\in  - \mK x ^{k_i+1} + \subgrad \mG(y^{k_i+1})
    \end{aligned}
  \right.
  \]
  where the left-hand side converges strongly to $(0,0)$. As
  $(x,y) \mapsto \bigl(\mK^* y + \subgrad \mF(x), -\mK x
  + \subgrad \mG(y) \bigr)$ is maximally
  monotone, it is in particular weak-strong closed. Consequently,
  $(0,0) \in \bigl(\mK^* y^* + \subgrad \mF(x^*),
  - \mK x^* + \subgrad \mG(y^*) \bigr)$,
  so $(x^*,y^*)$ is a saddle-point of~\eqref{eq:saddle-point-prob}.
  In particular, $x^* - \sigma \mK y^* = \bar x^*$ and
  $y^* + \tau \mK x^* = \bar y^*$, hence, the weak limit of
  $\seq{(x^{k_i+1},y^{k_i+1},\bar x^{k_i+1}, \bar y^{k_i+1})}$
  is a fixed-point of the
  iteration.
\end{proof}

\begin{proposition}
  \label{prop:dr-weak-convergence}
  \sloppy
  If~\eqref{eq:saddle-point-prob} possesses a solution, then
  iteration~\eqref{eq:douglas-rachford_}   converges weakly to a
  $(x^*,y^*, \bar x^*, \bar y^*)$ for which $(x^*,y^*)$ is a solution
  of~\eqref{eq:saddle-point-prob} and $(\bar x^*,\bar y^*)
  = (x^* - \sigma \mK^* y^*, y^* + \tau \mK x^*)$.
\end{proposition}

\begin{proof}
  Let $(x',y')$ be a solution of~\eqref{eq:saddle-point-prob}. Then,
  $0 \leq \mL(x^{k+1},y') - \mL(x',y^{k+1})$ for each $k$, so
  by~\eqref{eq:douglas-rachford-lagrangian-estimate} and recursion
  we have for $k_0 \leq k$ that
  \begin{align}
    \notag
    \frac{\norm{\bar x^{k} - x' + \sigma \mK^* y'}^2}{2\sigma}
    + \frac{\norm{\bar y^{k} - y' - \tau \mK x'}^2}{2\tau}
    &+ \biggl[ \sum_{k' = k_0}^{k-1}
    \frac{\norm{\bar x^{k'} - \bar x^{k'+1}}^2}{2\sigma}
    + \frac{\norm{\bar y^{k'} - \bar y^{k'+1}}^2}{2\tau}
    \biggr] \\
    \label{eq:dr_boundedness}
    &\leq
    \frac{\norm{\bar x^{k_0} - x' + \sigma \mK^* y'}^2}{2\sigma}
    + \frac{\norm{\bar y^{k_0} - y' - \tau \mK x'}^2}{2\tau}.
  \end{align}
  With $x'' = x' - \sigma \mK^* y'$ and $y'' = y' + \tau \mK x'$
  and $d_k = \frac1{2\sigma}
  \norm{\bar x^k - x''}^2 + \frac1{2\tau} \norm{\bar y^k - y''}^2$,
  this implies that $\seq{d_k}$
  is a non-increasing sequence with limit $d^*$. Furthermore,
  $\bar x^k - \bar x^{k+1} \to 0$ as well as $\bar y^k - \bar y^{k+1} \to 0$
  as $k \to \infty$.

  Now,~\eqref{eq:dr_boundedness} says that $\seq{(\bar x^k,\bar y^k)}$ is
  bounded whenever~\eqref{eq:saddle-point-prob} has a solution,
  hence there exists a weak accumulation point
  $(\bar x^*, \bar y^*)$. Setting $(x^*,y^*) \in Z$ as the unique
  solution to the equation
  \begin{equation}
    \label{eq:weak-limit-saddle-point}
    \left\{
      \begin{aligned}
        x^* - \sigma \mK^* y^* &= \bar x^*, \\
        y^* + \tau \mK x^* &= \bar y^*,
      \end{aligned}
    \right.
  \end{equation}
  gives, by virtue of Lemma~\ref{lem:douglas-rachford-fixed-points},
  that $(x^*, y^*, \bar x^*, \bar y^*)$ is a fixed-point of the iteration
  with $(x^*,y^*)$ being a solution of~\eqref{eq:saddle-point-prob} and
  an accumulation point of $\seq{(x^k, y^k)}$.
  Suppose that $(\bar x^{**}, \bar y^{**})$ is another
  weak accumulation point of the same sequence and
  choose the corresponding
  $(x^{**}, y^{**})$ according to~\eqref{eq:weak-limit-saddle-point}.
  Then,
  \begin{align*}
    \frac1\sigma \scp{\bar x^k}{\bar x^{**} - \bar x^{*}} +
    &\frac1\tau \scp{\bar y^k}{\bar y^{**} - \bar y^{*}} =
      \frac{\norm{\bar x^k - x^* + \sigma \mK^* y^*}^2}{2\sigma}
      +   \frac{\norm{\bar y^k - y^* - \tau \mK x^*}^2}{2\tau}
    \\
    &
      -\frac{\norm{\bar x^k - x^{**} + \sigma \mK^* y^{**}}^2}{2\sigma}
      -   \frac{\norm{\bar y^k - y^{**} - \tau \mK x^{**}}^2}{2\tau}
      - \frac{\norm{\bar x^*}^2}{2\sigma} - \frac{\norm{\bar y^*}^2}{2\tau} \\
    &
      + \frac{\norm{\bar x^{**}}^2}{2\sigma}
      + \frac{\norm{\bar y^{**}}^2}{2\tau}.
  \end{align*}
  As both $(x^*,y^*)$ and $(x^{**}, y^{**})$ are solutions
  to~\eqref{eq:saddle-point-prob}, the right-hand side converges
  as a consequence of~\eqref{eq:dr_boundedness}. Plugging in the subsequences
  weakly converging to $(\bar x^*, \bar y^*)$ and $(\bar x^{**}, \bar y^{**})$,
  respectively,
  on the left-hand side, we see that both limits must coincide:
  \[
  \frac1\sigma \scp{\bar x^*}{\bar x^* - \bar x^{**}} +
  \frac1\tau \scp{\bar y^*}{\bar y^* - \bar y^{**}}
  =
  \frac1\sigma \scp{\bar x^{**}}{\bar x^* - \bar x^{**}} +
  \frac1\tau \scp{\bar y^{**}}{\bar y^* - \bar y^{**}},
  \]
  hence $\bar x^* = \bar x^{**}$ and $\bar y^* = \bar y^{**}$, implying
  that the whole sequence $\seq{(\bar x^k, \bar y^k)}$ converges weakly
  to $(\bar x^*, \bar y^*)$. Inspecting the
  iteration~\eqref{eq:douglas-rachford} we see that
  $\wlim_{k \to \infty}(x^{k} - \sigma \mK^* y^{k}, y^k + \tau \mK x^k)
  = (\bar x^*, \bar y^*)$ and as solving for $(x^k,y^k)$ constitutes a
  continuous linear operator, also $\wlim_{k \to \infty} (x^k, y^k)
  = (x^*,y^*)$, what remained to show.
\end{proof}

In particular, restricted primal-dual gaps vanish for the iteration
and can thus be used as a stopping criterion.

\begin{corollary}
  In the situation of Proposition~\ref{prop:dr-weak-convergence},
  for $X_0 \times Y_0 \subset \dom \mF \times \dom \mG$ bounded
  and containing a saddle-point, the associated
  restricted primal-dual gap obeys
  \[
  \gap_{X_0 \times Y_0}(x^k, y^k) \geq 0 \quad
  \text{for each} \quad k,
  \qquad
  \lim_{k \to \infty} \gap_{X_0 \times Y_0}(x^k, y^k) = 0.
  \]
\end{corollary}

\begin{proof}
      The convergence
  follows from applying the Cauchy--Schwarz inequality
  to~\eqref{eq:dr-lagrangian-est0}, taking the supremum and
  noting that $\norm[\mM]{\bar z^k - \bar z^{k+1}} \to 0$ as
  $k \to \infty$ as shown in Proposition~\ref{prop:dr-weak-convergence}.
\end{proof}
This convergence comes, however, without a rate, making it necessary
to switch to ergodic sequences.
Here, the usual $\mO(1/k)$ convergence speed follows
almost immediately.

\begin{theorem}
  \label{thm:douglas-rachford-ergodic-rate}
  Let $X_0 \subset \dom \mF$, $Y_0 \subset \dom \mG$ be bounded and such that
  $X_0 \times Y_0$ contains a saddle-point, i.e., a
  solution of~\eqref{eq:saddle-point-prob}.
  Then, in addition to $\wlim_{k \to \infty} (x^k,y^k) = (x^*,y^*)$
  the ergodic sequences
  \begin{equation}
    \label{eq:douglas-rachford-erg-seq}
    x_{\erg}^k = \frac1k \sum_{k' = 1}^{k} x^{k'}, \qquad
    y_{\erg}^k = \frac1k \sum_{k' = 1}^{k} y^{k'}
  \end{equation}
  converge weakly to $x^*$ and $y^*$, respectively,
  and the restricted primal-dual gap obeys
  \begin{equation}
    \label{eq:dr-O1k-convergence}
    \gap_{X_0 \times Y_0}(x_{\erg}^k, y_{\erg}^k) \leq \frac 1k
    \sup_{(x,y) \in X_0 \times Y_0} \
    \Bigl[
    \frac{\norm{\bar x^0 - x + \sigma \mK^* y}^2}{2\sigma} +
    \frac{\norm{\bar y^0 - y - \tau \mK x}^2}{2\tau}
    \Bigr]
    = \mO(1/k).
  \end{equation}
\end{theorem}

\begin{proof}
  First of all, the assumptions state the existence of a saddle-point,
  so by Proposition~\ref{prop:dr-weak-convergence}, $\seq{(x^k,y^k)}$ converges
  weakly to a solution $(x^*,y^*)$
  of~\eqref{eq:saddle-point-prob}.
  To obtain the weak convergence of $\seq{x_{\erg}^k}$, test with
  an $x$ and utilize the Stolz--Ces\`aro theorem to obtain
  \[
  \lim_{k \to \infty}  \scp{x_{\erg}^k}{x} = \lim_{k \to \infty}
  \frac{\sum_{k'=1}^k \scp{x^{k'}}{x}}{\sum_{k' = 1}^k 1} =
  \lim_{k \to \infty} \scp{x^k}{x} = \scp{x^*}{x}.
  \]
  The property $\wlim_{k \to \infty} y_{\erg}^k = y^*$ can be proven analogously.

  Finally, in order to show the estimate on
  $\gap_{X_0 \times Y_0}(x^k_{\erg}, y^k_{\erg})$, observe that
  for each $(x,y) \in X_0 \times Y_0$, the function
  $(x',y') \mapsto \mL(x',y) - \mL(x,y')$ is convex. Thus,
  using~\eqref{eq:douglas-rachford-lagrangian-estimate} gives
  \begin{align*}
    \mL(x^k_{\erg}, y) - \mL(x, y^k_{\erg})
    &\leq \frac1k \sum_{k'=0}^{k-1} \mL(x^{k'+1},y) - \mL(x,y^{k'+1}) \\
    &\leq \frac1k \sum_{k' = 0}^{k-1}
      \frac1{\sigma}
      \Bigl(\frac{\norm{\bar x^{k'} - x + \sigma \mK^* y}^2}{2} -
      \frac{\norm{\bar x^{k'+1} - x + \sigma \mK^* y}^2}{2} \Bigr) \\
    & \qquad\qquad
      + \frac1\tau
      \Bigl(\frac{\norm{\bar y^{k'} - y - \tau \mK x}^2}{2} -
      \frac{\norm{\bar y^{k'+1} - y - \tau \mK x}^2}{2}
      \Bigr) \\
    &\leq \frac1k
      \Bigl( \frac{\norm{\bar x^0 - x + \sigma \mK^* y}^2}{2\sigma} +
      \frac{\norm{\bar y^0 - y - \tau \mK x}^2}{2\tau} \Bigr).
  \end{align*}
  Taking the supremum over all $(x,y) \in X_0 \times Y_0$
  gives $\gap_{X_0 \times Y_0}(x^k_{\erg}, y^k_{\erg})$ on the left-hand
  side and its estimate from above in~\eqref{eq:dr-O1k-convergence}.
        \end{proof}

\subsection{Preconditioning}
\label{subsec:dr-precon}

The iteration~\eqref{eq:douglas-rachford_} can now be preconditioned as
follows. We amend the dual variable $y$ by a preconditioning
variable $x_p \in X$ and the linear operator $\mK$ by
$\mH: X \to X$ linear, continuous and consider
the saddle-point problem
\begin{equation}
  \label{eq:precon-saddle-point}
  \min_{x \in \dom \mF} \max_{\substack{y \in \dom \mG,\\ x_p = 0}} \
  \scp{\mK x}{y} + \scp{\mH x}{x_p} + \mF(x) - \mG(y) - \mI_{\sett{0}}(x_p)
\end{equation}
whose solutions have exactly the form $(x^*, y^*, 0)$ with
$(x^*,y^*)$ being a saddle-point of~\eqref{eq:saddle-point-prob}.
Plugging this into the Douglas--Rachford iteration
for~\eqref{eq:precon-saddle-point} yields by
construction that $x_p^{k+1} = 0$ as well as
$\bar x^{k+1}_p = \tau \mH d^{k+1}$ for each $k$, so no new quantities
have to be introduced into the iteration. Introducing the notation
$T = \id + \sigma\tau \mK^*\mK$ and $M = \id + \sigma\tau(\mK^*\mK
+ \mH^*\mH)$, we see that the step $d^{k+1} = T^{-1} b^{k+1}$
in~\eqref{eq:douglas-rachford_} is just replaced by
\[
d^{k+1} = d^k + M^{-1}(b^{k+1} - Td^k),
\]
which can be interpreted as one iteration step of an operator
splitting method for the solution of $T d^{k+1} = b^{k+1}$ with
respect to the splitting $T = M - (M - T)$. The full iteration is
shown in Table~\ref{tab:precon-douglas-rachford} (the
term \emph{feasible preconditioner} is defined after the next paragraph).

\begin{table}
  \centering
  \begin{tabular}{p{0.15\linewidth}p{0.75\linewidth}}
    \toprule

    \multicolumn{2}{l}{\textbf{pDR}\ \ \ \ \textbf{Objective:}
    \hfill    Solve  \ \ $\min_{x\in \dom \mF} \max_{y \in  \dom \mG}
    \ \langle \mK x,y\rangle + \mF(x) - \mG(y)$
    \hfill\mbox{} }
    \\
    \midrule

    Initialization:
 &
   $(\bar x^0, \bar y^0, d^0) \in X \times Y \times X$
   initial guess, $\sigma > 0, \tau > 0$ step sizes, \\
 &
   $M$ feasible preconditioner for $T = \id + \sigma\tau\mK^*\mK$
    \\[\medskipamount]
    Iteration:
 &
   \raisebox{1em}[0mm][6.9\baselineskip]{\begin{minipage}[t]{1.0\linewidth}
       {       \renewcommand{\theHequation}{pdr}       \begin{equation}\label{eq:precon-douglas-rachford_}\tag{pDR}
         \left\{
           \begin{aligned}
             x^{k+1} &= (\id + \sigma \subgrad \mF)^{-1}(\bar x^k) \\
             y^{k+1} &= (\id + \tau \subgrad \mG)^{-1}(\bar y^k) \\
             b^{k+1} &= (2x^{k+1} - \bar x^k) - \sigma \mK^*(2y^{k+1} - \bar y^k) \\
             d^{k+1} &= d^k + M^{-1}(b^{k+1} - Td^k) \\
             \bar x^{k+1} &= \bar x^k - x^{k+1} + d^{k+1} \\
             \bar y^{k+1} &= y^{k+1} + \tau \mK d^{k+1}
           \end{aligned}
         \right.
       \end{equation}}
     \end{minipage}}
    \\
    \vspace*{-\smallskipamount}
    Output:
 &
   \vspace*{-\smallskipamount}
   $\seq{(x^k,y^k)}$  primal-dual sequence  \\
    \bottomrule
  \end{tabular}
  \caption{The preconditioned
    Douglas--Rachford iteration for the solution of
    convex-concave saddle-point problems of
    type~\eqref{eq:saddle-point-prob}.}
  \label{tab:precon-douglas-rachford}
\end{table}

Given $T = \id + \sigma\tau \mK^*\mK$, one usually aims at choosing
$M$ such that it corresponds to well-known solution techniques such as
the Gauss--Seidel or successive over-relaxation procedure in the
finite-dimensional case, for instance.
For this purpose, one has to ensure that a $\mH: X \to X$ linear and
continuous can be found for given $M$. The following definition from \cite{BS1}
gives a necessary and sufficient criterion.

\begin{definition}
  Let $M, T: X \to X$ linear, self-adjoint, continuous and positive
  semi-definite. Then, $M$ is called a \emph{feasible preconditioner}
  for $T$ if $M$ is boundedly invertible and $M - T$ is positive
  semi-definite.
\end{definition}

Indeed, given $\mH: X \to X$ linear and continuous, it is clear from the
definition of $M$ and $T$ that $M$ is a feasible preconditioner.
The other way around, if $M$ is a feasible preconditioner for $T$,
then one can take the square root of the operator $M - T$ and
consequently, there is linear, continuous and self-adjoint
$\mH: X \to X$ such that $\mH^2 = \frac1{\sigma\tau} (M-T)$ which implies
$M = \id + \sigma\tau(\mK^*\mK + \mH^*\mH)$. Summarized, the
preconditioned Douglas--Rachford iterations based
on~\eqref{eq:precon-saddle-point} exactly corresponds
to~\eqref{eq:precon-douglas-rachford_} with
$M$ being a feasible preconditioner. Its convergence then follows from
Proposition~\ref{prop:dr-weak-convergence} and
Theorem~\ref{thm:douglas-rachford-ergodic-rate}, with weak convergence
and asymptotic rate being maintained.

\begin{theorem}
  \label{thm:precon-douglas-rachford-ergodic-rate}
  Let $X_0 \subset \dom \mF$, $Y_0 \subset \dom \mG$ be bounded and such that
  $X_0 \times Y_0$ contains a saddle-point, i.e., a
  solution of~\eqref{eq:saddle-point-prob}.
  Further, let $M$ be a feasible preconditioner for $T = \id
  + \sigma\tau \mK^*\mK$.

  Then, for the iteration generated by~\eqref{eq:precon-douglas-rachford_},
  it holds that
  $\wlim_{k \to \infty} (x^k,y^k) = (x^*,y^*)$ with
  $(x^*, y^*)$ being a saddle-point of~\eqref{eq:saddle-point-prob}.
  Additionally,
  the ergodic sequence $\seq{(x_{\erg}^k, y_{\erg}^k)}$ given
  by~\eqref{eq:douglas-rachford-erg-seq} converges
  weakly to $(x^*,y^*)$
  and the restricted primal-dual gap obeys
  \begin{equation}
    \label{eq:pdr-O1k-convergence}
    \gap_{X_0 \times Y_0}(x_{\erg}^k, y_{\erg}^k) \leq \frac 1k
    \sup_{(x,y) \in X_0 \times Y_0} \
    \Bigl[
    \frac{\norm{\bar x^0 - x + \sigma \mK^* y}^2}{2\sigma}
    + \frac{\norm[M-T]{d^0 - x}^2}{2\sigma}
    + \frac{\norm{\bar y^0 - y - \tau \mK x}^2}{2\tau}
    \Bigr]
  \end{equation}
  where $\norm[M-T]{x}^2 = \scp{(M-T)x}{x}$.
\end{theorem}

\begin{proof}
  The above considerations show that one can find a $\mH: X \to X$ such
  that the iteration~\eqref{eq:precon-douglas-rachford_}
  corresponds to the Douglas--Rachford iteration for
  the solution of~\eqref{eq:precon-saddle-point}.
  Proposition~\ref{prop:dr-weak-convergence} then already yields weak
  convergence and Theorem~\ref{thm:douglas-rachford-ergodic-rate}
  gives the estimate~\eqref{eq:dr-O1k-convergence} on the
  restricted primal-dual gap associated with~\eqref{eq:precon-saddle-point}
  (note that we apply it for $X_0 \times (Y_0 \times \sett{0})$).
  Now, observe that as $x_p = 0$ must hold,
        the latter coincides with the restricted primal-dual gap associated
  with~\eqref{eq:saddle-point-prob}. Plugging in
  $\bar x^0_p = \tau \mH d^0$ and $M-T = \sigma\tau \mH^*\mH$
  into~\eqref{eq:dr-O1k-convergence} finally
  yields~\eqref{eq:pdr-O1k-convergence}.
\end{proof}

The condition that $M$ should be feasible for $T$ allows for a great
flexibility in preconditioning. Basically, classical splitting methods
such as symmetric Gauss--Seidel or symmetric successive over-relaxation
are feasible as a $n$-fold application thereof. We summarize the
most important statements regarding feasibility of classical operator
splittings in Table~\ref{tab:choice-preconditioner} and some general
results in the following propositions and refer
to \cite{BS1,BS2} for further details.

\begin{table}\centering \begin{tabular}{l@{\ \ }c@{\ \ }c@{\ \ }c@{\ \ }c@{\ \ }c}   \toprule
  Preconditioner & $T$ & $\lambda \id$ &$(\lambda+1)D$ &$M_{\SGS}$& $M_{\SSOR}$\\
  \midrule
  Conditions & --- & $\lambda \geq \|T\|$  &  $\lambda \geq
   \lambda_{\max}(T - D)$  & --- & $\omega \in (0,2)$\\[\smallskipamount]
Iteration type &
\begin{minipage}{0.12\linewidth}
  \centering
  Douglas--\\Rachford
\end{minipage}
& Richardson & Damped Jacobi &
\begin{minipage}{0.15\linewidth}
  \centering
  Symmetric Gauss--Seidel
\end{minipage}
  &
\begin{minipage}{0.12\linewidth}
  \centering
Symmetric\\
SOR
\end{minipage}
\\[\smallskipamount]
\bottomrule \end{tabular}

\medskip

$D = \diag(T)$, $T = D - E - E^*$, $E$ lower triangular,
$M_{\SGS} = (D - E) D^{-1} (D - E^*)$,

$M_{\SSOR} = (\tfrac1\omega D - E)
(\tfrac{2-\omega}{\omega}D)^{-1} (\tfrac 1\omega D - E^*)$

\caption{Summary of common preconditioners and possible conditions
  for feasibility.
} \label{tab:choice-preconditioner}
\end{table}

\begin{proposition}
  \label{prop:feasible-precon}
  Let $T: X \to X$ linear, continuous, symmetric
  and positive definite be given. Then, the following
  update schemes correspond to feasible preconditioners.
  \begin{enumerate}
  \item
    For $M_0: X \to X$ linear, continuous such that $M_0 - \tfrac12 T$
    is positive definite,
    \begin{equation}
      \label{eq:precon-symmetrization}
      \left\{
        \begin{aligned}
          d^{k+1/2} &= d^k + M_0^{-1}(b^{k+1} - Td^k), \\
          d^{k+1} &= d^{k+1/2} + M_0^{-*}(b^{k+1} - Td^{k+1/2}).
        \end{aligned}
      \right.
    \end{equation}
  \item
    For $M: X \to X$ feasible for $T$ and $n \geq 1$,
    \begin{equation}
      \label{eq:precon-nfold}
      \left\{
        \begin{aligned}
          d^{k+(i+1)/n} &= d^{k+i/n} + M^{-1}(b^{k+1} - Td^{k+i/n}), \\
          i &= 0,\ldots, n-1.
        \end{aligned}
      \right.
    \end{equation}
  \item
    For $T = T_1 + T_2$, with $T_1,T_2: X \to X$ linear, continuous,
    symmetric, $T_1$ positive definite, $T_2$ positive semi-definite and
    $M: X \to X$ feasible for $T_1$, $M - T_2$ boundedly invertible,
    \begin{equation}
      \label{eq:precon-additive}
      \left\{
        \begin{aligned}
          d^{k+1/2} &= d^k + M^{-1}\bigl((b^{k+1} - T_2d^k) - T_1d^k \bigr), \\
          d^{k+1} &= d^k + M^{-1}\bigl((b^{k+1} - T_2d^{k+1/2}) - T_1d^k \bigr).
        \end{aligned}
      \right.
    \end{equation}
  \end{enumerate}
\end{proposition}
The update scheme~\eqref{eq:precon-symmetrization} is useful if one
has a non-symmetric ``feasible'' preconditioner $M_0$ for $\tfrac12 T$ as
then, the concatenation with the adjoint preconditioner will be symmetric
and feasible. Likewise,~\eqref{eq:precon-nfold} corresponds to the $n$-fold
application of a feasible preconditioner which is again feasible.
Finally,~\eqref{eq:precon-additive} is useful if $T$ can be split into
$T_1 + T_2$ for which $T_1$ can be easily preconditioned by $M$. Then,
one obtains a feasible preconditioner for $T$ only using forward evaluation of
$T_2$.

\section{Acceleration strategies}
\label{sec:acceleration}

We now turn to the main results of the paper.
As already mentioned in the introduction, in the case that the
functionals $\mF$ and $\mG$ possess strong
convexity properties, the iteration~\eqref{eq:douglas-rachford_} can
be modified such that the asymptotic ergodic rates for the restricted
primal-dual gap can further be improved.

\begin{definition}
  A convex functional $\mF: X \to \RR_\infty$ admits the
  \emph{modulus of strong convexity} $\gamma_1 \geq 0$ if for each
  $x, \bar x \in \dom \mF$ and $\lambda \in [0,1]$ we have
  \[
  \mF\bigl(\lambda x + (1-\lambda) \bar x\bigr)
  \leq
  \lambda \mF(x) + (1 - \lambda) \mF(\bar x)
  - \gamma_1 \lambda (1 - \lambda) \frac{\norm{x - \bar x}^2}{2}.
  \]
  Likewise, a concave functional $\mF$
  admits the \emph{modulus of strong concavity} $\gamma_1$ if
  $-\mF$ admits the modulus of strong convexity $\gamma_1$.
\end{definition}
Note that for $x \in \dom \subgrad \mF$, $\xi \in \subgrad \mF(x)$ and
$\bar x \in \dom \mF$ it follows that
\[
\mF(x) + \scp{\xi}{\bar x - x} + \frac{\gamma_1}{2} \norm{\bar x - x}^2
\leq \mF(\bar x).
\]
The basic idea for acceleration is now, in case the modulus of
strong convexity is positive, to use the additional quadratic terms
to obtain estimates of the
type
\begin{align*}
  \rho \bigl( \mL(x^{k+1}, y) - \mL(x,y^{k+1}) \bigr)
  &+ \res(\hat x^k, \hat
    y^k, x^{k+1}, y^{k+1}, x_\rho, y_\rho) \\
  & \leq \frac{\norm{\hat x^k - x_\rho + \sigma \mK^*y_\rho}^2}{2\sigma} +
    \frac{\norm{\hat y^k - y_\rho - \tau \mK x_\rho}^2}{2\tau} \\
  & \quad - \frac{1}{\vartheta} \Bigl[
    \frac{\norm{\hat x^{k+1} - x_\rho + \sigma' \mK^*y_\rho}^2}{2\sigma'} +
    \frac{\norm{\hat y^{k+1} - y_\rho - \tau' \mK x_\rho}^2}{2\tau'}
    \Bigr]
\end{align*}
for a modified iteration based on~\eqref{eq:douglas-rachford_}.
Here, the convex combination $(x_\rho, y_\rho) = (1- \rho)(x',y')
+ \rho(x,y)$ of $(x,y) \in \dom \mF \times \dom \mG$ and a saddle-point
$(x',y')$ is taken for $\rho \in [0,1]$, $(\hat x^k,\hat y^k)$ denotes
an additional primal-dual pair that is updated during the iteration,
$\res$ a non-negative residual functional
that will become apparent later,
$0 < \vartheta < 1$ an acceleration factor and $\sigma' > 0$, $\tau'
> 0$ step-sizes that are possibly different from $\sigma$ and $\tau$.
In order to achieve this, we consider a pair
$(\hat x^k, \hat y^k) \in X \times Y$ and set
\begin{equation}
  \label{eq:acc_update_x_y}
  \left\{
    \begin{aligned}
      x^{k+1} &= (\id + \sigma \subgrad \mF)^{-1}(\hat x^k), \\
      y^{k+1} &= (\id + \tau \subgrad \mG)^{-1}(\hat y^k), \\
    \end{aligned}
\right.
\end{equation}
as well as, with $\vartheta_p, \vartheta_d > 0$,
\begin{equation}\label{update:x:y:k+1}
  \left\{
    \begin{aligned}
      \tilde x^{k+1} &= \sigma \mK^*[y^{k+1} + \vartheta_p(\bar y^{k+1} - \hat y^k)], \\
      \tilde y^{k+1} &= - \tau \mK[x^{k+1} + \vartheta_d(\bar x^{k+1} - \hat x^k)],
    \end{aligned}
\right.
\end{equation}
where, throughout this section, we always denote
$\tilde x^{k+1} = x^{k+1} - \bar x^{k+1}$ and
$\tilde y^{k+1} = y^{k+1} - \bar y^{k+1}$. Observe that for
$\vartheta_p = \vartheta_d = 1$ and $(\hat x^{k+1}, \hat y^{k+1})
= (\bar x^{k+1}, \bar y^{k+1})$, the iteration~\eqref{eq:douglas-rachford} is
recovered which corresponds to~\eqref{eq:douglas-rachford_}.
In the following, however, we will choose these parameters differently depending
on the type of strong convexity of the $\mF$ and $\mG$.
To analyze the step~\eqref{eq:acc_update_x_y} and~\eqref{update:x:y:k+1},
let us first provide an estimate analogous
to~\eqref{eq:douglas-rachford-lagrangian-estimate} in
Lemma~\ref{lem:douglas-rachford-lagrangian}.

\begin{lemma}
  \label{lem:acc_lagrangian_est}
  Let $\rho \in [0,1]$, denote by
  $\gamma_1, \gamma_2 \geq 0$ the moduli of strong convexity
  for $\mF$ and $\mG$, respectively, and let $(x',y') \in \dom \mF
  \times \dom \mG$ be a saddle-point of $\mL$.
  For $(\hat x^k, \hat y^k) \in X \times Y$
  as well as $(x,y) \in \dom \mF \times \dom \mG$,
  the equations~\eqref{eq:acc_update_x_y} and~\eqref{update:x:y:k+1}
  imply
  \begin{align}
    \label{eq:acc_lagrangian_est}
    \notag
    \rho\bigl(\mL(x^{k+1},y) &- \mL(x,y^{k+1}) \bigr) \\
    \notag
    &
      \quad + \frac1{2\sigma}
      \norm{\bar x^{k+1} - x_\rho + \sigma \mK^*y_\rho}^2
      + \frac{\gamma_1}{1+\rho} \norm{x^{k+1} - x_\rho}^2
      + \frac1{2\sigma} \norm{\hat x^k - \bar x^{k+1}}^2 \\
    \notag
    & \quad + (1-\vartheta_p)
      \scp{\hat y^k - \bar y^{k+1}}
      {\mK[x^{k+1} + \bar x^{k+1} - \hat x^k - x_\rho]} \\
    \notag
    & \quad + \frac1{2\tau}
      \norm{\bar y^{k+1} - y_\rho - \tau \mK x_\rho}^2
      + \frac{\gamma_2}{1+\rho} \norm{y^{k+1} - y_\rho}^2
      + \frac1{2\tau} \norm{\hat y^k - \bar y^{k+1}}^2 \\
    \notag
    & \quad - (1 - \vartheta_d)
      \scp{y^{k+1} + \bar y^{k+1} - \hat y^k - y_\rho}
      {\mK[\hat x^k - \bar x^{k+1}]} \\
    & \leq \frac1{2\sigma}
      \norm{\hat x^k - x_\rho + \sigma \mK^*y_\rho}^2 + \frac1{2\tau}
      \norm{\hat y^k - y_\rho - \tau \mK x_\rho}^2
  \end{align}
  where $x_\rho = (1-\rho) x' + \rho x$ and
  $y_\rho = (1-\rho) y' + \rho y$.
\end{lemma}
\begin{proof}
  First, optimality of $(x',y')$ according
  to~\eqref{eq:saddle-point-optimality}, strong convexity
  and the subgradient inequality give
  \[
  \frac{\gamma_1}{2} \norm{x^{k+1} - x'}^2
  + \frac{\gamma_2}{2} \norm{y^{k+1} - y'}^2
  \leq \mL(x^{k+1}, y') - \mL(x', y^{k+1}).
  \]
  Combined with strong concavity of
  $x \mapsto -\mL(x, y^{k+1})$ and $y \mapsto \mL(x^{k+1},y)$
  with moduli $\gamma_1$ and
  $\gamma_2$, respectively,
  this yields for the convex combinations
  $x_\rho$ and $y_\rho$ that
  \begin{multline*}
    (1-\rho) \frac{\gamma_1}{2}
    \bigl[ \norm{x^{k+1} - x'}^2
    + \rho \norm{x - x'}^2 \bigr]
    + (1-\rho) \frac{\gamma_2}{2}
    \bigl[ \norm{y^{k+1} - y'}^2
    + \rho \norm{y - y'}^2 \bigr]
    \\
    + \rho\bigl( \mL(x^{k+1}, y) - \mL(x, y^{k+1}) \bigr)
    \leq \mL(x^{k+1}, y_\rho) - \mL(x_\rho, y^{k+1}).
  \end{multline*}
  Employing convexity once more, we arrive at
  \[
  \frac{1- \rho}{1 + \rho} \frac{\gamma_1}{2} \norm{x^{k+1} - x_\rho}^2
  \leq
  (1-\rho) \frac{\gamma_1}{2}
  \bigl[ \norm{x^{k+1} - x'}^2
  + \rho \norm{x - x'}^2 \bigr]
  \]
  and the analogous statement for the dual variables. Putting things
  together yields
  \begin{multline}
    \label{eq:strong_convexity_lagrange_est}
    \frac{1- \rho}{1 + \rho}
    \Bigl[ \frac{\gamma_1}{2} \norm{x^{k+1} - x_\rho}^2
    + \frac{\gamma_2}{2} \norm{y^{k+1} - y_\rho}^2 \Bigr]
    + \rho \bigl( \mL(x^{k+1},y) - \mL(x, y^{k+1}) \bigr) \\
    \leq \mL(x^{k+1},y_\rho) - \mL(x_\rho, y^{k+1}).
  \end{multline}
  The next steps aim at estimating the right-hand side
  of~\eqref{eq:strong_convexity_lagrange_est}.
  Using again the strong convexity of $\mF$ and $\mG$ and
  the subgradient inequality
  for~\eqref{eq:acc_update_x_y}, we obtain
  \begin{align*}
    \frac{\gamma_1}2 \norm{x^{k+1} - x_\rho}^2
    &+ \mF(x^{k+1}) - \mF(x_\rho) \\
    &\leq \frac1\sigma
      \scp{\hat x^k - x^{k+1}}{x^{k+1} - x_\rho}  \\
    & = \frac1\sigma \scp{\hat x^k - x^{k+1} + \tilde x^{k+1}}
      {x^{k+1} - \tilde x^{k+1} - x_\rho} + \frac1\sigma \scp{\hat x^k
      + \tilde x^{k+1} - 2 x^{k+1} + x_\rho}{\tilde x^{k+1}}
    \\
    &= \frac1\sigma \scp{\hat x^k - \bar x^{k+1}}{\bar x^{k+1} - x_\rho}
      - \frac1\sigma \scp{x^{k+1} + \bar x^{k+1} - \hat x^k - x_\rho}
      {\tilde x^{k+1}}
  \end{align*}
  as well as
  \begin{align*}
    \frac{\gamma_2}2\norm{y^{k+1} - y_\rho}^2 &+ \mG(y^{k+1}) - \mG(y_\rho)
    \\
    &
    \leq \frac1\tau \scp{\hat y^k - \bar y^{k+1}}{\bar y^{k+1} - y_\rho}
           - \frac1\tau \scp{y^{k+1} + \bar y^{k+1} - \hat y^k - y_\rho}
           {\tilde y^{k+1}}.
  \end{align*}
  Collecting the terms involving $\tilde x^{k+1}$ and $\tilde y^{k+1}$,
  adding the primal-dual
  coupling terms $\scp{\mK x^{k+1}}{y_\rho}  - \scp{\mK x_\rho}{y^{k+1}}$
  as well as employing~\eqref{update:x:y:k+1} gives
  \begin{align*}
    - \frac1\sigma
    \scp{x^{k+1} + \bar x^{k+1} - \hat x^k - x_\rho}{\tilde x^{k+1}}
    &- \frac1\tau \scp{y^{k+1} + \bar y^{k+1}
      - \hat y^k - y_\rho}{\tilde y^{k+1}} \\
    + \scp{\mK x^{k+1}}{y_\rho}  - \scp{\mK x_\rho}{y^{k+1}}
    &= \scp{y^{k+1} + \bar y^{k+1} - \hat y^k - y_\rho}{\mK[x^{k+1}
      + \bar x^{k+1} - \hat x^k]} \\
    & \quad
      - \scp{y^{k+1} + \bar y^{k+1} - \hat y^k}{\mK[x^{k+1}
      + \bar x^{k+1} - \hat x^k - x_\rho]} \\
    &\quad
      + (1-\vartheta_p)
      \scp{\bar y^{k+1} - \hat y^k}
      {\mK[x^{k+1} + \bar x^{k+1} - \hat x^k - x_\rho]} \\
    &\quad
      -(1 - \vartheta_d)
      \scp{y^{k+1} + \bar y^{k+1} - \hat y^k - y_\rho}
      {\mK[\bar x^{k+1} - \hat x^k]}
    \\
    &\quad + \scp{\mK x^{k+1}}{y_\rho}  - \scp{\mK x_\rho}{y^{k+1}} \\
    &= (1-\vartheta_p)
      \scp{\bar y^{k+1} - \hat y^k}
      {\mK[x^{k+1} + \bar x^{k+1} - \hat x^k - x_\rho]} \\
    & \quad
      -(1 - \vartheta_d)
      \scp{y^{k+1} + \bar y^{k+1} - \hat y^k - y_\rho}
      {\mK[\bar x^{k+1} - \hat x^k]} \\
    &\quad + \scp{y_\rho}{\mK[\hat x^k - \bar x^{k+1}]}
      - \scp{\hat y^k - \bar y^{k+1}}{\mK x_\rho}.
  \end{align*}
  It then follows that
  \begin{align}
    \label{eq:acc1_lagrangian_est0}
    \notag
    \mF(x^{k+1})
    &- \mF(x_\rho) + \mG(y^{k+1}) - \mG(y_\rho)
      + \scp{\mK x^{k+1}}{y_\rho} - \scp{\mK x_\rho}{y^{k+1}}
    \\
    \notag
    &\leq
      \frac1\sigma \scp{\hat x^k - \bar x^{k+1}}{\bar x^{k+1} - x_\rho + \sigma \mK^*y_\rho}
      + \frac1\tau \scp{\hat y^k - \bar y^{k+1}}{\bar y^{k+1} - y_\rho - \tau \mK x_\rho}
    \\
    \notag
    & \quad
      - \frac{\gamma_1}2 \norm{x^{k+1} - x_\rho}^2
      + (1-\vartheta_p)
      \scp{\bar y^{k+1} - \hat y^k}
      {\mK[x^{k+1} + \bar x^{k+1} - \hat x^k - x_\rho]} \\
    &\quad
      - \frac{\gamma_2}2 \norm{y^{k+1} - y_\rho}^2
      -(1 - \vartheta_d)
      \scp{y^{k+1} + \bar y^{k+1} - \hat y^k - y_\rho}
      {\mK[\bar x^{k+1} - \hat x^k]}.
  \end{align}
  As before,
  employing the Hilbert-space identity $\scp{u}{v} = \tfrac12 \norm{u+v}^2
  - \tfrac12 \norm{u}^2 - \tfrac12 \norm{v}^2$,
  plugging the result into~\eqref{eq:strong_convexity_lagrange_est}
  and rearranging finally
  yields~\eqref{eq:acc_lagrangian_est}.
\end{proof}

\subsection{\texorpdfstring{Strong convexity of $\mF$}{Strong convexity of F}}

Now, assume that $\gamma_1 > 0$ while for $\gamma_2$, no restrictions are
made, such that we set $\gamma_2 = 0$. In this case, we can rewrite the
quadratic terms involving $\bar x^{k+1}$ and $\bar y^{k+1}$
in~\eqref{eq:acc_lagrangian_est} as follows.

\begin{lemma}
  \label{lem:acc_sqr_norm_repr}
  In the situation of Lemma~\ref{lem:acc_lagrangian_est} and for
  $\gamma > 0$ we have
  \begin{align}
    \label{eq:acc_quadratic_terms}
    \notag
    \frac1{2\sigma} \| \bar x^{k+1} - x_\rho +
    &\sigma \mK^*y_\rho \|^2
      + \frac\gamma2 \norm{x^{k+1} - x_\rho}^2
      + \frac1{2\tau} \norm{\bar y^{k+1} - y_\rho - \tau \mK x_\rho}^2 \\
    \notag
    & =
      \frac1\vartheta \Bigl[
      \frac1{2\sigma'} \norm{\hat x^{k+1} - x_\rho + \sigma' \mK^*y_\rho}^2
      + \frac1{2\tau'} \norm{\hat y^{k+1} - y_\rho - \tau' \mK x_\rho}^2
      \Bigr] \\
    \notag
    &  \quad
      + \frac{1 - \vartheta}{\vartheta}
      \Bigl[
      \frac{1}{\tau} \scp{y^{k+1} - y_\rho}{\tau \mK x_\rho + \tilde y^{k+1}}
      - \frac1\sigma \scp{x^{k+1} - x_\rho}{\sigma \mK^*y_\rho - \tilde x^{k+1}}
      \Bigr] \\
    & \quad
      - \frac{\sigma \gamma}{2\tau}  \norm{\tau \mK x_\rho + \tilde y^{k+1}}^2
  \end{align}
  where
  \begin{equation}
    \label{eq:acc_final_step}
    \vartheta = \frac{1}{\sqrt{1 + \sigma\gamma}}, \qquad
    \left\{
      \begin{aligned}
        \sigma' &= \vartheta \sigma, \\
        \tau' &= \vartheta^{-1} \tau,
      \end{aligned}
    \right.
    \qquad
    \left\{
      \begin{aligned}
        \hat x^{k+1} &= x^{k+1} - \vartheta \tilde x^{k+1} , \\
        \hat y^{k+1} &= y^{k+1} - \vartheta^{-1} \tilde y^{k+1}.
      \end{aligned}
    \right.
  \end{equation}
\end{lemma}

\begin{proof}
  This is a result of straightforward computations which we
  present for completeness and reader's convenience.
  Computing
  \begin{align*}
    \frac1{2\sigma} \norm{\bar x^{k+1} - x_\rho + \sigma \mK^*y_\rho}^2
    &+ \frac\gamma2
      \norm{x^{k+1} - x_\rho}^2 \\
    &=
      \frac{1 + \sigma \gamma}{2\sigma} \norm{x^{k+1} - x_\rho}^2
      + \frac{1 + \sigma\gamma}\sigma
      \scp{x^{k+1} - x_\rho}{\frac{1}{\sqrt{1 + \sigma\gamma}}
      [\sigma \mK^*y_\rho - \tilde x^{k+1}]} \\
    & \quad
      + \frac{1 + \sigma\gamma}{2\sigma} \norm{\frac{1}{\sqrt{1 + \sigma
      \gamma}} [\sigma \mK^*y_\rho - \tilde x^{k+1}]}^2  \\
    & \quad
      - \frac1{\sigma} (\sqrt{1 + \sigma\gamma} - 1)
      \scp{x^{k+1} - x_\rho}{\sigma \mK^*y_\rho - \tilde x^{k+1}}
    \\
    &= \sqrt{1 + \sigma\gamma} \frac{\sqrt{1 + \sigma\gamma}}{2\sigma}
      \norm{x^{k+1} - x_\rho + \frac1{\sqrt{1 + \sigma\gamma}}
      [\sigma \mK^*y_\rho - \tilde x^{k+1}]}^2 \\
    & \quad
      - \frac1{\sigma} (\sqrt{1 + \sigma\gamma} - 1)
      \scp{x^{k+1} - x_\rho}{\sigma \mK^*y_\rho - \tilde x^{k+1}}
  \end{align*}
  as well as
  \begin{align*}
    \frac1{2\tau} \| \bar y^{k+1} - y_\rho &- \tau \mK x_\rho \|^2
    = \frac{1}{2\tau} \norm{y^{k+1} - y_\rho}^2
      - \frac{1}{\tau} \scp{y^{k+1} - y_\rho}
      {\sqrt{1+\sigma \gamma}[\tau \mK x_\rho + \tilde y^{k+1}]} \\
    & \quad
      + \frac{1}{2\tau}\norm{\sqrt{1 + \sigma\gamma}[\tau \mK x_\rho
      + \tilde y^{k+1}]}^2 + \frac{1}{\tau} (\sqrt{1+\sigma \gamma} - 1)
      \scp{y^{k+1} - y_\rho}{\tau \mK x_\rho + \tilde y^{k+1}}
    \\
    & \quad - \frac{\sigma \gamma}{2\tau}  \norm{\tau \mK x_\rho + \tilde y^{k+1}}^2 \\
    & = \frac{\sqrt{1 + \sigma\gamma}}{2\tau \sqrt{1 + \sigma\gamma}}
      \norm{y^{k+1}
      - y_\rho - \sqrt{1+\sigma\gamma} [\tau \mK x_\rho + \tilde y^{k+1}]}^2 \\
    & \quad
      + \frac{1}{\tau} (\sqrt{1+\sigma \gamma} - 1)
      \scp{y^{k+1} - y_\rho}{\tau \mK x_\rho + \tilde y^{k+1}}
      - \frac{\sigma \gamma}{2\tau}  \norm{\tau \mK x_\rho + \tilde y^{k+1}}^2
  \end{align*}
  and plugging in the definitions~\eqref{eq:acc_final_step}
  already gives~\eqref{eq:acc_quadratic_terms}.
\end{proof}

As the next step, we aim at combining the ``error terms'' on the
right-hand side of~\eqref{eq:acc_quadratic_terms} and left-hand side
of~\eqref{eq:acc_lagrangian_est} such that they become non-negative.
One important point here is to choose $\vartheta_p$ and $\vartheta_d$
such that possibly negative terms involving $y^{k+1} - y_\rho$ cancel out.

\begin{lemma}
  \label{lem:acc_error_term_est}
  In the situation of Lemma~\ref{lem:acc_sqr_norm_repr},
  with $\vartheta_p = \vartheta^{-1}$ and $\vartheta_d = \vartheta$,
  we have
  \begin{align}
    \label{eq:acc_error_terms}
    \notag
    (1-\vartheta_p)
    &\scp{\hat y^k - \bar y^{k+1}}
      {\mK[x^{k+1} + \bar x^{k+1} - \hat x^k - x_\rho]}
                      - (1 - \vartheta_d)
      \scp{y^{k+1} + \bar y^{k+1} - \hat y^k - y_\rho}
          {\mK[\hat x^k - \bar x^{k+1}]} \\
    \notag
    &
      + \frac{1 - \vartheta}{\vartheta}
      \Bigl[
      \frac{1}{\tau} \scp{y^{k+1} - y_\rho}{\tau \mK x_\rho + \tilde y^{k+1}}
      - \frac1\sigma \scp{x^{k+1} - x_\rho}{\sigma \mK^*y_\rho - \tilde x^{k+1}}
      \Bigr]
    \\
    &
      \notag
      - \frac{\sigma \gamma}{2\tau}  \norm{\tau \mK x_\rho + \tilde y^{k+1}}^2 \\
    &\geq
      - \frac{c}{2\sigma} \norm{\hat x^k - \bar x^{k+1}}^2
      - \frac{c}{2\tau} \norm{\hat y^k - \bar y^{k+1}}^2
      - \frac{c(1+\sigma \gamma)\sigma \gamma\tau \|\mK\|^2}
      {2(c^2 - \sigma^2 \gamma \tau \|\mK\|^2)} \|x^{k+1} - x_\rho\|^2
  \end{align}
  whenever $\sigma^2\gamma\tau\norm{\mK}^2 < c^2 < 1$ holds
  for some $c > 0$.
\end{lemma}

\begin{proof}
  We start with rewriting the scalar-product terms from
  the right-hand side of~\eqref{eq:acc_quadratic_terms} by plugging
  in~\eqref{update:x:y:k+1} as follows
      \begin{align*}
    \frac1\tau \langle y^{k+1} - y_\rho,
    &\tau \mK x_\rho + \tilde y^{k+1}\rangle
      - \frac1\sigma \scp{x^{k+1} - x_\rho}{\sigma \mK^*y_\rho - \tilde x^{k+1}} \\
    &= \scp{y^{k+1} - y_\rho}{\mK[x_\rho - x^{k+1}
      + \bar x^{k+1} - \hat x^k]} - \scp{y^{k+1} - y_\rho + \bar y^{k+1}
      - \hat y^k}{\mK[x_\rho - x^{k+1}]} \\
    & \quad - (1 + \vartheta_d)\scp{y^{k+1} - y_\rho}{\mK[\bar x^{k+1} - \hat x^k]}
      + (1 - \vartheta_p) \scp{\bar y^{k+1} - \hat y^k}{\mK[x_\rho - x^{k+1}]}
    \\
    &= \scp{\hat y^k - \bar y^{k+1}}{\mK[x_\rho - x^{k+1} + \bar x^{k+1} - \hat x^k]}
      - \scp{y^{k+1} - y_\rho + \bar y^{k+1} - \hat y^k}{\mK[\hat x^k - \bar x^{k+1}]}
    \\
    & \quad + (1 + \vartheta_d)\scp{y^{k+1} - y_\rho}{\mK[\hat x^k - \bar x^{k+1}]}
      + (1 - \vartheta_p) \scp{\bar y^{k+1} - \hat y^k}{\mK[x_\rho - x^{k+1}]}
    \\
    &= \scp{\hat y^k - \bar y^{k+1}}{\mK[\vartheta_p(x_\rho - x^{k+1}) + \bar x^{k+1} - \hat x^k]}
    \\
    & \quad
      - \scp{-\vartheta_d(y^{k+1} - y_\rho) + \bar y^{k+1} - \hat y^k}{\mK[\hat x^k - \bar x^{k+1}]}.
  \end{align*}
  Incorporating the scalar-product terms from the left-hand side
  of~\eqref{eq:acc_lagrangian_est}, the
  choice of $\vartheta_p$ and $\vartheta_d$ yields
      \begin{align*}
  (1 - \vartheta_p)
  &\scp{\hat y^k - \bar y^{k+1}}
  {\mK[x^{k+1} + \bar x^{k+1} - \hat x^k - x_\rho]} -(1 - \vartheta_d)
  \scp{y^{k+1} + \bar y^{k+1} - \hat y^k - y_\rho}
  {\mK[\hat x^k - \bar x^{k+1}]} \\
  & \quad - \frac{1 - \vartheta}{\vartheta}
    \scp{-\vartheta_d(y^{k+1} - y_\rho) + \bar y^{k+1} - \hat y^k}
    {\mK[\hat x^k - \bar x^{k+1}]} \\
  & \quad + \frac{1 - \vartheta}{\vartheta}
    \scp{\hat y^k - \bar y^{k+1}}{\mK[\vartheta_p(x_\rho - x^{k+1})
    + \bar x^{k+1} - \hat x^k]} \\
  &=
    \scp{\hat y^k - \bar y^{k+1}}{\mK[(1 - \vartheta_p^2)(x^{k+1}- x_\rho)
    + (\vartheta_d - \vartheta_p)(\bar x^{k+1} - \hat x^k)]} \\
  &= (1 - \vartheta_p^2) \scp{\hat y^k - \bar y^{k+1}}
    {\mK[(x^{k+1} - x_\rho) + \vartheta_d(\bar x^{k+1} - \hat x^{k})]}.
\end{align*}
Using Young's inequality, this allows to estimate, as $c > 0$,
\begin{align*}
(1 - \vartheta_p^2) &\scp{\hat y^k - \bar y^{k+1}}
    {\mK[(x^{k+1} - x_\rho) + \vartheta_d(\bar x^{k+1} - \hat x^{k})]} \\
&\geq
- \frac{c}{2\tau} \norm{\hat y^k - \bar y^{k+1}}^2 -
\frac{(1-\vartheta_p^2)^2\tau}{2c} \norm{\mK[(x^{k+1} - x_\rho)
  + \vartheta_d(\bar x^{k+1} - \hat x^k)]}^2.
\end{align*}
Taking the missing
quadratic term into account,
using that $c < 1$ and
once again the definitions of $\vartheta$, $\vartheta_p$ and
$\vartheta_d$, gives, for $\varepsilon > 0$, the estimate
\begin{align*}
  \frac{(1-\vartheta_p^2)^2\tau}{2c} &\norm{\mK[(x^{k+1} - x_\rho) + \vartheta_d(\bar x^{k+1} - \hat x^k)]}^2
    + \frac{\sigma\gamma}{2\tau} \norm{\tau \mK x_\rho + \tilde y^{k+1}}^2  \\
  &=  \Bigl(\frac{(\vartheta_p^2-1)^2}{2c} + \frac{\vartheta_p^2-1}{2}\Bigr)\tau\norm{\mK[(x^{k+1} - x_\rho) + \vartheta_d(\bar x^{k+1} - \hat x^k)]}^2
  \\
  & \leq \frac{\vartheta_p^4-\vartheta_p^2}{2c}\tau\norm{\mK[(x^{k+1} - x_\rho) + \vartheta_d(\bar x^{k+1} - \hat x^k)]}^2\\
  & \leq (1 + \varepsilon)\frac{\vartheta_p^2 - 1}{2c} \tau\|\mK\|^2\|\bar x^{k+1} - \hat x^k\|^2
  + \Bigl(1 + \frac{1}{\varepsilon}\Bigr)\frac{\vartheta_p^4- \vartheta_p^2}{2c} \tau \|\mK\|^2 \|x^{k+1}-x_\rho\|^2.
\end{align*}
We would like to set
\[
 (1 + \varepsilon)\frac{\vartheta_p^2 - 1}{2c} \tau\|\mK\|^2 = \frac{c}{2 \sigma},
\]
which is equivalent to
\[
\varepsilon
= \frac{c^2 - \sigma^2\gamma\tau \norm{\mK}^2}{\sigma^2\gamma\tau \norm{\mK}^2}
\]
and leading to $\varepsilon > 0$ since
$\sigma^2\gamma\tau \norm{\mK}^2 < c^2$ by assumption. Plugged into
the last term in the above estimate, we get
\begin{align*}
  \Bigl(1 + \frac{1}{\varepsilon}\Bigr)\frac{\vartheta_p^4- \vartheta_p^2}{2c} \tau \|\mK\|^2 \|x^{k+1}-x_\rho\|^2
  &= \frac{c^2}{c^2 - \sigma^2\gamma\tau\|\mK\|^2} \frac{\vartheta_p^2-1}{2c}\vartheta_p^2 \tau \|\mK\|^2\|x^{k+1}-x_\rho \|^2 \\
  & = \frac{c(1+\sigma \gamma)\sigma \gamma \tau \|\mK\|^2}{2(c^2 - \sigma^2 \gamma \tau \|\mK\|^2)} \|x^{k+1} - x_\rho \|^2.
\end{align*}
Putting all estimates together then yields~\eqref{eq:acc_error_terms}.
\end{proof}
In view of combining~\eqref{eq:acc_lagrangian_est},
\eqref{eq:acc_quadratic_terms} and~\eqref{eq:acc_error_terms}, the factor
in front of $\norm{x^{k+1} - x_\rho}^2$ in~\eqref{eq:acc_error_terms} should not
be too large. Choosing $\gamma > 0$ small enough, one can indeed control
this quantity. Doing so, one arrives at the following result.

\begin{proposition}
  \label{prop:acc1_lagrangian_est}
  With the definitions in~\eqref{eq:acc_final_step} and $\gamma$ chosen such
  that
  \begin{equation}
    \label{eq:acc_gamma_restr}
    0 < \gamma < \frac{2\gamma_1}{1 + \rho
      + (1 + \rho + 2\sigma \gamma_1)\sigma\tau \norm{\mK}^2}
  \end{equation}
  there is a $0 < c < 1$ and a $c' > 0$ such that
  the following estimate holds:
  \begin{align}
    \label{eq:acc1_lagrangian_est}
    \notag
    \rho \bigl( \mL(x^{k+1},y) - \mL(x,y^{k+1}) \bigr)
    &
      + \frac1\vartheta \Bigl[
      \frac1{2\sigma'} \norm{\hat x^{k+1} - x_\rho + \sigma' \mK^*y_\rho}^2
      + \frac1{2\tau'} \norm{\hat y^{k+1} - y_\rho - \tau' \mK x_\rho}^2
      \Bigr]
    \\
    \notag
    & \quad + (1-c) \Bigl[
      \frac1{2\sigma} \norm{\hat x^k - \bar x^{k+1}}^2 +
      \frac1{2\tau} \norm{\hat y^k - \bar y^{k+1}}^2 \Bigr]
      + \frac{c'}{2} \norm{x^{k+1} - x_\rho}^2
    \\
    & \leq \frac1{2\sigma}
      \norm{\hat x^k - x_\rho + \sigma \mK^*y_\rho}^2 + \frac1{2\tau}
      \norm{\hat y^k - y_\rho - \tau \mK x_\rho}^2.
  \end{align}
\end{proposition}
\begin{proof}
  Note that~\eqref{eq:acc_gamma_restr} implies that
  $\sigma^2\gamma\tau \norm{\mK}^2 < 1$
  and
  \[
  \frac{(1+\sigma \gamma) \sigma \gamma \tau \|\mK\|^2}{2(1 - \sigma^2 \gamma \tau \|\mK\|^2)} < \frac{\gamma_1}{1 + \rho} - \frac{\gamma}{2},
  \]
  so by continuity of $c \mapsto c(1+\sigma\gamma)\sigma\gamma \tau
  \norm{\mK}^2/\bigl(2(c^2 - \sigma^2\gamma\tau \norm{\mK}^2) \bigr)$
  in $c =1$, one can find a $c < 1$ with
  $c^2 > \sigma^2\gamma\tau \norm{\mK}^2$ such that
  \[
  c' = \frac{\gamma_1}{1 + \rho} - \frac{\gamma}{2} - \frac{c(1+\sigma \gamma)\sigma \gamma\tau \|\mK\|^2}{2(c^2 - \sigma^2 \gamma \tau \|\mK\|^2)} > 0.
  \]
  The prerequisites for Lemma~\ref{lem:acc_error_term_est} are satisfied,
  hence one can combine~\eqref{eq:acc_lagrangian_est},
  \eqref{eq:acc_quadratic_terms} and~\eqref{eq:acc_error_terms} in order
  to get~\eqref{eq:acc1_lagrangian_est}.
\end{proof}

\begin{remark}
  From the proof it is also immediate
  that if~\eqref{eq:acc_gamma_restr} holds for a
  $\sigma_0 > 0$ instead of $\sigma$, the
  estimate~\eqref{eq:acc1_lagrangian_est} will still hold for
  all $0 < \sigma \leq \sigma_0$ and $\tau > 0$ such that
  $\sigma\tau = \sigma_0\tau_0$
  with $c$ and $c'$ independent from
  $\sigma$.
\end{remark}

The estimate~\eqref{eq:acc1_lagrangian_est} suggests to adapt the
step-sizes $(\sigma,\tau) \to (\sigma',\tau')$ in each iteration step.
This yields the \emph{accelerated Douglas--Rachford iteration} which obeys
the following recursion:
\begin{equation}
  \label{eq:accelerated-douglas-rachford:notable}
  \left\{
    \begin{aligned}
      \vartheta_k &= \frac1{\sqrt{1 + \sigma_k\gamma}}, \\
      x^{k+1} &= (\id + \sigma_k\subgrad \mF)^{-1}(\hat x^k), \\
      y^{k+1} &= (\id + \tau_k\subgrad \mG)^{-1}(\hat y^k), \\
      \tilde x^{k+1} &= \sigma_k \mK^*[y^{k+1} + \vartheta_k^{-1}(\bar y^{k+1} - \hat y^k)], \\
      \tilde y^{k+1} &= - \tau_k \mK[x^{k+1} + \vartheta_k(\bar x^{k+1} - \hat x^k)], \\
      \hat x^{k+1} &= x^{k+1} - \vartheta_k \tilde x^{k+1} , \\
      \hat y^{k+1} &= y^{k+1} - \vartheta_k^{-1} \tilde y^{k+1}, \\
      \sigma_{k+1} &= \vartheta_k\sigma_k,
      \quad \tau_{k+1} = \vartheta_k^{-1}\tau_k.
    \end{aligned}
  \right.
\end{equation}
As before, the iteration can be written down explicitly and in a
simplified manner, as, for instance, the product satisfies
$\sigma_k\tau_k = \sigma_0\tau_0$ and some auxiliary variables
can be omitted. Such a version is summarized in
Table~\ref{tab:accelerated-douglas-rachford} along with conditions we
will need in the following convergence analysis.
One sees in particular that~\eqref{eq:accelerated-douglas-rachford_}
requires only negligibly more computational and implementational
effort than~\eqref{eq:douglas-rachford_}.

\begin{table}
  \centering
  \begin{tabular}{p{0.15\linewidth}p{0.75\linewidth}}
    \toprule

    \multicolumn{2}{l}{\textbf{aDR}\ \ \ \ \textbf{Objective:}
    \hfill    Solve  \ \ $\min_{x\in \dom \mF} \max_{y \in  \dom \mG}
    \ \langle \mK x,y\rangle + \mF(x) - \mG(y)$
    \hfill\mbox{} }
    \\
    \midrule

    Prerequisites:
 &
   $\mF$ is strongly convex with modulus $\gamma_1 > 0$
    \\[\smallskipamount]
    Initialization:
 &
   $(\hat x^0, \hat y^0) \in X \times Y$
   initial guess, $\sigma_0 > 0, \tau_0 > 0$ initial step sizes, \\
 &
   $0 < \gamma < \frac{2\gamma_1}{1 + \sigma_0\tau_0 \norm{\mK}^2}$
   acceleration factor, $\vartheta_0 = \frac1{\sqrt{1 + \sigma_0\gamma}}$
    \\[\medskipamount]
    Iteration:
 &
   \raisebox{1.6em}[0mm][9\baselineskip]{\begin{minipage}[t]{1.0\linewidth}
       {       \renewcommand{\theHequation}{adr}       \begin{equation}\label{eq:accelerated-douglas-rachford_}\tag{aDR}
         \left\{
           \begin{aligned}
             x^{k+1} &= (\id + \sigma_k\subgrad \mF)^{-1}(\hat x^k) \\
             y^{k+1} &= (\id + \tau_k\subgrad \mG)^{-1}(\hat y^k) \\
             b^{k+1} &= ((1 + \vartheta_k) x^{k+1} - \vartheta_k\hat x^k)
             - \sigma_k\mK^*((1 + \vartheta_k)y^{k+1} - \hat y^k) \\
             d^{k+1} &= (\id + \sigma_0\tau_0 \mK^*\mK)^{-1}b^{k+1} \\
             \hat x^{k+1} &= \vartheta_k(\hat x^k - x^{k+1}) + d^{k+1} \\
             \hat y^{k+1} &= y^{k+1} + \vartheta_k^{-1} \tau_k \mK d^{k+1} \\
             \sigma_{k+1} &= \vartheta_k\sigma_k,
             \quad \tau_{k+1} = \vartheta_k^{-1}\tau_k,
             \quad \vartheta_{k+1} = \tfrac1{\sqrt{1 + \sigma_{k+1}\gamma}} \\
           \end{aligned}
         \right.
       \end{equation}}
     \end{minipage}}
    \\
    \vspace*{-\smallskipamount}
    Output:
 &
   \vspace*{-\smallskipamount}
   $\seq{(x^k,y^k)}$  primal-dual sequence  \\
    \bottomrule
  \end{tabular}
  \caption{The accelerated Douglas--Rachford iteration for the solution of
    convex-concave saddle-point problems of
    type~\eqref{eq:saddle-point-prob}.}
  \label{tab:accelerated-douglas-rachford}
\end{table}

\begin{remark}
  Of course, the iteration~\eqref{eq:accelerated-douglas-rachford_} can
  easily be adapted to involve $(\id +\sigma_0\tau_0\mK\mK^*)^{-1}$
  instead of $(\id +\sigma_0\tau_0\mK^*\mK)^{-1}$, in case the former
  can be computed more efficiently, for instance. A straightforward
  application of Woodbury's formula, however, would introduce an
  additional evaluation of $\mK$ and $\mK^*$, respectively, and hence,
  potentially higher computational effort.
  As in this situation, the issue cannot be resolved by simply interchanging
  primal and dual variable, we explicitly state the necessary
  modifications in order to maintain one evaluation of $\mK$ and $\mK^*$
  in each iteration step:
  \[
  \left\{
    \begin{aligned}
      b^{k+1} &= \bigl((1 + \vartheta_k^{-1})y^{k+1} - \vartheta_k^{-1} \hat y^k \bigr)
      + \tau_k \mK \bigl( (1 + \vartheta_k^{-1})x^{k+1} - \hat x^k \bigr),
      \quad
      d^{k+1} = (\id + \sigma_0\tau_0 \mK\mK^*)^{-1} b^{k+1} \\
      \hat x^{k+1} &= x^{k+1} - \vartheta_k \sigma_k \mK^* d^{k+1},
      \quad
      \hat y^{k+1} = \vartheta_k^{-1}(\hat y^k - y^{k+1}) + d^{k+1}.
    \end{aligned}
  \right.
  \]
\end{remark}

Now, as $\mF$ is strongly convex, we might hope for improved convergence
properties for $\seq{x^k}$ compared to~\eqref{eq:douglas-rachford_}. This
is indeed the case.
For the convergence analysis, we introduce the following quantities.
\begin{equation}
  \label{eq:acc1_erg_weights}
  \lambda_k = \prod_{k'=0}^{k-1} \frac{1}{\vartheta_{k'}},
  \qquad
  \nu_k = \Bigl[\sum_{k'=0}^{k-1} \lambda_{k'}\Bigr]^{-1}
\end{equation}
\begin{lemma}
  \label{lem:acc1_num_seq}
  The sequences $\seq{\lambda_k}$ and $\seq{\nu_k}$ obey, for all
  $k \geq 1$,
  \begin{equation}
    \label{eq:acc1_num_seq_est}
    1 + \frac{k\sigma_0\gamma}{\sqrt{1 + \sigma_0\gamma} + 1}
    \leq \lambda_k \leq 1 + \frac{k\sigma_0\gamma}{2},
    \qquad
    \nu_k \leq \Bigl[ k + \frac{(k-1)k
      \sigma_0\gamma}{2(\sqrt{1 + \sigma_0\gamma} + 1)} \Bigr]^{-1}
    = \mO(1/k^2).
  \end{equation}
\end{lemma}

\begin{proof}
  Observing that $\lambda_{k+1} = \frac1{\vartheta_k} \lambda_k$ as well
  as $\sigma_k = \frac{\sigma_0}{\lambda_k}$ we find that $\lambda_k \geq 1$
  and
  \[
  \lambda_{k+1} - \lambda_k = \lambda_k
  \bigl( \sqrt{1 + \sigma_k\gamma} - 1 \bigr)
  = \frac{\lambda_k}{\sqrt{1 + \sigma_k\gamma} + 1}{\sigma_k\gamma}
  = \frac{\sigma_0 \gamma}{\sqrt{1 + \sigma_k \gamma} + 1}.
  \]
  for all $k \geq 0$.
  The sequence $\seq{\sigma_k}$ is monotonically decreasing and positive,
  so estimating the denominator accordingly gives the bounds on $\lambda_k$.
  Summing up yields the estimate on $\nu_k$, in particular, for
  $k \geq 2$,
  we have $\nu_k \leq \frac{4(\sqrt{1 + \sigma_0 \gamma} + 1)}{\sigma_0\gamma}
  \frac1{k^2}$, i.e.,
  $\nu_k = \mO(1/k^2)$.
\end{proof}

\begin{proposition}
  \label{prop:acc1_weak_convergence}
  If~\eqref{eq:saddle-point-prob} possesses a solution, then
  the iteration~\eqref{eq:accelerated-douglas-rachford_} converges
  to a saddle-point $(x^*,y^*)$ of~\eqref{eq:saddle-point-prob}
  in the following sense:
  \[
  \lim_{k \to \infty} x^k = x^* \quad
  \text{with} \quad \norm{x^k - x^*}^2 = \mO(1/k^2),
  \qquad
  \wlim_{k \to \infty} y^k = y^*.
  \]
  In particular, each saddle-point $(x',y')$ of~\eqref{eq:saddle-point-prob}
  satisfies $x' = x^*$.
\end{proposition}

\begin{proof}
  Observe that since~\eqref{eq:accelerated-douglas-rachford_} requires
  $\gamma < 2\gamma_1/(1 + \sigma_0\tau_0 \norm{\mK}^2)$,
  and since $\sigma_k = \sigma_0/\lambda_k$ as well as
  Lemma~\ref{lem:acc1_num_seq} implies
  $\lim_{k \to \infty} \sigma_k = 0$, there is a
  $k_0 \geq 0$ such that~\eqref{eq:acc_gamma_restr} is satisfied
  for $\rho = 0$ and
  $\sigma_k$ for all $k \geq k_0$. Letting $(x',y')$ be a saddle-point,
                      applying~\eqref{eq:acc1_lagrangian_est}
  recursively gives
  \begin{align}
    \label{eq:acc1_convergence_est}
    \notag
    \lambda_k
    \Bigl[
    \frac1{2\sigma_k} & \norm{\hat x^{k} - x' + \sigma_{k} \mK^*y'}^2
    + \frac1{2\tau_k} \norm{\hat y^{k} - y' - \tau_{k} \mK x'}^2
      \Bigr]
    \\
    \notag
    & \quad + \sum_{k'=k_0}^{k-1}
      \lambda_{k'} \Bigl[
      \frac{1-c}{2\sigma_{k'}} \norm{\hat x^{k'} - \bar x^{k'+1}}^2 +
      \frac{1-c}{2\tau_{k'}} \norm{\hat y^{k'} - \bar y^{k'+1}}^2
      + \frac{c'}{2} \norm{x^{k'+1} - x'}^2\Bigr]
    \\
    & \leq \lambda_{k_0} \Bigl[ \frac1{2\sigma_{k_0}}
      \norm{\hat x^{k_0} - x' + \sigma_{k_0} \mK^*y'}^2 + \frac1{2\tau_{k_0}}
      \norm{\hat y^{k_0} - y' - \tau_{k_0} \mK x'}^2 \Bigr].
  \end{align}
  This implies, on the one hand, the convergence of
  the series
  \[
  \sum_{k=0}^{\infty}
  \lambda_{k} \Bigl[
  \frac{1}{2\sigma_{k}} \norm{\hat x^{k} - \bar x^{k+1}}^2 +
  \frac{1}{2\tau_{k}} \norm{\hat y^{k} - \bar y^{k+1}}^2
  + \frac{1}{2} \norm{x^{k+1} - x'}^2\Bigr] < \infty.
  \]
  As $\lambda_k = \sigma_0/\sigma_k = \tau_k/\tau_0$,
  we have in particular that $\lim_{k \to \infty}
  \frac1{\sigma_k}(\hat x^k - \bar x^{k+1}) = 0$ and
  $\lim_{k \to \infty} (\hat y^k - \bar y^{k+1}) = 0$.
  Furthermore, defining
  \[
  d^k(x',y') =
  \lambda_k
  \Bigl[
  \frac1{2\sigma_k} \norm{\hat x^{k} - x' + \sigma_{k} \mK^*y'}^2
  + \frac1{2\tau_k} \norm{\hat y^{k} - y' - \tau_{k} \mK x'}^2
  \Bigr],
  \]
  the limit $\lim_{k \to \infty} d^k(x',y') = d^*(x',y')$
  has to exist: On the one hand, the sequence
  is bounded from below and admits a finite limes
  inferior. On the other hand, traversing
  with $k_0$ a subsequence that converges to the
  limes inferior, the estimate~\eqref{eq:acc1_convergence_est}
  yields that limes superior and limes inferior have to coincide, hence,
  the sequence is convergent.
  Denoting by
  \[
  \xi^k = \hat x^{k} - x' + \sigma_{k} \mK^*y', \qquad
  \zeta^k = \hat y^{k} - y' - \tau_{k} \mK x',
  \]
  we further conclude that $\seq{\frac1{\sigma_k} \xi^k}$
  as well as $\seq{\zeta^k}$
  are bounded.
  Plugging in the iteration~\eqref{eq:accelerated-douglas-rachford:notable}
  gives the identities
  \begin{equation}
    \label{eq:acc1_inversion_step}
    \left\{
      \begin{aligned}
        x^k - x' - \sigma_k \mK^*(y^k - y') &= \xi^k
        + \vartheta_{k-1}^{-1} \sigma_k \mK^*(\bar y^k - \hat y^{k-1}), \\
        y^k - y' + \tau_k \mK(x^k - x') &= \zeta^k
        - \vartheta_{k-1} \tau_k \mK(\bar x^k - \hat x^{k-1}), \\
      \end{aligned}
    \right.
  \end{equation}
                which can be solved with respect to $x^k - x'$ and $y^k - y'$ yielding
  \[
  \left\{
    \begin{aligned}
      x^k - x' &=
      (\id + \sigma_0\tau_0\mK^*\mK)^{-1}\bigl[\xi^k
      + \sigma_k \mK^*\bigl(\zeta^k  +
      \vartheta_{k-1}^{-1}(\bar y^k - \hat y^{k-1}) \bigr)
      - \sigma_0\tau_0\vartheta_{k-1}
      \mK^*\mK(\bar x^k - \hat x^{k-1})
      \bigr], \\
      y^k - y' &=
      (\id + \sigma_0\tau_0\mK\mK^*)^{-1}\bigl[
      \zeta^k
      - \tau_k \mK\bigl(\xi^k +
      \vartheta_{k-1}(\bar x^k - \hat x^{k-1}) \bigr)
      - \sigma_0\tau_0 \vartheta_{k-1}^{-1}
      \mK\mK^*(\bar y^k - \hat y^{k-1}) \bigr].
    \end{aligned}
  \right.
  \]
  Regarding the norm of $x^k - x'$, the
  boundedness and convergence properties of the involved terms allow
  to conclude that
  \[
  \norm{x^k - x'}^2 \leq C \sigma_k^2 = \mO(1/k^2)
  \]
  and, in particular, $\lim_{k \to \infty} x^k = x'$
  as well as coincidence of the primal part for each saddle-point.
  Further, choosing a subsequence associated with the indices
  $\seq{k_i}$ such that
  $\wlim_{i \to \infty} \frac1{\sigma_{k_i}} \xi^{k_i} = \xi'$ and
  $\wlim_{i \to \infty} \zeta^{k_i} = \zeta'$
  (such a subsequence must exist), we see with
  $\tau_k = \sigma_0 \tau_0 /\sigma_k$
  that
  \[
  \begin{aligned}
    \wlim_{i \to \infty}\ \tfrac1{\sigma_{k_i}} (x^{k_i} - x')
    &= (\id + \sigma_0\tau_0 \mK^*\mK)^{-1}(\xi' + \mK^*\zeta'), \\
    \wlim_{i \to \infty}\ (y^{k_i} - y')
    &= (\id + \sigma_0\tau_0 \mK \mK^*)^{-1}(\zeta' - \sigma_0\tau_0 \mK \xi').
  \end{aligned}
  \]
  Weakening the statements, we
  conclude that $\wlim_{i \to \infty} (x^{k_i}, y^{k_i}) = (x', y^*)$
  for some $y^* \in Y$.
          Looking at the iteration~\eqref{eq:accelerated-douglas-rachford:notable} again
  yields
  \[
  \left\{
    \begin{aligned}
      \frac1{\sigma_{k_i-1}} (\hat x^{k_i - 1} - \bar x^{k_i})
      - \vartheta_{k_i-1}^{-1}\mK^*(\bar y^{k_i} - \hat y^{k_i - 1})
      &\in \mK^* y^{k_i} + \subgrad \mF(x^{k_i}), \\
      \frac1{\tau_{k_i-1}} (\hat y^{k_i-1} - \bar y^{k_i})
      + \vartheta_{k_i-1}\mK(\bar x^{k_i} - \hat x^{k_i-1})
      &\in  - \mK x ^{k_i} + \subgrad \mG(y^{k_i}),
    \end{aligned}
  \right.
  \]
  so by weak-strong closedness of maximally monotone operators, it follows
  that
  $(0,0) \in
  \bigl(\mK^* y^* + \subgrad \mF(x'), -\mK x' + \subgrad  \mG(y^*) \bigr)$,
  i.e., $(x', y^*)$ is a saddle-point. As the subsequence was arbitrary,
  each weak accumulation point of $\seq{(x^k,y^k)}$ is a saddle-point.

  Suppose that $(q^{*}, y^{*})$ and $(q^{**}, y^{**})$ are both
  weak accumulation points of $\seq{\bigl( \tfrac1{\sigma_k}(x^k - x'),
  y^k \bigr)}$, i.e.,
  $\wlim_{i \to \infty} \bigl( \tfrac1{\sigma_{k_i}}(x^{k_i} - x'),y^{k_i} \bigr)
  = (q^*,y^*)$ and
  $\wlim_{i \to \infty} \bigl( \tfrac1{\sigma_{k_i}} (x^{k'_i} - x'), y^{k'_i} \bigr)
  = (q^{**}, y^{**})$
  for some $\seq{k_i}$ and $\seq{k'_i}$, respectively.
  Denote by $(\xi^k, \zeta^k)$ as above but with $y'$ replaced
  by $y^*$.
                Without loss of
  generality, we may assume that
  \[
  \wlim_{i \to \infty}\ (\tfrac1{\sigma_{k_i}} \xi^{k_i}, \zeta^{k_i})
  = (\xi^{*}, \zeta^{*}),
  \qquad
  \wlim_{i \to \infty}\ (\tfrac1{\sigma_{k'_i}} \xi^{k'_i}, \zeta^{k'_i})
  = (\xi^{**}, \zeta^{**})
  \]
  for some $(\xi^{*}, \zeta^{*}), (\xi^{**}, \zeta^{**}) \in X \times Y$.
  Dividing the first identity in~\eqref{eq:acc1_inversion_step} by
  $\sigma_k$ and passing both identities to the respective weak
  subsequential limits yields
  \[
  \xi^* - \xi^{**} = q^* - q^{**} - \mK^*(y^* - y^{**}),
  \qquad \zeta^* - \zeta^{**} = y^* - y^{**} + \sigma_0\tau_0 \mK (q^* - q^{**}).
  \]
  Now, plugging in the definitions,
  we arrive at
  \begin{align*}
    d^k(x',y^*) - d^k(x',y^{**})
    &= \frac{\sigma_0}{2} \scp{\tfrac{2}{\sigma_k}\xi^k
      + \mK^*(y^{**} - y^*)}{
      \mK^*(y^* - y^{**})}
    + \frac1{2\tau_0} \scp{2\zeta^k + y^{*} - y^{**}}{y^{**} - y^{*}}
                          \end{align*}
  so that rearranging
      implies
  \begin{align*}
    \sigma_0 \scp{\tfrac1{\sigma_k}\xi^k}{\mK^*(y^* - y^{**})}
    + &\frac{1}{\tau_0}\scp{\zeta^k}{y^{**} - y^{*}}
    \\
    &=
      d^k(x', y^*) - d^k(x', y^{**})
      + \frac{\sigma_0}2 \norm{\mK^*(y^* - y^{**})}^2
      + \frac{1}{2\tau_0} \norm{y^* - y^{**}}^2.
  \end{align*}
  The right-hand side converges for the whole sequence, so passing
  the left-hand side to the respective subsequential weak limits
  and using the identities for $\xi^* - \xi^{**}$ and $\zeta^* - \zeta^{**}$
  gives
  \[
  0 = \sigma_0 \scp{\xi^* - \xi^{**}}{\mK^*(y^{**} -y^{*})}
  + \frac1{\tau_0} \scp{\zeta^* - \zeta^{**}}{y^{*} - y^{**}}
  = \sigma_0 \norm{\mK^*(y^* - y^{**})}^2 + \frac1{\tau_0}
  \norm{y^* - y^{**}}^2,
  \]
  and hence, $y^{**} = y^{*}$. Consequently, $y^k \weakto y^*$
  as $k \to \infty$ what was left to show.
                                                                \end{proof}

In addition to the convergence speed $\mO(1/k^2)$ for $\norm{x^k - x}^2$,
restricted primal-dual gaps also converge and a restricted
primal error of energy also admits a rate.

\begin{corollary}
  Let
  $X_0 \times Y_0 \subset \dom \mF \times \dom \mG$ be bounded and
  contain a saddle-point. Then, for the sequence generated
  by~\eqref{eq:accelerated-douglas-rachford_}, it holds that
  \[
  \gap_{X_0 \times Y_0}(x^k, y^k) \to 0 \quad
  \text{as} \quad k \to \infty, \qquad
  \err^p_{Y_0}(x^k) = \mo(1/k).
  \]
\end{corollary}

\begin{proof}
  From the estimate~\eqref{eq:acc1_lagrangian_est0} in the proof
  of Lemma~\ref{lem:acc_lagrangian_est} with $\rho = 1$
  as well as the
  estimates $1 - \vartheta_k \leq \sigma_k
  \gamma$ and $1/\vartheta_k - 1 \leq
  \sigma_k \gamma$  we infer for $(x,y) \in X_0 \times Y_0$ that
  \begin{align}
    \label{eq:acc1_nonergodic_est}
    \notag
    \mL(&x^{k+1},y)
    - \mL(x,y^{k+1}) \\
    \notag
    & \leq
      \sigma_k
      \norm{\tfrac1{\sigma_k}(\hat x^{k} - \bar x^{k+1})}^2
      +
      \norm{\tfrac1{\sigma_k}(\hat x^{k} - \bar x^{k+1})}
      \norm{\hat x^{k} - x + \sigma_k \mK^*y}
          \\
    \notag
    & \quad +
      \sigma_k \frac{1}{\sigma_0\tau_0}
      \bigl[
      \norm{\hat y^{k} - \bar y^{k+1}}^2 +
      \norm{\hat y^{k} - \bar y^{k+1}}
      \norm{\hat y^{k} - y - \tau_k \mK x}
      \bigr]
    \\
    & \quad
      + \sigma_k \gamma \norm{\mK} \bigl[
    \norm{\bar y^{k+1} - \hat y^{k}}
      \norm{x^{k+1} + \bar x^{k+1} - \hat x^{k} - x}
      +
      \norm{\bar x^{k+1} - \hat x^{k}}
      \norm{y^{k+1} + \bar y^{k+1} - \hat y^{k} - y} \bigr].
  \end{align}
  In the proof of Proposition~\ref{prop:acc1_weak_convergence}
  we have seen that
  $\frac1{\sigma_k} \norm{\hat x^{k} - \bar x^{k+1}} \to 0$ and
  $\norm{\hat y^{k} - \bar y^{k+1}} \to 0$ as $k \to \infty$.
  Moreover,
  denoting by $(x', y') \in X_0 \times Y_0$
  a saddle-point of~\eqref{eq:saddle-point-prob} and using the
  notation and results
  from the proof of
  Proposition~\ref{prop:acc1_weak_convergence}, it follows
  that
  \begin{align*}
    \norm{\hat x^k - x + \sigma_k \mK^*y}
    & \leq \norm{\xi^k} + \sup_{x \in X_0} \norm{x - x'} + \sigma_0
      \norm{\mK} \sup_{y \in Y_0} \norm{y - y'} \leq C,
      \pagebreak[0]
    \\
    \sigma_k \norm{\hat y^k - y - \tau_k \mK x}
    & \leq \sigma_0 \norm{\zeta^k} + \sigma_0 \sup_{y \in Y_0}
      \norm{y - y'} + \sigma_0 \tau_0 \norm{\mK} \sup_{x \in X_0}
      \norm{x - x'} \leq C
  \end{align*}
  for a suitable $C > 0$, so we see together with the boundedness
  of $\seq{(x^k,y^k)}$ that the right-hand
  side of~\eqref{eq:acc1_nonergodic_est} tends to zero as $k \to \infty$,
  giving the estimate on $\gap_{X_0\times Y_0}$.

  Regarding the rate on $\err^p_{Y_0}(x^k)$, choosing $X_0 = \sett{x'}$
  and employing the boundedness of $\seq{\tfrac1{\sigma_k}\xi^k}$,
  the refined estimates
  \[
  \norm{\hat x^k - x' + \sigma_k \mK^*y} \leq
  \sigma_k \bigl[ C' + \norm{\mK} \sup_{y \in Y_0} \norm{y- y'} \bigr],
  \quad
  \sigma_k \norm{\hat y^k - y - \tau_k \mK x'}
  \leq
  \sigma_k \bigl[ C' + \sup_{y \in Y_0} \norm{y - y'} \bigr]
  \]
  follow for some $C' > 0$.
  Plugged into~\eqref{eq:acc1_nonergodic_est}, one obtains
  with $\sigma_k = \mO(1/k)$ the rate   \[
  \gap_{\sett{x'} \times Y_0} (x^{k+1},y^{k+1}) =
  \sup_{y \in Y_0} \ \mL(x^{k+1},y) - \mL(x^*,y^{k+1}) = \mo(1/k).
  \]
  Proposition~\ref{prop:restricted-error-func}
  then yields the desired estimate for $\err^p_{Y_0}$.
            \end{proof}

Now, switching again to ergodic sequences, the optimal rate $\mO(1/k^2)$
can be obtained with the accelerated iteration.

\begin{theorem}
  For $X_0 \times Y_0 \subset \dom\mF \times \dom\mG$ bounded and
  containing a saddle-point $(x',y')$,
  the ergodic sequences according to
  \[
  x^k_{\erg} = \nu_k \sum_{k'=0}^{k-1} \lambda_{k'} x^{k'+1},
  \qquad
  y^k_{\erg} = \nu_k \sum_{k'=0}^{k-1} \lambda_{k'} y^{k'+1},
  \]
  where $\seq{(x^k,y^k)}$ is generated
  by~\eqref{eq:accelerated-douglas-rachford_} and $\lambda_k, \nu_k$
  are given by~\eqref{eq:acc1_erg_weights}, converge
  to the same limit $(x^*,y^*)$ as $\seq{(x^k,y^k)}$,
  in the sense that $\norm{x^k_{\erg} - x^*}^2 = \mO(1/k^2)$
  and $\wlim_{k \to \infty} y^k_{\erg} = y^*$.
  The associated
  restricted primal-dual gap satisfies, for some $k_0 \geq 0$
  and $\rho > 0$,
  the estimate
  \begin{align}
    \label{eq:acc1_erg_convergence}
    \notag
        \gap_{X_0 \times  Y_0}(x^k_{\erg}, y^k_{\erg}) \leq
    \nu_k \sup_{(x,y) \in X_0 \times Y_0}\ \Bigl[
    \frac{\sigma_0}{2\rho\sigma_{k_0}^2}
    &\norm{\hat x^{k_0} - x_\rho + \sigma_{k_0} \mK^*y_\rho}^2  +
      \frac{1}{2\rho\tau_{0}} \norm{\hat y^{k_0} - y_\rho - \tau_{k_0} \mK x_\rho}^2
    \\
    & + \sum_{k'=0}^{k_0 - 1} \lambda_{k'}\bigl(\mL(x^{k'+1}, y)
      - \mL(x, y^{k'+1})\bigr)
      \Bigr] = \mO(1/k^2),
  \end{align}
  where again $x_\rho = \rho x + (1-\rho)x'$ and $y_\rho = \rho y + (1-\rho) y'$.
\end{theorem}

\begin{proof}
  As each $x^k_{\erg}$ is a convex combination, the convergence
  $\norm{x^k - x^*} = \mO(1/k)$ implies, together with the
  estimates of Lemma~\ref{lem:acc1_num_seq}, that
  \[
  \norm{x^k_{\erg} - x^*} \leq \nu_k \sum_{k'=0}^{k-1}
  \lambda_{k'} \norm{x^{k'+1} - x^*} \leq C \nu_k \sum_{k'=0}^{k-1}
  \frac{\lambda_{k'}}{k'+1} \leq
  C' \nu_k k = \mO(1/k)
  \]
  for appropriate $C, C' > 0$.
  The weak convergence of $\seq{y^k_{\erg}}$ follows again by the
  Stolz--Ces\`aro theorem: Indeed, for $y \in Y$, we have
  \[
  \lim_{k \to \infty} \scp{y^k_{\erg}}{y} = \lim_{k \to \infty}
  \frac{\sum_{k'=0}^{k-1} \lambda_{k'} \scp{y^{k'+1}}{y}}
  {\sum_{k'=0}^{k-1} \lambda_{k'}} = \lim_{k \to \infty}
  \frac{\lambda_k \scp{y^{k}}{y}}{\lambda_k} = \scp{y^*}{y}.
  \]
  For proving the estimate on the restricted primal-dual gap,
  first observe that as~\eqref{eq:accelerated-douglas-rachford_} requires
  $\gamma < 2\gamma_1/(1 + \sigma_0\tau_0\norm{\mK}^2)$ we also have
  $\gamma < 2\gamma_1/(1 + \rho + (1 + \rho)\sigma_0\tau_0\norm{\mK}^2)$
  for some
  $\rho > 0$ and we can choose again
  $k_0$ such that Proposition~\ref{prop:acc1_lagrangian_est}
  can be applied for
  $k \geq k_0$.
  For $(x,y) \in X_0 \times Y_0$, convexity of
  $(x'',y'') \mapsto \mL(x'', y) - \mL(x,y'')$ and
  estimate~\eqref{eq:acc1_lagrangian_est} yield
  \begin{align*}
    \rho \bigl( \mL&(x^{k}_{\erg},y)
    - \mL(x,y^{k}_{\erg}) \bigr)
      \leq \nu_k \sum_{k'=0}^{k-1} \lambda_{k'}
      \rho \bigl(\mL(x^{k'+1},y) - \mL(x,y^{k'+1}) \bigr)\\
    & \leq \nu_k \sum_{k'=0}^{k_0-1} \lambda_{k'}
      \rho \bigl(\mL(x^{k'+1},y) - \mL(x,y^{k'+1}) \bigr) \\
    & \quad + \nu_k \sum_{k'=k_0}^{k-1}
      \lambda_{k'} \Bigl[ \frac1{2\sigma_{k'}} \norm{\hat x^{k'} - x_\rho + \sigma_{k'} \mK^*y_\rho}^2
      + \frac1{2\tau_{k'}} \norm{\hat y^{k'} - y_\rho - \tau_{k'} \mK x_\rho}^2 \Bigr] \\
    & \qquad \qquad
      - \lambda_{k'+1} \Bigl[ \frac1{2\sigma_{k'+1}} \norm{\hat x^{k'+1} - x_\rho + \sigma_{k'+1} \mK^*y_\rho}^2
      + \frac1{2\tau_{k'+1}} \norm{\hat y^{k'+1} - y_\rho - \tau_{k'+1} \mK x_\rho}^2 \Bigr] \\
    &= \nu_k \sum_{k'=0}^{k_0-1} \lambda_{k'}
      \rho\bigl(\mL(x^{k'+1},y) - \mL(x,y^{k'+1})\bigr) \\
    &\quad  + \nu_k \lambda_{k_0}
      \Bigl[ \frac1{2\sigma_{k_0}} \norm{\hat x^{k_0} - x_\rho + \sigma_{k_0} \mK^*y_\rho}^2
      + \frac1{2\tau_{k_0}} \norm{\hat y^{k_0} - y_\rho - \tau_{k_0} \mK x_\rho}^2 \Bigr] \\
    & \quad
      - \nu_k\lambda_k \Bigl[ \frac1{2\sigma_{k}} \norm{\hat x^{k} - x_\rho + \sigma_{k} \mK^*y_\rho}^2
      + \frac1{2\tau_{k}} \norm{\hat y^{k} - y_\rho - \tau_{k} \mK x_\rho}^2 \Bigr].
  \end{align*}
  The estimate from above for the restricted primal-dual gap
  in~\eqref{eq:acc1_convergence_est} is thus valid.
        Moreover, the rate follows by Lemma~\ref{lem:acc1_num_seq} as
  $\nu_k = \mO(1/k^2)$ and the fact that the supremum on the
  right-hand side of~\eqref{eq:acc1_erg_convergence} is bounded.
  Indeed, the latter follows, from
  applying~\eqref{eq:acc1_lagrangian_est0} for $k' = 0,\ldots, k_0-1$
  as well as taking the boundedness of $X_0$ and $Y_0$ into account.
\end{proof}

\begin{corollary}
  The full dual error obeys
  \[
  \err^d(y_{\erg}^k) = \mO(1/k^2).
  \]
  If, moreover, $\mG$ is strongly coercive, then additionally,
  the primal error satisfies
  \[
  \err^p(x_{\erg}^k) = \mO(1/k^2).
  \]
\end{corollary}

\begin{proof}
  Employing Lemma~\ref{lem:full-gap-errors} on a bounded $Y_0 \subset
  \dom \mG$ containing $y_{\erg}^k$ for all $k$ and, subsequently,
  Proposition~\ref{prop:restricted-error-func}, it is sufficient for
  the estimate on $\err^d$ that
  $\mF$ is strongly coercive. This follows immediately from
  the strong convexity.   The analogous argument can be used in order to obtain the rate for
  $\err^p$, however, we have to assume strong
  coercivity of $\mG$.
        \end{proof}

In particular, we have an optimization scheme for the
dual problem with the optimal rate as well as weak convergence,
comparable to the modified FISTA method presented in \cite{CD}.
Moreover, in many situations, even the primal problem obeys the
optimal rate.

\begin{remark}
  If~\eqref{eq:acc_gamma_restr} is already satisfied for $\sigma_0$,
  $\tau_0$ and $\rho > 0$,
  which is possible for fixed $\sigma_0\tau_0$ by
  choosing $\sigma_0$ and $\rho > 0$
  small enough, then one may choose $k_0 = 0$
  such that
  the estimate on the restricted primal-dual gap reads as
  \begin{equation}
    \label{eq:acc1-dr-O1k2-convergence}
    \gap_{X_0 \times Y_0}(x_{\erg}^k, y_{\erg}^k) \leq \nu_k
    \sup_{(x,y) \in X_0 \times Y_0} \ \Bigl[
    \frac{\norm{\hat x^0 - x_\rho + \sigma_0 \mK^* y_\rho}^2}{2\rho\sigma_0} +
    \frac{\norm{\hat y^0 - y_\rho - \tau_0 \mK x_\rho}^2}{2\rho\tau_0}
    \Bigr],
  \end{equation}
  which has the same structure as~\eqref{eq:dr-O1k-convergence},
  the corresponding estimate for the basic Douglas--Rachford iteration.
      However, for small $\sigma_0\gamma$,
  see~\eqref{eq:acc1_num_seq_est},
  the smallest constant $C > 0$ such that $\nu_k \leq C/k^2$ for
  all $k \geq 1$ might become large.
  This potentially leads to~\eqref{eq:acc1-dr-O1k2-convergence} being
  worse than~\eqref{eq:acc1_erg_convergence} for a
  greater $\sigma_0$.
  Thus, choosing $\sigma_0$ sufficiently small
  such~\eqref{eq:acc1-dr-O1k2-convergence} holds does not offer a clear
  advantage.
\end{remark}

\subsection{\texorpdfstring{Strong convexity of $\mF$ and $\mG$}{Strong convexity of F and G}}

Assume in the following that both $\mF$ and $\mG$ are strongly convex, i.e.,
$\gamma_1 > 0$ and $\gamma_2 > 0$. We obtain an analogous identity
to~\eqref{eq:acc_quadratic_terms}, however, with different $\vartheta$
and without changing the step-sizes $\sigma$ and $\tau$.

\begin{lemma}
  \label{lem:acc2_sqr_norm_repr}
  Let, in the situation of Lemma~\ref{lem:acc_lagrangian_est},
  the values $\gamma >0$ and $\gamma' > 0$ satisfy
  $\sigma\gamma = \tau\gamma'$.
  Then,
  \begin{align}
    \label{eq:acc2_quadratic_terms}
    \notag
    \frac1{2\sigma} \|\bar x^{k+1} - x_\rho &+ \sigma \mK^*y_\rho\|^2
    + \frac\gamma2 \norm{x^{k+1} - x_\rho}^2
      + \frac1{2\tau} \norm{\bar y^{k+1} - y_\rho - \tau \mK x_\rho}^2
      + \frac{\gamma'}2 \norm{y^{k+1} - y_\rho}^2 \\
    \notag
    & =
      \frac1{\vartheta} \Bigl[
      \frac1{2\sigma} \norm{\bar x^{k+1} - x_\rho + \sigma \mK^*y_\rho}^2
      + \frac1{2\tau} \norm{\bar y^{k+1} - y_\rho - \tau \mK x_\rho}^2
      \Bigr] \\
    \notag
    &  \quad
      + \frac{1 - \vartheta}{\vartheta}
      \Bigl[
      \frac{1}{\tau} \scp{y^{k+1} - y_\rho}{\tau \mK x_\rho + \tilde y^{k+1}}
      - \frac1\sigma \scp{x^{k+1} - x_\rho}{\sigma \mK^*y_\rho - \tilde x^{k+1}}
      \Bigr] \\
    & \quad
      - \frac{\gamma}{2}  \norm{\sigma \mK^* y_\rho - \tilde x^{k+1}}^2
      - \frac{\gamma'}{2}  \norm{\tau \mK x_\rho + \tilde y^{k+1}}^2
  \end{align}
  for
  \[
  \vartheta  = \frac{1}{1 + \sigma\gamma}
  = \frac{1}{1 + \tau \gamma'}.
  \]
\end{lemma}

\begin{proof}
  This follows immediately from the identities
  \begin{align*}
    \frac{1}{2\sigma} \|\bar x^{k+1} &- x_\rho
    + \sigma \mK^* y_\rho\|^2
      + \frac{\gamma}2 \norm{x^{k+1} - x_\rho}^2 \\
    &= \frac{1 + \sigma\gamma}{2\sigma} \norm{\bar x^{k+1} - x_\rho
      + \sigma \mK^* y_\rho}^2
      - \gamma \scp{x^{k+1} - x_\rho}{\sigma \mK^* y_\rho - \tilde x^{k+1}}
      - \frac{\gamma}2 \norm{\sigma \mK^* y_\rho - \tilde x^{k+1}}^2
  \end{align*}
  as well as
  \begin{align*}
    \frac{1}{2\tau} \|&\bar y^{k+1}
     - y_\rho
      - \tau \mK x_\rho \|^2
      + \frac{\gamma'}2 \norm{y^{k+1} - y_\rho}^2 \\
    &= \frac{1 + \tau\gamma'}{2\tau} \norm{\bar y^{k+1} - y_\rho - \tau \mK x_\rho}^2
      + \gamma' \scp{y^{k+1} - y_\rho}{\tau \mK x_\rho + \tilde y^{k+1}}
      - \frac{\gamma'}2 \norm{\tau \mK x_\rho + \tilde y^{k+1}}^2.
      \qedhere
  \end{align*}
\end{proof}

\begin{proposition}
  \label{prop:acc2_lagrangian_est}
  In the situation of Lemma~\ref{lem:acc2_sqr_norm_repr},
  with $\vartheta_p = \vartheta_d = \vartheta$,
  $\rho \in [0,1]$ and
  $\gamma > 0$, $\gamma' > 0$ chosen such that
  \begin{equation}
    \label{eq:acc2_gamma_restr}
    \gamma \leq \frac{2\gamma_1}{1 +
      \rho + (1 + \rho + 2\sigma\gamma_1) \sigma\tau \norm{\mK}^2},
    \qquad
    \gamma' = \frac{\gamma\gamma_2}{\gamma_1}
      \end{equation}
  there is a $0 < c < 1$ and a $c' > 0$ such that
  the following estimate holds:
  \begin{align}
    \label{eq:acc2_lagrangian_est}
    \notag
    \rho \bigl( \mL(x^{k+1}&, y)
    - \mL(x, y^{k+1}) \bigr)
      + \frac1\vartheta \Bigl[
      \frac1{2\sigma} \norm{\bar x^{k+1} - x_\rho + \sigma \mK^*y_\rho}^2
      + \frac1{2\tau} \norm{\bar y^{k+1} - y_\rho - \tau \mK x_\rho}^2
      \Bigr]
    \\
    \notag
    & \quad + (1-c) \Bigl[
      \frac1{2\sigma} \norm{\hat x^k - \bar x^{k+1}}^2 +
      \frac1{2\tau} \norm{\hat y^k - \bar y^{k+1}}^2 \Bigr]
      + \frac{c'}{2} \norm{x^{k+1} - x_\rho}^2
      + \frac{c'}{2}\norm{y^{k+1} - y_\rho}^2
    \\
    & \leq \frac1{2\sigma}
      \norm{\hat x^k - x_\rho + \sigma \mK^*y_\rho}^2 + \frac1{2\tau}
      \norm{\hat y^k - y_\rho - \tau \mK x_\rho}^2.
  \end{align}
\end{proposition}

\begin{proof}
  Plugging in the definitions~\eqref{update:x:y:k+1}, the
  scalar-product terms in~\eqref{eq:acc2_quadratic_terms} can
  be written as
  \begin{multline*}
    \frac{1-\vartheta}{\vartheta} \Bigl[ \frac1{\tau} \scp{y^{k+1} - y_\rho}{\tau \mK x_\rho + \tilde y^{k+1}}
    - \frac1\sigma \scp{x^{k+1} - x_\rho}{\sigma \mK^* y_\rho - \tilde x^{k+1}} \Bigr] \\
    = (1 - \vartheta) \bigl[ \scp{\mK[x^{k+1} - x_\rho]}{\bar y^{k+1} - \hat y^k}
    -  \scp{y^{k+1} - y_\rho}{\mK[\bar x^{k+1} - \hat x^k]} \bigr].
  \end{multline*}
  Adding the scalar-product terms in \eqref{eq:acc_lagrangian_est},
  we arrive at
  \begin{align*}
    (1 - \vartheta)
    &\scp{\hat y^k - \bar y^{k+1}}{\mK[x^{k+1} + \bar x^{k+1}
      - \hat x^k - x_\rho]} \\
    & \quad - (1-\vartheta) \scp{y^{k+1} + \bar y^{k+1} - \hat y^k - y_\rho}
      {\mK[\hat x^k - \bar x^{k+1}]} \\
    & \quad + (1- \vartheta) \scp{\mK[x^{k+1} - x_\rho]}{\bar y^{k+1} - \hat y^k}
      - (1 - \vartheta) \scp{y^{k+1} - y_\rho}{\mK[\bar x^{k+1} - \hat x^k]} \\
    &= (1 - \vartheta)\scp{\hat y^k - \bar y^{k+1}}{\mK[\bar x^{k+1} - \hat x^k]}
      - (1 - \vartheta)\scp{\bar y^{k+1} - \hat y^k}
      {\mK[\hat x^k - \bar x^{k+1}]} \\
    &= 0,
  \end{align*}
  so one only needs to estimate the squared-norm terms on the right-hand
  side of~\eqref{eq:acc2_quadratic_terms}. Doing so, employing
  the restriction on $\gamma$ yields $\sigma\gamma\sigma\tau \norm{\mK}^2 < 1$,
  so choosing $c$ according to
  \[
  \max \Bigl( \sigma\gamma\sigma\tau \norm{\mK}^2,
  \frac{1}{(1 + \sigma\gamma)^2} \Bigr) < c < 1
  \]
  allows to obtain
  \begin{align*}
    \frac\gamma2 \norm{\sigma \mK^* y_\rho - \tilde x^{k+1}}^2
    &= \frac{\sigma^2\gamma}2\norm{\mK^*[y^{k+1} - y_\rho + \vartheta(\bar y^{k+1}
      - \hat y^k)]}^2 \\
    &\leq \frac{c}{2\tau} \norm{\bar y^{k+1} - \hat y^k}^2
      + \frac{c(1 + \sigma\gamma)^2\sigma^2\gamma \norm{\mK}^2}
      {2(c(1 + \sigma\gamma)^2 - \sigma^2\tau \gamma \norm{\mK}^2)}
      \norm{y^{k+1} - y_\rho}^2,
  \end{align*}
  as well as
  \[
  \frac{\gamma'}{2} \norm{\tau \mK x_\rho + \tilde y^{k+1}}^2
  \leq \frac{c}{2\sigma} \norm{\bar x^{k+1} - \hat x^k}^2 +
  \frac{c(1 + \tau\gamma')^2\tau^2\gamma' \norm{\mK}^2}
  {2(c(1 + \tau\gamma')^2 - \sigma\tau^2 \gamma' \norm{\mK}^2)}
  \norm{x^{k+1} - x_\rho}^2.
  \]
  As before, we would like to have that the factors in front of
  $\norm{x^{k+1} - x_\rho}^2$ and $\norm{y^{k+1} - y_\rho}^2$ are strictly below
  $\frac{\gamma_1}{1+\rho} - \frac{\gamma}{2}$ and
  $\frac{\gamma_2}{1+\rho} - \frac{\gamma'}{2}$,
  respectively.
  Indeed, with the restriction on $\gamma$, $\gamma_2 = \tfrac{\sigma}{\tau}
  \gamma_1$, $\gamma' = \tfrac{\sigma}{\tau} \gamma$, strict
  monotonicity as well as
  $c(1+\sigma\gamma)^2 > 1$, one sees that
  \[
  \frac{\sigma^2\gamma\norm{\mK}^2}{1 - \sigma\gamma\sigma\tau \norm{\mK}^2}
  \leq \frac2{1+\rho}\gamma_2 - \gamma'
  \quad
  \text{and}
  \quad
  \frac{c(1 + \sigma\gamma)^2\sigma^2\gamma \norm{\mK}^2}
  {2(c(1 + \sigma\gamma)^2 - \sigma^2\tau \gamma \norm{\mK}^2)}
  \leq \frac{\gamma_2}{1 + \rho} - \frac{\gamma'}{2} - \frac{c'_1}{2}
  \]
  for a $c'_1 > 0$. Analogously, again with
  $\gamma' = \tfrac{\sigma}{\tau} \gamma$, it also follows that
  \[
  \frac{\tau^2\gamma'\norm{\mK}^2}{1 - \tau\gamma'\sigma\tau \norm{\mK}^2}
  \leq \frac{2}{1 + \rho}\gamma_1 - \gamma
  \quad
  \text{and}
  \quad
  \frac{c(1 + \tau\gamma')^2\tau^2\gamma' \norm{\mK}^2}
  {2(c(1 + \tau\gamma')^2 - \sigma \tau^2 \gamma' \norm{\mK}^2)}
  \leq \frac{\gamma_1}{1 + \rho} - \frac{\gamma}{2} - \frac{c'_2}{2}
  \]
  for a $c_2' > 0$. Choosing $c' = \min(c_1',c_2')$ and
  combining~\eqref{eq:acc_lagrangian_est},~\eqref{eq:acc2_quadratic_terms}
  with the above estimates thus yields~\eqref{eq:acc2_lagrangian_est}.
\end{proof}

Letting $\hat x^{k+1} = \bar x^{k+1}$ and $\hat y^{k+1} = \bar y^{k+1}$,
we arrive at the \emph{accelerated Douglas--Rachford method for
  strongly convex-concave saddle-point problems}
\begin{equation}
  \label{eq:acc-douglas-rachford-sc}
  \left\{
    \begin{aligned}
      x^{k+1} &= (\id + \sigma\subgrad \mF)^{-1}(\bar x^k), \\
      y^{k+1} &= (\id + \tau\subgrad \mG)^{-1}(\bar y^k), \\
      \bar x^{k+1} &= x^{k+1} - \sigma \mK^*[y^{k+1} + \vartheta(\bar y^{k+1} - \bar y^k)], \\
      \bar y^{k+1} &= y^{k+1} + \tau \mK[x^{k+1} + \vartheta(\bar x^{k+1} - \bar x^k)],
    \end{aligned}
  \right.
\end{equation}
which, together with the restrictions $\sigma\gamma = \tau\gamma'$
and~\eqref{eq:acc2_gamma_restr},
is also summarized in Table~\ref{tab:acc-douglas-rachford-sc}.
\begin{table}
  \centering
  \begin{tabular}{p{0.15\linewidth}p{0.75\linewidth}}
    \toprule

    \multicolumn{2}{l}{\textbf{aDR$^{sc}$}\ \ \ \ \textbf{Objective:}
    \hfill    Solve  \ \ $\min_{x\in \dom \mF} \max_{y \in  \dom \mG}
    \ \langle \mK x,y\rangle + \mF(x) - \mG(y)$
    \hfill\mbox{} }
    \\
    \midrule

    Prerequisites:
 &
   $\mF, \mG$ are strongly convex with respective moduli $\gamma_1,\gamma_2
   > 0$
    \\[\smallskipamount]
    Initialization:
 &
   $(\bar x^0, \bar y^0) \in X \times Y$
   initial guess, $\sigma, \tau > 0$ step sizes with $\sigma\gamma_1
   = \tau\gamma_2$, \\
 &
   $\gamma \leq
   \frac{2\gamma_1}{1 + (1 + 2\sigma\gamma_1)\sigma\tau\norm{\mK}^2}$
   acceleration factor,
   $\vartheta = \frac1{1 + \sigma\gamma}$
    \\[\medskipamount]

    Iteration:
 &
   \raisebox{0.5em}[0mm][7.5\baselineskip]{\begin{minipage}[t]{1.0\linewidth}
       {       \renewcommand{\theHequation}{adrsc}       \begin{equation}\label{eq:acc-douglas-rachford-sc_}\tag{aDR$^{sc}$}
         \left\{
           \begin{aligned}
             x^{k+1} &= (\id + \sigma \subgrad \mF)^{-1}(\bar x^k) \\
             y^{k+1} &= (\id + \tau \subgrad \mG)^{-1}(\bar y^k) \\
             b^{k+1} &=
             (\tfrac{1 + \vartheta}{\vartheta} x^{k+1} - \bar x^k)
             - \vartheta \sigma
             \mK^*(\tfrac{1 + \vartheta}{\vartheta}y^{k+1} - \bar y^k)
             \\
             d^{k+1} &= (\id + \vartheta^2\sigma \tau \mK^*\mK)^{-1}b^{k+1} \\
             \bar x^{k+1} &= \bar x^k - \vartheta^{-1} x^{k+1} + d^{k+1} \\
             \bar y^{k+1} &= y^{k+1} + \vartheta \tau \mK d^{k+1}
           \end{aligned}
         \right.
       \end{equation}}
     \end{minipage}}
    \\
    \vspace*{-\smallskipamount}
    Output:
 &
   \vspace*{-\smallskipamount}
   $\seq{(x^k,y^k)}$  primal-dual sequence  \\
    \bottomrule
  \end{tabular}
  \caption{The accelerated
    Douglas--Rachford iteration for the solution of
    strongly convex-concave saddle-point problems of
    type~\eqref{eq:saddle-point-prob}.}
  \label{tab:acc-douglas-rachford-sc}
\end{table}
As we will see in the following,
the estimate~\eqref{eq:acc2_lagrangian_est} implies R-linear
convergence of the sequences $\seq{(x^k,y^k)}$ and
$\seq{(\bar x^k, \bar y^k)}$, i.e., an exponential rate.

\begin{theorem}
  The iteration~\eqref{eq:acc-douglas-rachford-sc_} converges
  R-linearly to the unique
  saddle-point $(x^*,y^*)$ in the following sense:
  \[
  \lim_{k \to \infty} (x^k,y^k) = (x^*,y^*), \qquad
  \norm{x^k - x^*}^2 = \mo(\vartheta^{k}), \qquad
  \norm{y^k - y^*}^2 = \mo(\vartheta^{k}).
  \]
  Furthermore, denoting $\bar x^* = x^* - \sigma \mK^* y^*$ and
  $\bar y^* = y^* + \tau \mK x^*$, we have
  \[
  \lim_{k \to \infty} (\bar x^k, \bar y^k) = (\bar x^*, \bar y^*), \qquad
  \norm{\bar x^{k} - \bar x^*}^2 = \mo(\vartheta^k),
  \qquad
  \norm{\bar y^{k} - \bar y^*}^2 = \mo(\vartheta^k).
  \]
\end{theorem}

\begin{proof}
  As $\mF$ and $\mG$ are strongly convex, a unique saddle-point
  $(x^*,y^*)$ must exist. Letting $(x',y') = (x^*,y^*)$,
  the choice of $\gamma$ in~\eqref{eq:acc-douglas-rachford-sc_} leads to
  Proposition~\ref{prop:acc2_lagrangian_est} being applicable for $\rho = 0$,
  such that~\eqref{eq:acc2_lagrangian_est} yields
  \begin{align*}
    \frac1{\vartheta^k}\Bigl[
    &\frac1{2\sigma} \norm{\bar x^{k} - x^*
      + \sigma\mK^*y^*}^2 + \frac1{2\tau} \norm{\bar y^{k} - y^*
      - \tau \mK x^*}^2 \Bigr]
    \\
    & \quad + \sum_{k'=0}^{k-1}
      \frac1{\vartheta^{k'}}
      \biggl[ (1-c) \Bigl[
      \frac{\norm{\bar x^{k'} - \bar x^{k'+1}}^2}{2\sigma}
      + \frac{\norm{\bar y^{k'} - \bar y^{k'+1}}^2}{2\tau} \Bigr]
      + c' \frac{\norm{x^{k'+1} - x^*}^2 + \norm{y^{k'+1} - y^*}^2}{2}
       \biggr] \\
    & \leq
      \frac1{2\sigma} \norm{\bar x^{0} - x^*
      + \sigma\mK^*y^*}^2 + \frac1{2\tau} \norm{\bar y^{0} - y^*
      - \tau \mK x^*}^2
  \end{align*}
  for each $k$. Consequently,
  the sequences
  $\seq{\vartheta^{-k}\norm{\bar x^k - \bar x^{k+1}}^2}$,
  $\seq{\vartheta^{-k}\norm{\bar y^k - \bar y^{k+1}}^2}$ as well as
  $\seq{\vartheta^{-k}\norm{x^k - x^*}^2}$,
  $\seq{\vartheta^{-k}\norm{y^k - y^*}^2}$ are summable. This implies the
  stated rates for $\seq{x^k}$ and $\seq{y^k}$ as well as
  $\norm{\bar x^{k+1} - \bar x^{k}}^2 = \mo(\vartheta^k)$ and
  $\norm{\bar y^{k+1} - \bar y^{k}}^2 = \mo(\vartheta^k)$. The convergence
  statements and rates for $\seq{\bar x^k}$ and $\seq{\bar y^k}$
  are then a consequence of~\eqref{eq:acc-douglas-rachford-sc}.
\end{proof}

\begin{corollary}
  The primal-dual gap obeys
  \[
  \gap(x^k,y^k) = \mo(\vartheta^{k/2}).
  \]
\end{corollary}
\begin{proof}
  As $\mF$ and $\mG$ are strongly convex, these
  functionals are also strongly coercive and Lemma~\ref{lem:full-gap-errors}
  yields that $\gap(x^k,y^k) = \gap_{X_0 \times Y_0}(x^k,y^k)$
  for some bounded $X_0 \times Y_0 \subset \dom \mF \times \dom \mG$
  containing the saddle-point $(x^*,y^*)$ and all $k$.
  For $(x,y) \in X_0 \times Y_0$, the
  estimate~\eqref{eq:acc1_lagrangian_est0} gives, for $\rho = 1$,
  \begin{align*}
    \mL(x^{k+1},y) - \mL(x,y^{k+1})
    &\leq \norm{\bar x^k - \bar x^{k+1}} \Bigl[
      \frac1\sigma \norm{\bar x^{k+1} - x + \sigma \mK^*y}
      + \norm{\mK^*(y^{k+1} + \bar y^{k+1} - \bar y^k - y)}
      \Bigr] \\
    & \quad
      + \norm{\bar y^k - \bar y^{k+1}}
      \Bigl[
      \frac1\tau \norm{\bar y^{k+1} - y - \tau \mK x} +
      \norm{\mK(x^{k+1} + \bar x^{k+1} - \bar x^k - x)}
      \Bigr],
  \end{align*}
  so by boundedness, $\gap_{X_0 \times Y_0}(x^{k+1},y^{k+1})
  \leq C (\norm{\bar x^k - \bar x^{k+1}} + \norm{\bar y^k - \bar y^{k+1}} )
  = \mo(\vartheta^{k/2})$.
\end{proof}

\begin{remark}
  Considering the ergodic sequences
  \[
  x^k_{\erg} = \nu_k \sum_{k'=0}^{k-1} \vartheta^{-k'} x^{k'+1},
  \qquad
  y^k_{\erg} = \nu_k \sum_{k'=0}^{k-1} \vartheta^{-k'} y^{k'+1},
  \qquad
  \nu_k = \frac{(1 - \vartheta)\vartheta^{k-1}}{1 - \vartheta^k}
  = \Bigl[\sum_{k'=0}^{k-1} \vartheta^{-k'} \Bigr]^{-1},
  \]
  gives the slightly worse rates $\norm{x^k_{\erg} - x^*}^2 = \mO(\vartheta^k)$
  and $\norm{y^k_{\erg} - y^*}^2 = \mO(\vartheta^k)$.
  Moreover, if the restriction on
  $\gamma$ in~\eqref{eq:acc-douglas-rachford-sc_} is not sharp,
  i.e., $\gamma < \frac{2\gamma_1}{1 + (1 + 2\sigma\gamma_1)
    \sigma\tau\norm{\mK}^2}$, then \eqref{eq:acc2_gamma_restr}
  holds for some $\rho > 0$ and
  the estimate~\eqref{eq:acc2_lagrangian_est} also implies
  \begin{align*}
    \rho \bigl( \mL(x^{k}_{\erg},y) - \mL(x,y^{k}_{\erg}) \bigr)
    & \leq
      \nu_k
      \Bigl[ \frac1{2\sigma} \norm{\bar x^{0} - x_\rho + \sigma \mK^*y_\rho}^2
      + \frac1{2\tau} \norm{\bar y^{0} - y_\rho - \tau \mK x_\rho}^2 \Bigr] \\
    & \quad
      - \nu_k\vartheta^{-k} \Bigl[ \frac1{2\sigma} \norm{\bar x^{k} - x_\rho + \sigma \mK^*y_\rho}^2
      + \frac1{2\tau} \norm{\bar y^{k} - y_\rho - \tau \mK x_\rho}^2 \Bigr].
  \end{align*}
  This leads to the following estimate on the restricted
  primal-dual gap
  \begin{align*}
    \gap_{X_0 \times  Y_0}(x^k_{\erg}, y^k_{\erg}) \leq
    \nu_k \Bigl[ \sup_{X_0 \times Y_0}
    &\frac{1}{2\rho\sigma}
      \norm{\bar x^0 - x_\rho + \sigma \mK^*y_\rho}^2  +
      \frac{1}{2\rho\tau} \norm{\bar y^0 - y_\rho - \tau \mK x_\rho}^2 \Bigr],
  \end{align*}
  and, choosing a bounded $X_0 \times Y_0$ such that
  $\gap(x^k_{\erg}, y^k_{\erg}) = \gap_{X_0 \times Y_0}(x^k_{\erg}, y^k_{\erg})$
  for all $k$,
  to the rate
  \[
  \gap(x^k_{\erg}, y^k_{\erg})
  = \mO(\vartheta^{k}).
  \]
  Thus, while having roughly the same qualitative convergence
  behavior, the ergodic sequences admit a better rate in the
  exponential decay.
\end{remark}

\begin{remark}
  With the choice
  $\sigma \gamma_1 = \tau \gamma_2$,
  the best $\vartheta$ with respect to the
  restrictions~\eqref{eq:acc2_gamma_restr} for fixed
  $\sigma, \tau > 0$ is given by
  \[
  \vartheta = \frac1{1 + \sigma\gamma},
  \qquad
  \gamma = \frac{2\gamma_1}{1 + (1 + 2\sigma\gamma_1) \sigma\tau \norm{\mK}^2}.
  \]
  For varying $\sigma$ and $\tau$,
  this choice of $\vartheta$ becomes minimal when
  $\sigma\gamma$ becomes maximal. Performing the maximization
  in the non-trivial case $\mK \neq 0$
  leads to
  the problem of
  finding the positive root of the third-order polynomial
  $\sigma \mapsto 1 - \tfrac{\gamma_1}{\gamma_2} \norm{\mK}^2 \sigma^2
  -4 \tfrac{\gamma_1^2}{\gamma_2} \norm{\mK}^2 \sigma^3$.
  A very rough approximation of this root
  can be obtained by dropping the cubic term, leading to the easily
  computable parameters
  \[
  \sigma = \sqrt{\frac{\gamma_2}{\gamma_1\norm{\mK}^2}},
  \qquad
  \tau = \sqrt{\frac{\gamma_1}{\gamma_2\norm{\mK}^2}},
  \qquad
  \gamma = \frac{\gamma_1
    \norm{\mK}}{\norm{\mK} + \sqrt{\gamma_1\gamma_2}},
  \qquad
  \gamma' = \frac{\gamma_2
    \norm{\mK}}{\norm{\mK} + \sqrt{\gamma_1\gamma_2}}
  \]
  which, in turn, can be expected to give reasonable convergence properties.
  \end{remark}

\subsection{Preconditioning}

The iterations~\eqref{eq:accelerated-douglas-rachford_}
and~\eqref{eq:acc-douglas-rachford-sc_} can now be preconditioned
with the same technique as for the basic iteration (see
Subsection~\ref{subsec:dr-precon}), i.e., based on the amended saddle-point
problem~\eqref{eq:precon-saddle-point}.

Let us first discuss preconditioning
of~\eqref{eq:accelerated-douglas-rachford_}: With $\mF$ being strongly
convex, the primal functional in~\eqref{eq:precon-saddle-point} is strongly
convex, so we may indeed employ the accelerated Douglas--Rachford
iteration to this problem and
obtain, in particular, an ergodic convergence rate
of $\mO(1/k^2)$.
Taking the explicit structure into account similar to
Subsection~\ref{subsec:dr-precon},
one arrives at the \emph{preconditioned accelerated Douglas--Rachford}
method according to Table~\ref{tab:pre:accelerated-douglas-rachford}.
In particular,
the update step for $d^{k+1}$ again corresponds to an iteration
step for the solution of $Td^{k+1} = b^{k+1}$ with respect to the matrix
splitting $T = M + (T - M)$.
In order to obtain convergence,
one needs that $M$ is a feasible preconditioner for $T =
\id + \sigma_0\tau_0 \mK^*\mK$ and has choose the parameter
$\gamma$ according to $\gamma < 2\gamma_1/(1 +
\sigma_0\tau_0\norm{\mK^*\mK + \mH^*\mH}) = 2\gamma_1 / \norm{M}$.

\begin{table}
  \centering
  \begin{tabular}{p{0.15\linewidth}p{0.75\linewidth}}
    \toprule

    \multicolumn{2}{l}{\textbf{paDR}\ \ \ \ \textbf{Objective:}
    \hfill    Solve  \ \ $\min_{x\in \dom \mF} \max_{y \in  \dom \mG}
    \ \langle \mK x,y\rangle + \mF(x) - \mG(y)$
    \hfill\mbox{} }
    \\
    \midrule

    Prerequisites:
 &
   $\mF$ is strongly convex with modulus $\gamma_1 > 0$     \\[\smallskipamount]
    Initialization:
 &
   $(\hat x^0, \hat y^0) \in X \times Y$
   initial guess, $\sigma_0 > 0, \tau_0 > 0$ initial step sizes, \\
 &
   $M$ feasible preconditioner for
   $T = \id + \sigma_0\tau_0 \mK^*\mK$, \\
 &
   $0 < \gamma < \frac{2\gamma_1}{\norm{M}}$
   acceleration factor,
   $\vartheta_0 = \frac1{\sqrt{1 + \sigma_0\gamma}}$
    \\[\medskipamount]
    Iteration:
 &
   \raisebox{1.6em}[0mm][9.1\baselineskip]{\begin{minipage}[t]{1.0\linewidth}
       {       \renewcommand{\theHequation}{padr}       \begin{equation}\label{eq:pre:accelerated-douglas-rachford}\tag{paDR}
         \left\{
           \begin{aligned}
             x^{k+1} &= (\id + \sigma_k\subgrad \mF)^{-1}(\hat x^k) \\
             y^{k+1} &= (\id + \tau_k\subgrad \mG)^{-1}(\hat y^k) \\
             b^{k+1} &= ((1 + \vartheta_k) x^{k+1} - \vartheta_k\hat x^k)
             - \sigma_k\mK^*((1 + \vartheta_k)y^{k+1} - \hat y^k) \\
             d^{k+1} &= d^{k} + M^{-1} (b^{k+1} - T d^{k}) \\
             \hat x^{k+1} &= \vartheta_k(\hat x^k - x^{k+1}) + d^{k+1} \\
             \hat y^{k+1} &= y^{k+1} + \vartheta_k^{-1} \tau_k \mK d^{k+1} \\
             \sigma_{k+1} &= \vartheta_k\sigma_k,
             \quad \tau_{k+1} = \vartheta_k^{-1}\tau_k,
             \quad \vartheta_{k+1} = \tfrac1{\sqrt{1 + \sigma_{k+1}\gamma}} \\
           \end{aligned}
         \right.
       \end{equation}}
     \end{minipage}}
    \\
    \vspace*{-\smallskipamount}
    Output:
 &
   \vspace*{-\smallskipamount}
   $\seq{(x^k,y^k)}$  primal-dual sequence  \\
    \bottomrule
  \end{tabular}
  \caption{The preconditioned accelerated Douglas--Rachford iteration for the solution of
    convex-concave saddle-point problems of
    type~\eqref{eq:saddle-point-prob}.}
  \label{tab:pre:accelerated-douglas-rachford}
\end{table}

The latter condition, however, requires an estimate
on $\norm{M}$ which usually has to be
obtained case by case. One observes, nevertheless, the following properties
for the feasible preconditioner $M_n$ corresponding
to the $n$-fold application of the feasible preconditioner $M$.
In particular, it will turn out that
applying a preconditioner multiple times only requires
an estimate on $\norm{M}$.
Furthermore, the norms $\norm{M_n}$ approach $\norm{T}$ as $n \to \infty$.

\begin{proposition}\label{multiple:precon:norm}
  Let $M$ be a feasible preconditioner for $T$.
  For each $n \geq 1$, denote
  by $M_n$ the $n$-fold application of $M$, i.e., such that
  $d^{k+1} = d^k + M_n^{-1}(b^{k+1} - Td^k)$ holds for $d^{k+1}$
  in~\eqref{eq:precon-nfold} of Proposition~\ref{prop:feasible-precon}.
  Then,
  the norm of the $M_n$ is monotonically decreasing for increasing
  $n$, i.e.,
  \[
  \|M_{n}\| \leq \|M_{n-1}\| \leq \ldots \leq \norm{M_1} = \norm{M}.
  \]
  Furthermore, if $T$ is invertible, the spectral radius of
  $\id - M^{-1}T$ obeys
  $\rho(\id - M^{-1}T) < 1$ and
  \begin{equation}
    \label{eq:preconditioner_norm_est}
    \norm{M_n} \leq \frac{\norm{T}}{1
      -
                  \norm{M^{-1}} \rho(\id - M^{-1}T)^n
    }
  \end{equation}
  whenever the expression on the right-hand side is non-negative.
\end{proposition}

\begin{proof}
  The proof is adapted from the proof of Proposition 2.15 of \cite{BS1}
  from which we need the following two conclusions.
    One is that for any linear, self-adjoint and positive definite
  pair of operators $B$ and $S$ in the Hilbert space $X$,
  the following three statements are equivalent:
  $B - S \geq 0$, $\id - B^{-1/2}SB^{-1/2} \geq 0$ and $\sigma(\id - B^{-1}S)
  \subset {[{0,\infty}[}$ where $\sigma$ denotes the spectrum.
  Here, $B^{1/2}$ is the square root of $B$, and $B^{-1/2}$ is its inverse.
  Actually, we also have
  \[
  \sigma(\id - B^{-1}S) = \sigma \bigl(B(\id - B^{-1}S)B^{-1} \bigr) 
  = \sigma(\id - SB^{-1})\subset {[{0,\infty}[},
  \]
  while $\sigma(\id - SB^{-1})\subset {[{0,\infty}[}$ also means
  $S^{-1} \geq B^{-1}$.

  The other conclusion from Proposition 2.15 of \cite{BS1} is the recursion relation for $M_n$ according to
  \begin{equation*}
  M_{n+1}^{-1}   = M_n^{-1} + (\id - M^{-1}T)^nM^{-1}
  = M_n^{-1} + M^{-1/2}(\id - M^{-1/2}TM^{-1/2})^n M^{-1/2}
  \end{equation*}
  with $\id - M^{-1/2}TM^{-1/2}$ being a positive semidefinite operator. Actually, since $M^{-1/2}$ is also self-adjoint, we have that $\id - M^{-1/2}TM^{-1/2}$ is also self-adjoint, and it follows that
  \[
  M^{-1/2}(\id - M^{-1/2}TM^{-1/2})^n M^{-1/2} \geq 0 \quad
  \Rightarrow  \quad M_{n+1}^{-1}  \geq M_{n}^{-1}.
  \]
  By the above equivalence, this leads to
  \[
  M_{n+1} \leq M_{n}, \quad \text{and}  \quad \|M_{n+1}\| \leq \|M_{n}\|
 \]
 from which the monotonicity behavior follows by induction.

 For the estimate on $\norm{M_n}$ in case $T$ is invertible, introduce
 an equivalent norm on $X$ associated with the $M$-scalar product
 $\scp[M]{x}{x'} = \scp{Mx}{x'}$. Then, $\id - M_n^{-1}T$ is
 self-adjoint with respect to the $M$-scalar product:
 \begin{align*}
   \scp[M]{(\id - M_n^{-1}T)x}{x'}
   = \scp{M(\id - M^{-1}T)^nx}{x'} &=
     \scp{(\id - TM^{-1})^n M x}{x'} \\
   &= \scp{Mx}{(\id - M^{-1}T)^n x'}
     =\scp[M]{x}{(\id - M_n^{-1}T)x'}
 \end{align*}
 for $x,x' \in X$.
  Let $0 \leq q \leq 1$ and suppose that $x \neq 0$
 is chosen such that
 $\norm[M]{(\id - M^{-1}T)x}^2 \geq q^2 \norm[M]{x}^2$. Then,
 \[
 (1-q^2) \scp{Mx}{x} =
 (1 - q^2)\norm[M]{x}^2 \geq \scp{Tx}{x}
 + \scp{(\id - T^{1/2}M^{-1}T^{1/2}) T^{1/2} x}{T^{1/2}x}
 \geq \scp{Tx}{x},
 \]
 as $\id - T^{1/2}M^{-1}T^{1/2}$ is positive semi-definite due to
 $\sigma(\id - T^{1/2}M^{-1}T^{1/2}) = \sigma(\id - M^{-1}T)
 \subset [0,\infty[$.
 Consequently, since $T$ is continuously invertible,
 \[
 \norm{T^{-1}}^{-1} \norm{x}^2 \leq (1-q^2) \norm{M} \norm{x}^2
 \qquad \Rightarrow
 \qquad
 q^2 \leq 1 - \norm{T^{-1}}^{-1}\norm{M}^{-1}.
 \]
 Hence, $q$ must be bounded away from $1$ and, consequently,
 \[
 \rho(\id - M^{-1}T) = \norm[M]{\id - M^{-1}T} < 1
 \]
 as the contrary would result in a contradiction. Now,
 $\norm[M]{\id - M_n^{-1}T} \leq \norm[M]{\id - M^{-1}T}^n
 = \rho(\id-M^{-1}T)^n$ as well as
 \begin{multline*}
   \norm{\id - M_n^{-1}T}
   = \norm[M]{M^{-1/2}(\id - M_n^{-1}T)M^{-1/2}}
   \\
   \leq \norm[M]{M^{-1/2}}^2 \rho(\id - M^{-1}T)^n
   = \norm{M^{-1}} \rho(\id - M^{-1}T)^n.
 \end{multline*}
 Eventually, estimating
 \[
 \norm{M_n} \leq \norm{T} + \norm{M_n}\norm{\id - M_n^{-1}T}
 \leq \norm{T} + \norm{M_n}\norm{M^{-1}} \rho(\id - M^{-1}T)^n
 \]
 gives, if $\norm{M^{-1}} \rho(\id - M^{-1}T)^n < 1$,
 the desired estimate~\eqref{eq:preconditioner_norm_est}.
                        \end{proof}

\begin{remark}
    Let us shortly
  discuss possible ways to  estimate $\norm{M}$ for the classical
  preconditioners in Table~\ref{tab:choice-preconditioner}.
  \begin{itemize}
  \item For the Richardson preconditioner, it is clear that $\|M\| = \lambda$ with $\lambda \geq \|T\|$ being an estimate on $\|T\|$.

  \item For the damped Jacobi preconditioner,
    one has to choose $\lambda \geq \lambda_{\max}(T - D)$
    where $\lambda_{\max}$ denotes the greatest eigenvalue,
    which results in
    $\norm{M} = (\lambda + 1) \norm{D}$. The norm of the diagonal matrix
    $D$ is easily obtainable as the maximum of the entries.
                
  \item
    For symmetric SOR preconditioners, which include the
    symmetric Gauss--Seidel case, we can express $M - T$ as
    \begin{equation}\label{eq:preconditioner:SSOR:general}
    M - T = \frac{\omega}{2 - \omega} \bigl( \tfrac{1 - \omega}{\omega} D
    + E \bigr) D^{-1} \bigl( \tfrac{1 - \omega}{\omega} D + E^* \bigr)
    \end{equation}
    for $T = D - E - E^*$, $D$ diagonal and $E$ lower triangular
    matrix of $T$, respectively, as well as $\omega \in {]{0,2}[}$.
    Hence, $\norm{M - T}$ can be estimated by estimating
    the squared norm of $\tfrac{1-\omega}{\omega} D^{1/2} + D^{-1/2}E^*$.
    This can, for instance, be done as follows.
    Denoting by $t_{ij}$ the entries of $T$ and performing some
    computational steps, we get for fixed $i$ that
    \begin{align*}
    (\tfrac{1-\omega}{\omega} D^{1/2}x + D^{-1/2}E^*x)_i^2
      &=
        \Bigl(\frac{1-\omega}{\omega} t_{ii}^{1/2} x_i -
        \frac{1}{t_{ii}^{1/2}}
        \sum_{j > i} t_{ij} x_j \Bigr)^2
      \\
      &\leq
      \Bigl[ \Bigabs{\frac{1-\omega}{\omega}}^2 t_{ii}
        + \Bigl( \sum_{j > i} \Bigabs{\frac{t_{ij}}{t_{ii}}} \Bigr)
        \max_{j > i} \abs{t_{ij}} \Bigr] \norm{x}^2
    \end{align*}
    leading to
    \[
    \bignorm{(\tfrac{1-\omega}{\omega} D^{1/2} + D^{-1/2}E^*)}^2
    \leq \abs{\tfrac{1 - \omega}{\omega}}^2 \norm{D} +
    \norm[\infty]{\diag(D^{-1}E^{*})} \max_{j \neq i} \abs{t_{ij}}
    = K.
    \]
    where $\norm[\infty]{\,\cdot\,}$ is the maximum row-sum norm
    and symmetry of $T$ has been used.
    This results in
    \[
    \norm{M} \leq \norm{T} + \frac{\omega}{2-\omega} K.
    \]
    In particular, in case that $T$ is weakly diagonally dominant, then
    $\norm[\infty]{D^{-1}E} \leq 1$ such that for the Gauss--Seidel
    preconditioner one obtains the easily computable estimate
    $\norm{M} \leq \norm{T} + \max_{j \neq i} \abs{t_{ij}}$.

\end{itemize}
\end{remark}

Let us next discuss preconditioning of~\eqref{eq:acc-douglas-rachford-sc_}.
It is immediate that for
strongly convex $\mF$ and $\mG$, the acceleration strategy in~\eqref{eq:acc-douglas-rachford-sc} may be
employed for the
modified saddle-point problem~\eqref{eq:precon-saddle-point}. As before,
this results in a preconditioned version
of~\eqref{eq:acc-douglas-rachford-sc_} with $M = \id + \vartheta^2
\sigma\tau(\mK^*\mK + \mH^*\mH)$ being a feasible preconditioner for
$T= \id + \vartheta^2 \sigma\tau \mK^*\mK$. However, as one usually
constructs $M$ for given $T$, it will in this case
depend on $\vartheta$ and, consequently,
on the restriction on $\gamma$ which reads in this context as
\begin{equation}
  \label{eq:precon_strong_convex_gamma}
  \gamma \leq \frac{2\gamma_1}{1 + (1 + 2\sigma\gamma_1)\vartheta^{-2}(\norm{M} - 1)}.
\end{equation}
As $\vartheta = \frac{1}{1 + \sigma\gamma}$, the condition on $\gamma$ is implicit and it is not clear whether it will be satisfied for a given $M$.
Nevertheless, if the latter is the case, preconditioning yields the iteration
scheme~\eqref{eq:pre:acc-douglas-rachford-sc_}
in Table~\ref{tab:pre:acc-douglas-rachford-sc-pre} which inherits all
convergence properties of the unpreconditioned iteration.
So, in order to satisfy the conditions on $M$ and $\gamma$, we
introduce and discuss a notion that leads to sufficient
conditions.

\begin{table}
  \centering
  \begin{tabular}{p{0.15\linewidth}p{0.75\linewidth}}
    \toprule

    \multicolumn{2}{l}{\textbf{paDR$^{sc}$}\ \ \ \ \textbf{Objective:}
    \hfill    Solve  \ \ $\min_{x\in \dom \mF} \max_{y \in  \dom \mG}
    \ \langle \mK x,y\rangle + \mF(x) - \mG(y)$
    \hfill\mbox{} }
    \\
    \midrule

    Prerequisites:
 &
   $\mF, \mG$ are strongly convex with respective moduli $\gamma_1,\gamma_2,
   > 0$
       \\[\smallskipamount]
    Initialization:
 &
   $(\bar x^0, \bar y^0) \in X \times Y$
   initial guess, $\sigma, \tau > 0$ step sizes with $\sigma\gamma_1
   = \tau\gamma_2$, \\
 &
   $\gamma$ acceleration factor, $M$
   feasible preconditioner
   for $T = \id + \vartheta^2\sigma\tau \mK^*\mK$
    \\
 &
   and such that
   $\gamma \leq \frac{2\gamma_1}{1 + (1 + 2\sigma\gamma_1)\vartheta^{-2}(\norm{M} - 1)}$
   where
   $\vartheta = \frac1{1 + \sigma\gamma}$
    \\[\medskipamount]

    Iteration:
 &
   \raisebox{0.9em}[0mm]{\begin{minipage}[t]{1.0\linewidth}
       \begin{equation}\label{eq:pre:acc-douglas-rachford-sc_}\tag{paDR$^{sc}$}
         \left\{
           \begin{aligned}
             x^{k+1} &= (\id + \sigma \subgrad \mF)^{-1}(\bar x^k) \\
             y^{k+1} &= (\id + \tau \subgrad \mG)^{-1}(\bar y^k) \\
             b^{k+1} &=
             (\tfrac{1 + \vartheta}{\vartheta} x^{k+1} - \bar x^k)
             - \vartheta \sigma
             \mK^*(\tfrac{1 + \vartheta}{\vartheta}y^{k+1} - \bar y^k)
             \\
             d^{k+1} &= d^{k} + M^{-1}(b^{k+1} - Td^{k}) \\
             \bar x^{k+1} &= \bar x^k - \vartheta^{-1} x^{k+1} + d^{k+1} \\
             \bar y^{k+1} &= y^{k+1} + \vartheta \tau \mK d^{k+1}
           \end{aligned}
         \right.
       \end{equation}
     \end{minipage}}
    \\
    \vspace*{-\smallskipamount}
    Output:
 &
   \vspace*{-\smallskipamount}
   $\seq{(x^k,y^k)}$  primal-dual sequence  \\
    \bottomrule
  \end{tabular}
  \caption{The preconditioned accelerated
    Douglas--Rachford iteration for the solution of
    strongly convex-concave saddle-point problems of
    type~\eqref{eq:saddle-point-prob}.}
  \label{tab:pre:acc-douglas-rachford-sc-pre}
\end{table}

\begin{definition}
  \label{def:theta_norm_monotone}
  Let $(M_\vartheta)_\vartheta$, $\vartheta \in [0,1]$
  be a family of feasible preconditioners for
  $T_\vartheta = \id + \vartheta^2 T'$ where $T' \geq 0$. We
  call
  this family \emph{$\vartheta$-norm-monotone with respect to $T'$}, if
  $\vartheta \leq \vartheta'$ implies $\vartheta^{-2}(\norm{M_\vartheta} - 1)
  \leq (\vartheta')^{-2}(\norm{M_{\vartheta'}}-1)$ 
  for $\vartheta,\vartheta' \in [0,1]$.
\end{definition}

\begin{lemma}
  \label{lem:strong_convex_gamma_choice}
  If $(M_\vartheta)_\vartheta$ is $\vartheta$-norm-monotone with respect to
  $\sigma\tau \mK^*\mK$ and
  $\gamma_0 > 0$ satisfies, for some $\vartheta' \in [0,1]$,
  \[
  0 <
  \gamma_0 \leq \frac{2\gamma_1}{1 + (1 + 2\sigma\gamma_1)(\vartheta')^{-2}(\norm{M_{\vartheta'}} - 1)},
  \qquad
  \frac1{1+\sigma\gamma_0} \leq \vartheta',
  \]
  then, $\gamma = \gamma_0$
  satisfies~\eqref{eq:precon_strong_convex_gamma} for
  $M = M_{\vartheta}$ with $\vartheta = \frac1{1 + \sigma\gamma_0}$ and
  we have, for each $\gamma > 0$, the implication
  \[
  \gamma_0 \leq \gamma \leq
  \frac{2\gamma_1}{1 + (1 + 2\sigma\gamma_1)\vartheta^{-2}(\norm{M_{\vartheta}} - 1)}
  \qquad
  \Rightarrow
  \qquad
  \frac1{1 + \sigma\gamma} \leq \vartheta.
  \]
\end{lemma}

\begin{proof}
  By the choice of $\gamma_0$ and $\vartheta$-norm-monotonicity we
  immediately see that $\gamma_0 \leq 2\gamma_1/\bigl(1 + (1 + 2\sigma\gamma_1)
  \vartheta^{-2}(\norm{M_{\vartheta}} - 1)\bigr)$
  for $\vartheta = (1 + \sigma\gamma_0)^{-1}$,
  giving~\eqref{eq:precon_strong_convex_gamma} as stated. The remaining
  conclusion is just a consequence of the monotonicity of $\gamma
  \mapsto (1 + \sigma\gamma)^{-1}$.
\end{proof}

\begin{remark}
  \label{rem:norm_monotone_iteration}
  The result of Lemma~\ref{lem:strong_convex_gamma_choice} can be used
  as follows. Choose a family of
  feasible preconditioners $(M_\vartheta)_{\vartheta}$
  for $\id + \vartheta^2 \sigma\tau \mK^*\mK$
  which is
  $\vartheta$-norm-monotone with respect to $\sigma\tau\mK^*\mK$.
  Starting with $\vartheta' = 1$ and $\norm{M_1}$,
  choosing $\gamma_0$ according to
  the lemma already establishes~\eqref{eq:precon_strong_convex_gamma} for
  $\gamma = \gamma_0$. However, the corresponding
  $\norm{M_\vartheta}$ yields an improved bound for $\gamma_0$ and a greater
  $\gamma_0$ can be used instead without
  violating~\eqref{eq:precon_strong_convex_gamma}. This procedure
  may be iterated in order to obtain a good acceleration factor $\gamma$.
  Observe that the argumentation remains valid if the norm $\norm{M_\vartheta}$
  is replaced by some estimate $K_\vartheta \geq \norm{M_\vartheta}$
  in Definition~\ref{def:theta_norm_monotone} and
  Lemma~\ref{lem:strong_convex_gamma_choice}.
\end{remark}

\begin{remark}
  Let us discuss feasible preconditioners for
  $T_\vartheta = \id + \vartheta^2T'$, $T' \geq 0$ that
  are $\vartheta$-norm-monotone with respect to $T'$.
  \begin{itemize}
  \item For the Richardson preconditioner, choose $\lambda: [0,1] \to \RR$
    for which $\vartheta \mapsto \vartheta^{-2} \lambda(\vartheta)$
    is monotonically increasing and
    $\lambda(\vartheta) \geq \vartheta^2\norm{T'}$ for each $\vartheta
    \in [0,1]$. Then, $M_\vartheta = (\lambda(\vartheta) + 1)\id$ defines
    a family of feasible preconditioners for $T_\vartheta$ which is
    $\vartheta$-norm-monotone as $\vartheta^{-2}(\norm{M_\vartheta}-1) =
    \vartheta^{-2}\lambda(\vartheta)$
    for each $\vartheta \in {]{0,1}]}$.
  \item
    For the damped Jacobi preconditioner, we see that for $D_\vartheta$
    the diagonal matrix of $T_\vartheta$ we have
    $\norm{D_\vartheta} = 1 + \vartheta^2 \norm{D'}$ where $D'$ is the diagonal
    matrix of $T'$. Furthermore, $T_\vartheta - D_\vartheta =
    \vartheta^2(T' - D')$, so choosing $\lambda: [0,1] \to \RR$ such that
    $\vartheta \mapsto \vartheta^{-2}\lambda(\vartheta)$
    is monotonically increasing and $\lambda(\vartheta) \geq
    \vartheta^2 \norm{T' - D'}$ yields feasible
    $M_\vartheta = (\lambda(\vartheta) + 1)D_\vartheta$. This family
    is $\vartheta$-norm-monotone with respect to $T'$
    as $\vartheta^{-2}(\norm{M_\vartheta}-1) = \vartheta^{-2}\lambda(\vartheta)
    + (\lambda(\vartheta) + 1) \norm{D'}$ is monotonically increasing.
  \item
    For the symmetric Gauss--Seidel preconditioner, we have
    $M_{\vartheta} = (D_\vartheta - E_\vartheta)D_\vartheta^{-1}(D_\vartheta
    - E_\vartheta^*) = T_\vartheta + E_\vartheta D_\vartheta^{-1} E_\vartheta^*$
    where $T_\vartheta = D_\vartheta - E_\vartheta - E_\vartheta^*$, $D_\vartheta$
    is the
    diagonal and $E_\vartheta$ lower diagonal
    matrix of $T_\vartheta$, respectively. Obviously, $D_\vartheta$ and
    $E_\vartheta$ admit the form $D_\vartheta = \id + \vartheta^2 D'$
    and $E_\vartheta = \vartheta^2 E'$ for respective $D'$ and $E'$.
    Observe that
    $\vartheta \mapsto \vartheta^2/(1 + \vartheta^2 d)$ is monotonically
    increasing for each $d \geq 0$. Consequently,
    it follows from
    $\vartheta \leq \vartheta'$ that $\vartheta^2 D_{\vartheta}^{-1}
    \leq (\vartheta')^2 D_{\vartheta'}^{-1}$ and, further that
    \[
    \vartheta^{-2} E_{\vartheta}D_{\vartheta}^{-1}E_\vartheta^* = E'
    \vartheta^2D^{-1}_\vartheta (E')^* \leq
    E' (\vartheta')^2D^{-1}_{\vartheta'} (E')^* =
    (\vartheta')^{-2} E_{\vartheta'}D_{\vartheta'}^{-1}E_{\vartheta'}^*.
    \]
    As $\vartheta^{-2} (T_\vartheta - \id)$ is independent from $\vartheta$,
    we get $\vartheta^{-2}(M_{\vartheta} - \id)
    \leq (\vartheta')^{-2}(M_{\vartheta'} - \id)$.
    In particular, $\vartheta^{-2}(\norm{M_\vartheta} - 1)
    \leq (\vartheta')^{-2}(\norm{M_{\vartheta'}}- 1)$
    which shows that the symmetric
    Gauss--Seidel preconditioners
    $(M_\vartheta)_\vartheta$ constitute a $\vartheta$-norm-monotone
    family.
  \item
    The SSOR methods are in general not $\vartheta$-norm-monotone,
    except for the symmetric Gauss--Seidel case $\omega = 1$,
    as discussed above.
                                                                                                  \end{itemize}
\end{remark}

\section{Numerical experiments}
\label{sec:numerics}

In order to assess the performance of the discussed Douglas--Rachford
algorithms, we performed numerical experiments for two specific
convex optimization problems in imaging: image denoising with
$L^2$-discrepancy and $\TV$ regularization as well as
$\TV$-Huber approximation, respectively.

\subsection{\texorpdfstring{$L^2$-$\TV$ denoising}{L2-TV denoising}}

We tested the basic and the accelerated Douglas--Rachford iteration
for a discrete $L^2$-$\TV$-denoising problem (also called \emph{ROF model})
according to
\begin{equation}
  \label{eq:rof_primal}
  \min_{u \in X} \ \frac{\norm[2]{u - f}^2}{2} + \alpha \norm[1]{\grad u},
\end{equation}
with $X$ the space of functions on a 2D regular rectangular grid and
$\grad: X \to Y$ a forward finite-difference gradient operator with
homogeneous Neumann conditions.
The minimization problem may be written as the
saddle-point problem~\eqref{eq:saddle-point-prob} with
primal variable $x = u \in X$, dual variable $y = p \in Y$,
$\mF(u) = \|u-f\|_{2}^2/2$, $\mG(p) = \mI_{\{ \|p\|_{\infty} \leq \alpha\}}(p)$
and $\mK = \grad$.
The associated dual problem then reads as
\begin{equation}\label{ROF:dual}
\min_{p\in Y} \ \frac{\norm[2]{\Div p + f}^2}{2}
  - \frac{\norm[2]{f}^2}{2} + \mI_{\{\norm[\infty]{p} \leq \alpha\}}(p).
\end{equation}
such that the primal-dual gap, which is used as the stopping
criterion, becomes
\begin{equation}
  \label{eq:l2-tv-gap}
  \gap_{L^2\text{-}\TV}(u,p) = \frac{\norm[2]{u-f}^2}{2}
  + \alpha \norm[1]{\grad u} + \frac{\norm[2]{\Div p + f}^2}{2}
  - \frac{\norm[2]{f}^2}{2} + \mI_{\{\norm[\infty]{p} \leq \alpha\}}(p).
\end{equation}
As $\mF$ is strongly convex, the accelerated
methods~\eqref{eq:accelerated-douglas-rachford_}
and~\eqref{eq:pre:accelerated-douglas-rachford} proposed in this
paper are applicable for solving $L^2$-TV denoising problems, in addition
to the basic methods~\eqref{eq:douglas-rachford_}
and~\eqref{eq:precon-douglas-rachford_}.
The computational
building blocks for the respective implementations are well-known, see
\cite{BS1,BS2,CP1}, for instance.
We compare with first-order
methods which also achieve $\mO(1/k^2)$ convergence rate in some sense,
namely ALG2 from \cite{CP1}, FISTA \cite{BT} and the fast Douglas--Rachford
iteration in \cite{PSB}.
The parameter settings for the $\mO(1/k^2)$-methods are as follows:
\begin{itemize}
\item ALG2: Accelerated primal-dual algorithm
  introduced in \cite{CP1} with parameters
  $\tau_{0} = 1/L$,
  $\tau_{k} \sigma_{k} L^2 = 1$, $\gamma = 0.35$ and $L = \sqrt{8}$.

\item FISTA:  Fast iterative shrinkage thresholding algorithm~\cite{BT}
  on the dual problem~\eqref{ROF:dual}. The Lipschitz constant estimate
  $L$ as in \cite{BT} here is chosen as $L = 8$.

\item FDR: Fast Douglas--Rachford splitting method proposed in \citep{PSB}
  on the dual problem~\eqref{ROF:dual} with parameters
    $L_{f} = 8$, $\gamma = \frac{\sqrt{2}-1}{L_{f}}$, $\lambda_{k} = \lambda = \frac{1-\gamma L_{f}}{1+\gamma L_{f}}$, $\beta_{0} = 0$, and $\beta_{k} = \frac{k-1}{k+2}$ for $k>0$.

\item \ref{eq:accelerated-douglas-rachford_}: Accelerated
  Douglas--Rachford method as in Table \ref{tab:accelerated-douglas-rachford}.
  Here, the discrete cosine transform (DCT) is used for solving
  the discrete elliptic equation with operator $\id
  - \sigma_0 \tau_0 \Delta$.
  The parameters read as
  $L = \sqrt{8}$, $\sigma_{0} = 1$, $\tau_{0} = 15/\sigma_{0}$, $\gamma  = 1/(1 + \sigma_{0}\tau_{0}L^2)$. 
\item \ref{eq:pre:accelerated-douglas-rachford}: Preconditioned accelerated
  Douglas--Rachford method as in Table
  \ref{tab:pre:accelerated-douglas-rachford} with two steps
  of the symmetric Red-Black Gauss--Seidel iteration
  as preconditioner and parameters
  $\sigma_{0} = 1$, $\tau_{0} = 15/\sigma_{0}$, $\gamma_{1} = 1$,
  $L = \sqrt{8}$.
  The norm $\norm{M}$ is estimated as $\norm{M} < \norm[\est]{T}
  + \norm[\est]{M - T}$ with
  $\norm[\est]{T} = 1+\sigma_0\tau_0 L^2$ and
  $\norm[\est]{M - T} = 4(\sigma_0\tau_0)^2/(1+4\sigma_0\tau_0)$.
  The parameter $\gamma$ is set as $\gamma = 2\gamma_1/(\norm[\est]{T}
  + \norm[\est]{M - T})$.
\end{itemize}
We also tested the $\mO(1/k)$ basic Douglas--Rachford
algorithms using the following parameters.
\begin{itemize}
\item \ref{eq:douglas-rachford_}: Douglas-Rachford method as in Table \ref{tab:douglas-rachford}. Again, the discrete cosine transform (DCT) is used
  for applying the inverse of the elliptic operator $\id - \sigma \tau \Delta$.
    Here, $\sigma= 1$, $\tau = 15/\sigma$.

\item \ref{eq:precon-douglas-rachford_}: Preconditioned Douglas-Rachford
  method as in Table \ref{tab:precon-douglas-rachford}. Again, the
  preconditioner is given by   two steps of the symmetric Red-Black Gauss--Seidel iteration.
  The step sizes are chosen as
  $\sigma = 1$, $\tau = 15/\sigma$.
\end{itemize}
Computations were performed for the image \emph{Taj Mahal}
(size 1920$\times$1080
pixels), additive Gaussian noise (noise level $0.1$) and different
regularization parameters
$\alpha$ in~\eqref{eq:rof_primal} using an Intel Xeon CPU (E5-2690v3,
2.6 GHz, 12 cores).
Figure \ref{mahal:denoise} includes the
original image, the noisy image and denoised images with different
regularization parameters.
\begin{figure}  \begin{center}
    \subfloat[Original image]
    {\includegraphics[width=0.49\textwidth]{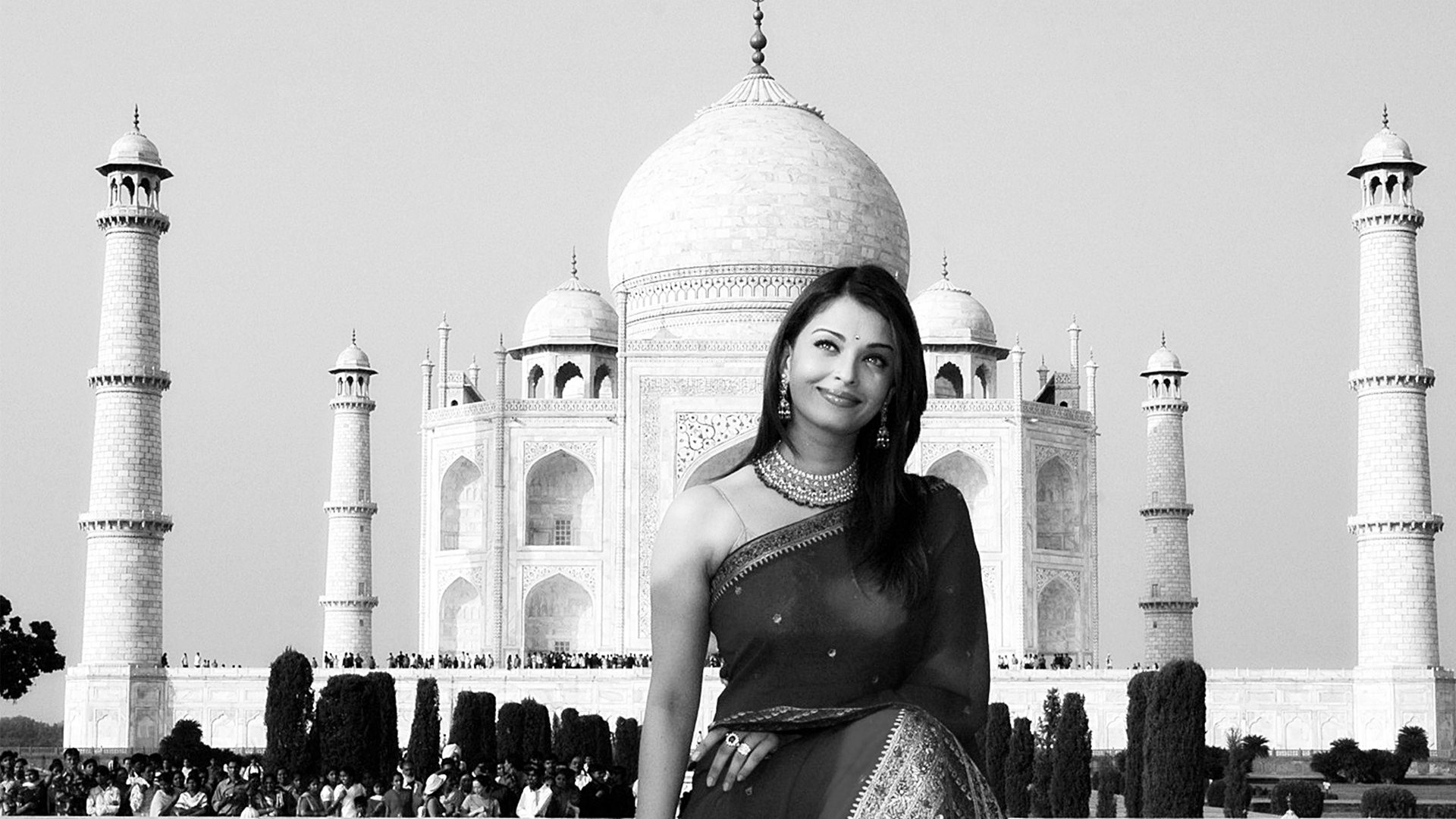}}\
    \subfloat[Noise level: $\sigma = 0.1$]
    {\includegraphics[width=0.49\textwidth]{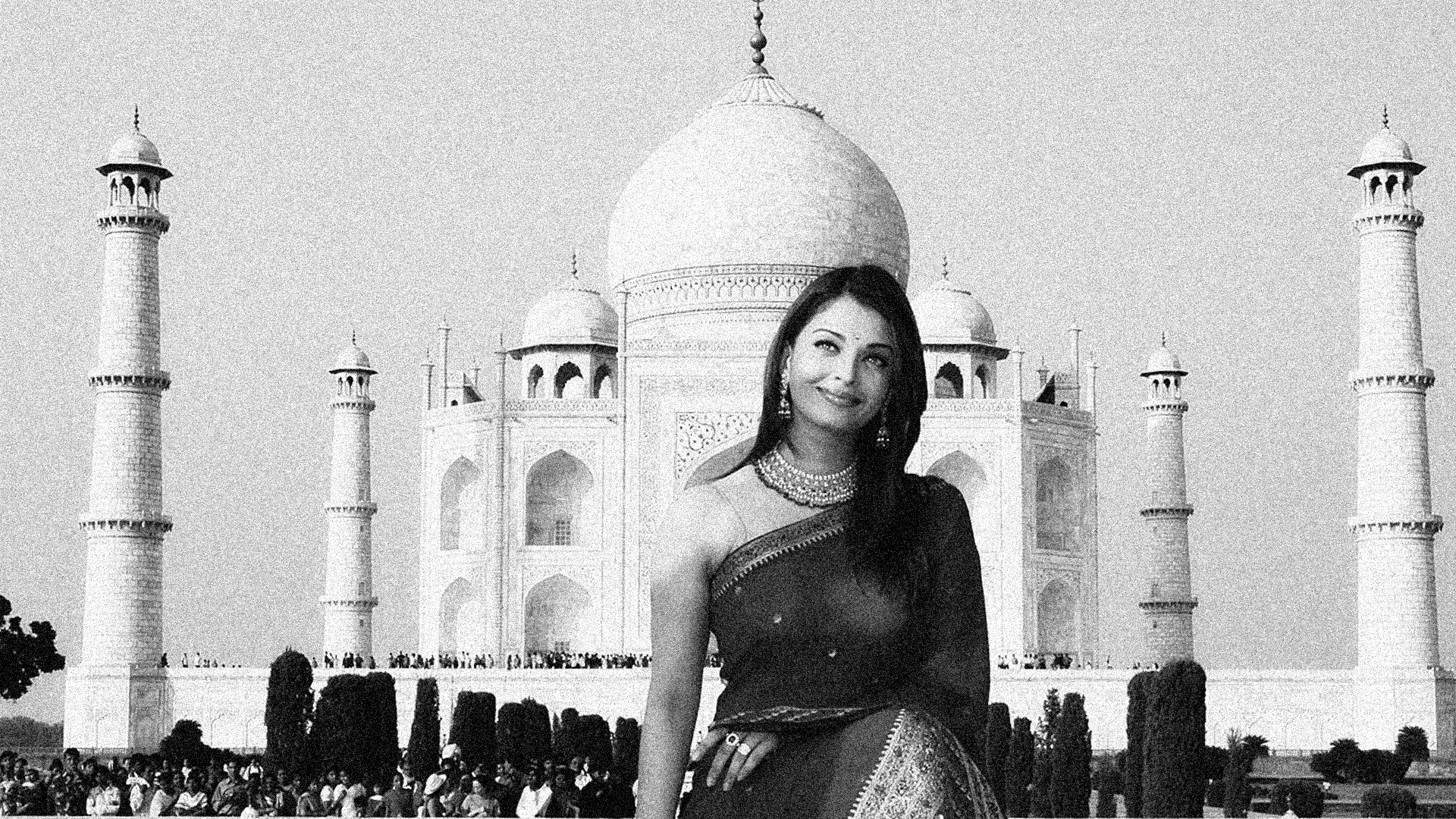}}\\[-0.5em]
    \subfloat[$\alpha = 0.2$]
    {\includegraphics[width=0.49\textwidth]{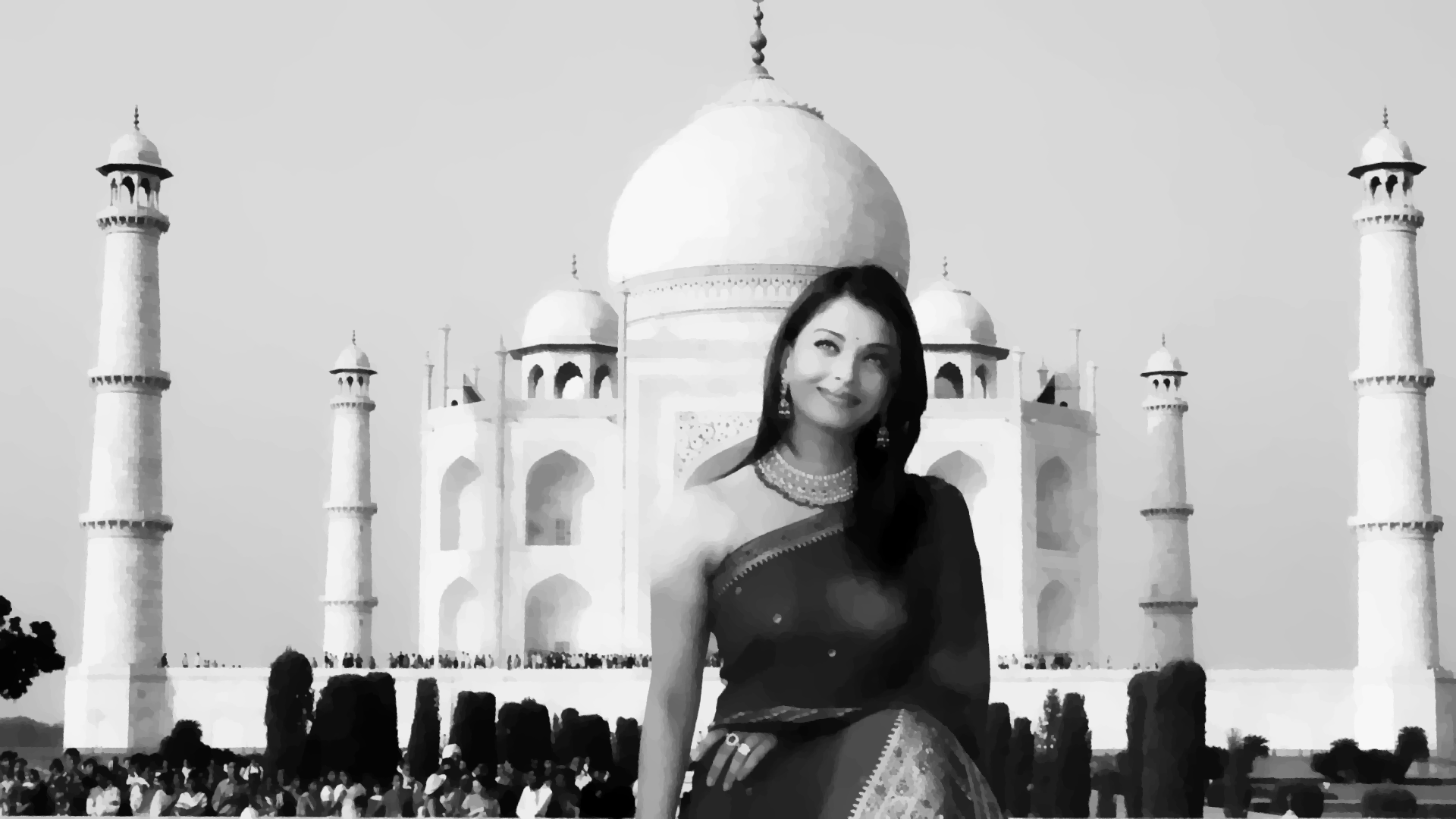}}\
    \subfloat[$\alpha = 0.5$]
    {\includegraphics[width=0.49\textwidth]{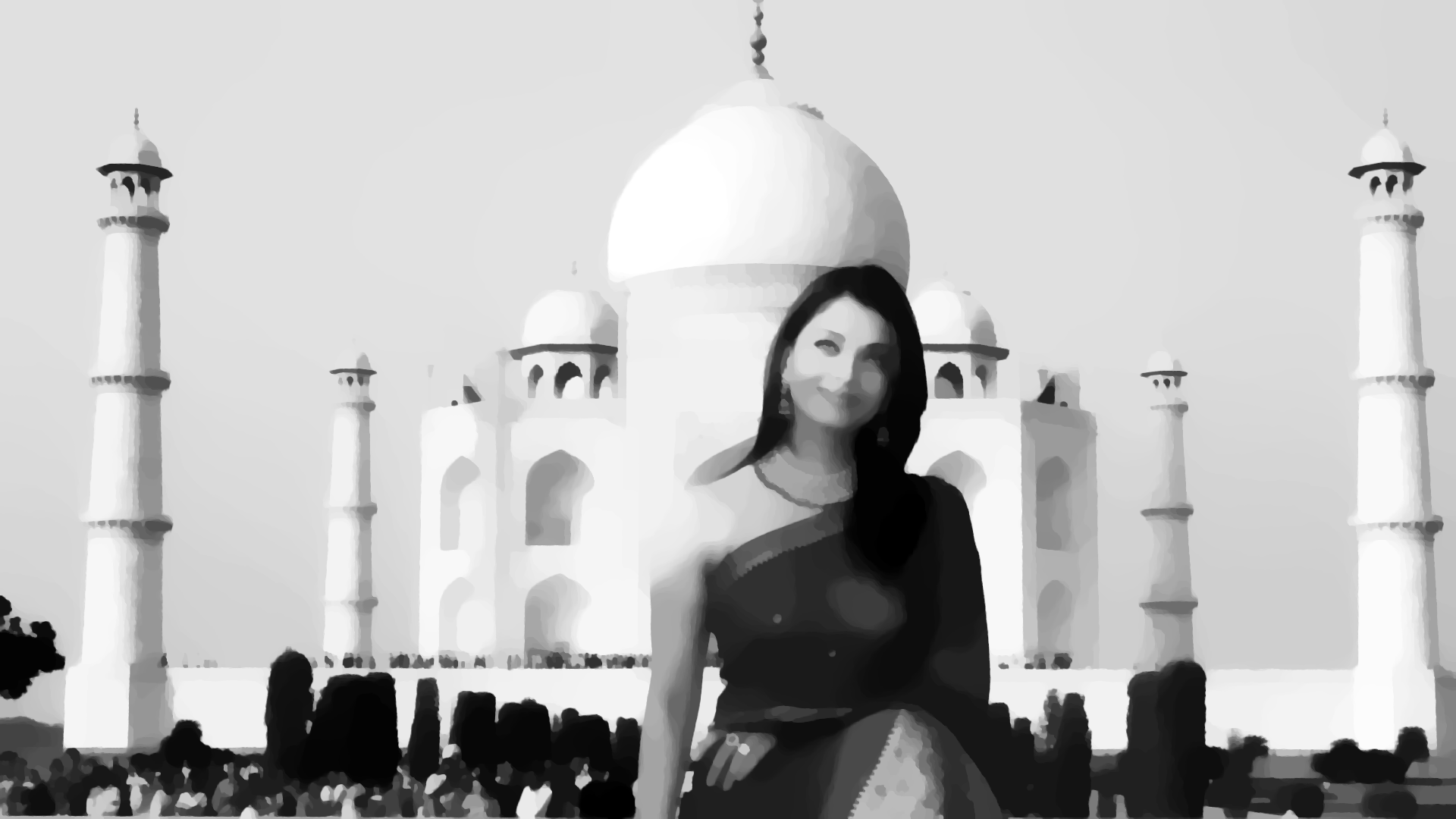}}
  \end{center}
  \vspace*{-0.5em}
  \caption{Results for variational $L^2$-TV denoising.
    All denoised images are obtained with the
    \ref{eq:pre:accelerated-douglas-rachford} algorithm
    which was stopped once the primal-dual gap normalized
    by the number of pixels became less
    than $10^{-7}$. (a) is the original image: \emph{Taj Mahal}
    (1920 $\times$ 1080 pixels, gray). (b) shows a noise-perturbed version
    of (a) (additive Gaussian noise, standard deviation $0.1$), (c)
    and (d) are the denoised images with $\alpha = 0.2$ and $\alpha = 0.5$,
    respectively.}
\label{mahal:denoise}
\end{figure}
It can be seen from Table~\ref{tab:l2tv_denoising} as well as
Figure~\ref{maj:denoise:rate} that the proposed algorithms
\ref{eq:precon-douglas-rachford_} and \ref{eq:pre:accelerated-douglas-rachford}
are competitive both in terms of iteration number and computation time,
especially \ref{eq:pre:accelerated-douglas-rachford}. While in terms of
iteration numbers, \ref{eq:douglas-rachford_} and
\ref{eq:precon-douglas-rachford_} as well as
\ref{eq:accelerated-douglas-rachford_} and
\ref{eq:pre:accelerated-douglas-rachford} perform roughly equally well,
the preconditioned variants massively benefit from the preconditioner
in terms of computation time with each iteration taking only a fraction
compared to the respective unpreconditioned variants.

\begin{table}\centering \begin{tabular}{lr@{\,}r@{\,}lr@{\,}r@{\,}lr@{\,}r@{\,}lr@{\,}r@{\,}l} \toprule
& \multicolumn{6}{c}{$\alpha = 0.2$}
& \multicolumn{6}{c}{$\alpha = 0.5$}\\
\cmidrule{2-7} \cmidrule{8-13}
& \multicolumn{3}{c}{$\varepsilon = 10^{-5}$}
& \multicolumn{3}{c}{$\varepsilon = 10^{-7}$}
& \multicolumn{3}{c}{$\varepsilon = 10^{-5}$}
& \multicolumn{3}{c}{$\varepsilon = 10^{-7}$}\\
\cmidrule{1-13} ALG2&& 134&(1.52s)  && 810&(8.77s) && 368&(3.99s) && 2174&(23.17s)\\
FISTA &&221&(2.80s) && 1659&(20.75s) && 812&(10.14s) && 4534&(56.27s)\\
FDR &&730&(74.39s) &&4927&(494.95s) &&2280&(221.76s) && 12879&(1245.19s)\\
\ref{eq:accelerated-douglas-rachford_}&& 68&(6.00s)  && 357&(31.11s) && 104&(9.24s) && 757&(66.22s)\\
\ref{eq:pre:accelerated-douglas-rachford}&& 75&(1.30s)  && 369&(6.15s) && 128&(2.15s) && 832&(13.70s)\\
\midrule
\ref{eq:douglas-rachford_}  && 57&(5.41s)  && 596&(54.45s) && 133&(12.50s) && 3067&(288.27s)\\
\ref{eq:precon-douglas-rachford_}  && 65&(1.07s)  && 638&(10.06s) && 142&(2.26s) && 3071&(47.87s)\\
\bottomrule \end{tabular}
\vspace*{-0.5em}
\caption{
  Numerical results for the $L^2$-TV image denoising (ROF) problem
  \eqref{eq:rof_primal} with noise level $0.1$ and regularization parameters
  $\alpha = 0.2$, $\alpha = 0.5$. For all algorithms, we use the pair
  $k(t)$ to denote the iteration number $k$ and CPU time cost $t$.
  The iteration is performed until the primal-dual gap~\eqref{eq:l2-tv-gap}
  normalized by the number of pixels
  is below $\varepsilon$.}
\label{tab:l2tv_denoising}
\end{table}

\begin{figure}\begin{center}
\subfloat[Convergence with respect to iteration number.]
{\includegraphics[width=0.49\textwidth,page=1]{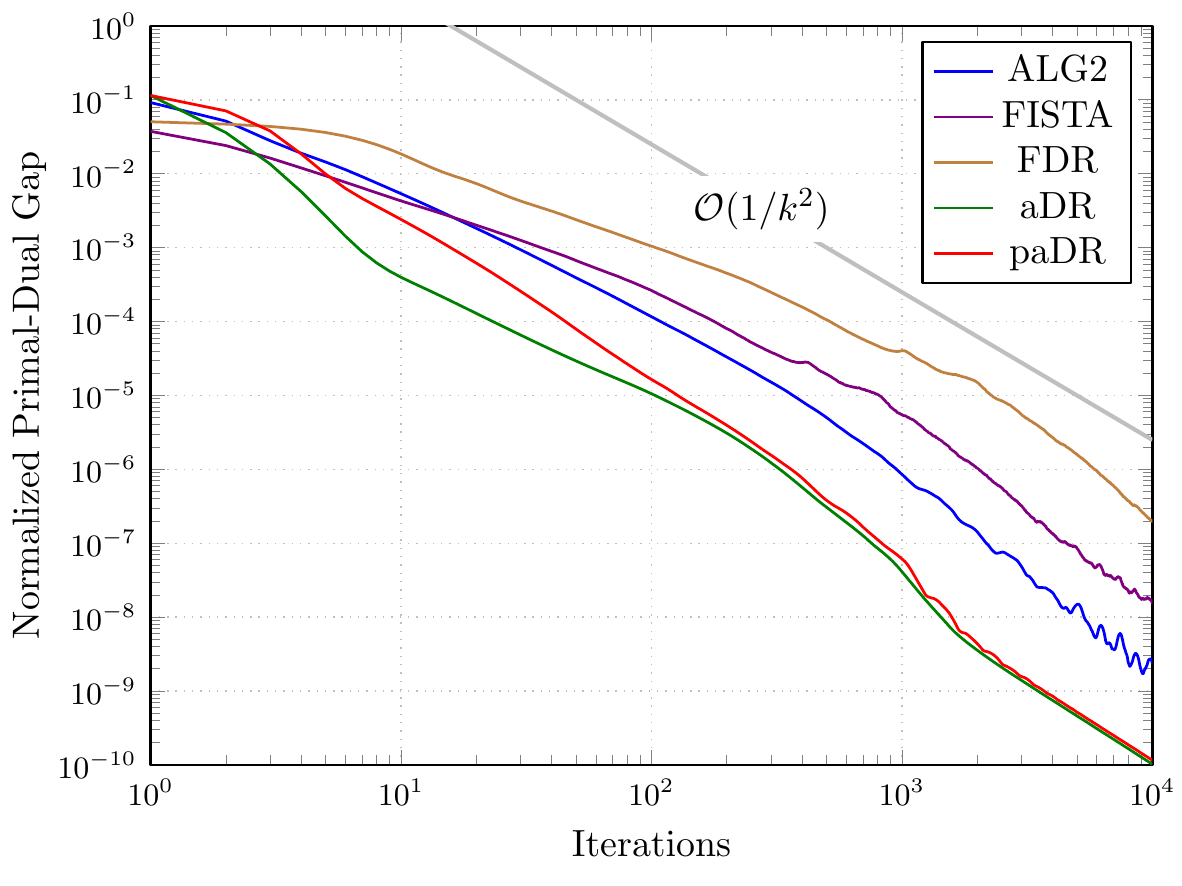}}\hfill
\subfloat[Convergence with respect to computation time.]
{\includegraphics[width=0.49\textwidth,page=2]{figures.pdf}}
\end{center}
\caption{Numerical convergence rate compared with iteration number and
  iteration time for the ROF model with regularization parameter
  $\alpha = 0.5$. The plot in (a) utilizes a double-logarithmic scale
  while the plot in (b) is semi-logarithmic.}
\label{maj:denoise:rate}
\end{figure}

\subsection{\texorpdfstring{$L^2$-Huber-$\TV$ denoising}{L2-Huber-TV denoising}}
There are several ways to approximate the
non-smooth $1$-norm term in~\eqref{eq:rof_primal} by a smooth
functional. One possibility is employing the so-called
\emph{Huber loss function} instead of the Euclidean vector norm,
resulting in
\[
\norm[\alpha,\lambda]{p} = \sum_{i \in \text{pixels} }
\abs[\alpha,\lambda]{p_i}, \qquad
\abs[\alpha,\lambda]{p_i} =
\begin{cases}
\frac{|p_i|^2}{2\lambda} & \text{if} \  |p_i| \leq \alpha \lambda, \\
\alpha |p_i| - \frac{\lambda \alpha^2}{2} & \text{if} \ |p_i| > \alpha\lambda,
\end{cases}
\]
and the associated $L^2$-Huber-$\TV$ minimization problem
\begin{equation}
  \label{eq:huber_primal}
  \min_{u \in X} \ \frac{\norm[2]{u - f}^2}{2} +
  \norm[\alpha,\lambda]{\grad u}.
\end{equation}
An appropriate saddle-point formulation can be obtained from the
saddle-point formulation for the ROF model by replacing $\mG$ by
the functional
$\mG(p) =  \mI_{\{ \|p\|_{\infty} \leq \alpha\}}(p)
+ \frac{\lambda}{2}\|p\|_{2}^2$.
The primal-dual gap
consequently reads as
\begin{equation}
  \label{eq:l2-tv-huber-gap}
  \gap_{\text{$L^2$-Huber-$\TV$}}(u,p) = \frac{\norm[2]{u-f}^2}{2}
  + \|\grad u\|_{\alpha, \lambda} + \frac{\norm[2]{\Div p + f}^2}{2}
  - \frac{\norm[2]{f}^2}{2} + \mI_{\{ \|p\|_{\infty} \leq \alpha\}}
  + \frac{\lambda}{2}\|p\|_{2}^2.
\end{equation}
Since both $\mF$ and $\mG$ are strongly convex functions with respective moduli $\gamma_1 = 1$ and $\gamma_2 = \lambda$, we can use
\eqref{eq:acc-douglas-rachford-sc_} and~\eqref{eq:pre:acc-douglas-rachford-sc_}
to solve the saddle-point problem.
Additionally, we compare to ALG3 from \cite{CP1}, a solution
algorithm for the
same class of strongly convex saddle-point problems.
All algorithms admit an $\mO(\vartheta^k)$-convergence rate for the
primal-dual gap.

\begin{itemize}
\item ALG3:   Accelerated primal-dual algorithm ALG3 for strongly convex
  saddle-point problems as in \cite{CP1}, using parameters
  $\gamma=1$, $\delta = \lambda$, $\mu = 2\sqrt{\gamma \delta}/L$,
  $L = \sqrt{8}$, $\vartheta= 1/(1+\mu)$.

\item \ref{eq:acc-douglas-rachford-sc_}: Accelerated Douglas--Rachford
  method for strongly convex saddle-point problems as in Table
  \ref{tab:acc-douglas-rachford-sc}, using parameters
  $\gamma_{1} = 1$, $\gamma_{2} = \lambda$, $\sigma = 0.2$,
  $\tau = \sigma \gamma_{1}/\gamma_{2}$,
  $\gamma =  2\gamma_1/\bigl(1 + (1 + 2\sigma\gamma_1)\sigma\tau L^2 \bigr)$
  with $L = \sqrt{8}$.

\item \ref{eq:pre:acc-douglas-rachford-sc_}:   Preconditioned accelerated Douglas--Rachford method for strongly convex
  saddle-point problems as in Table \ref{tab:pre:acc-douglas-rachford-sc-pre}.
  Computing three steps of the symmetric Red-Black Gauss--Seidel method
  is employed as preconditioner.
  The parameters are chosen as:  $\gamma_{1} = 1$,
  $\gamma_{2} = \lambda$, $\sigma = 0.15$, $\tau = \sigma \gamma_{1}/\gamma_{2}$,
  $L = \sqrt{8}$. The values $\vartheta$ and
  $\gamma$ are obtained by the procedure
  outlined in Remark~\ref{rem:norm_monotone_iteration}:
  Setting $\vartheta = 1$ first, then performing
  10 times the update       $\norm[\est]{T} \leftarrow 1 + \vartheta^2\sigma\tau L^2$,
  $\norm[\est]{M - T} 
  \leftarrow 4\vartheta^4(\sigma\tau)^2/(1 + 4\vartheta^2\sigma\tau)$,
  $\gamma \leftarrow 2\gamma_1/\bigl(1 + (1 + 2\sigma\gamma_1)\vartheta^{-2}(\norm[\est]{T} 
  + \norm[\est]{M - T}  - 1)\bigr)$, 
  $\vartheta \leftarrow 1/(1 + \sigma\gamma)$.
\end{itemize}
Numerical experiments have been performed for the image
\emph{Man} (1024$\times$1024 pixels, gray) which has been contaminated
with additive Gaussian noise (noise level $0.25$) using the same CPU
as for the $L^2$-$\TV$-denoising experiments. The results for
the parameters $\alpha = 1.0$ and $\lambda = 0.05$ are depicted in
Figure~\ref{fig:man_denoise}.
\begin{figure}\begin{center}
\subfloat[Original image]
{\includegraphics[width=0.24\textwidth]{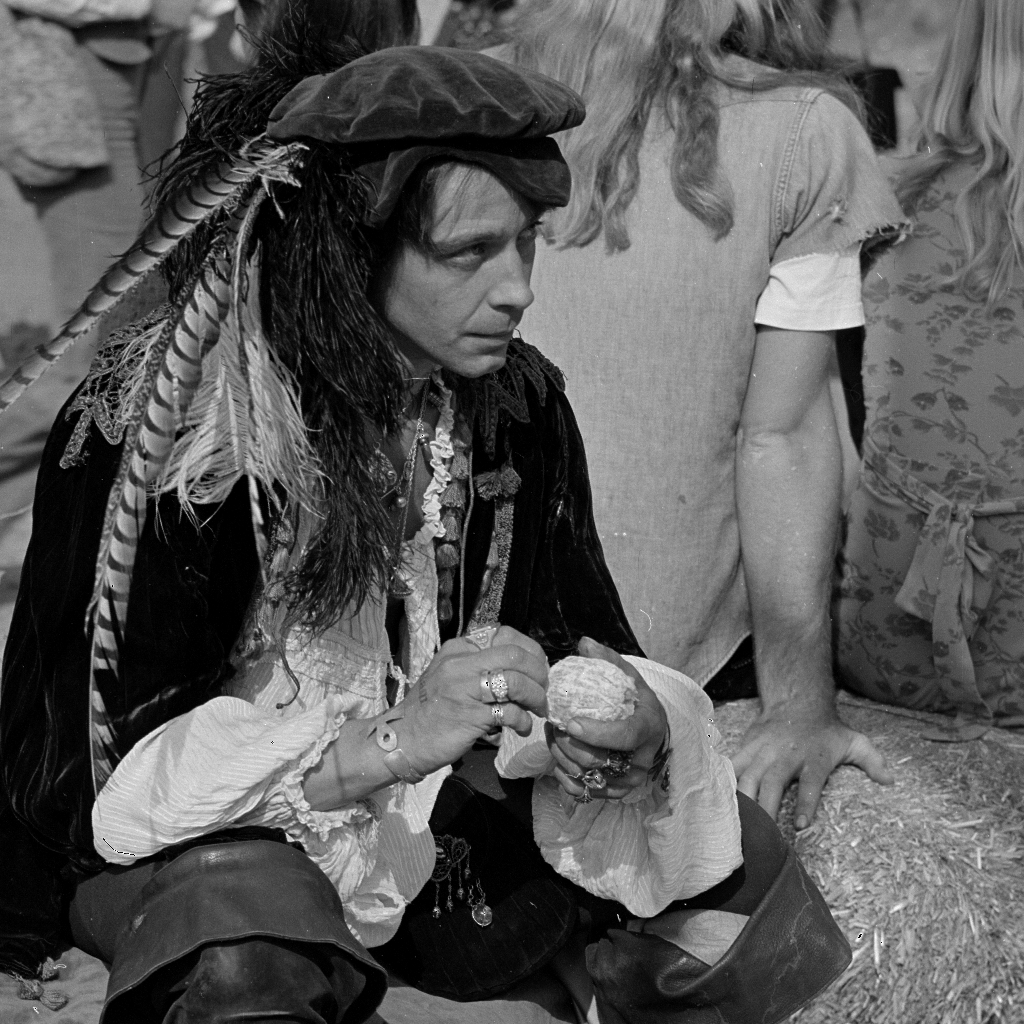}}\
\subfloat[Noise level: $\sigma = 0.25$]
{\includegraphics[width=0.24\textwidth]{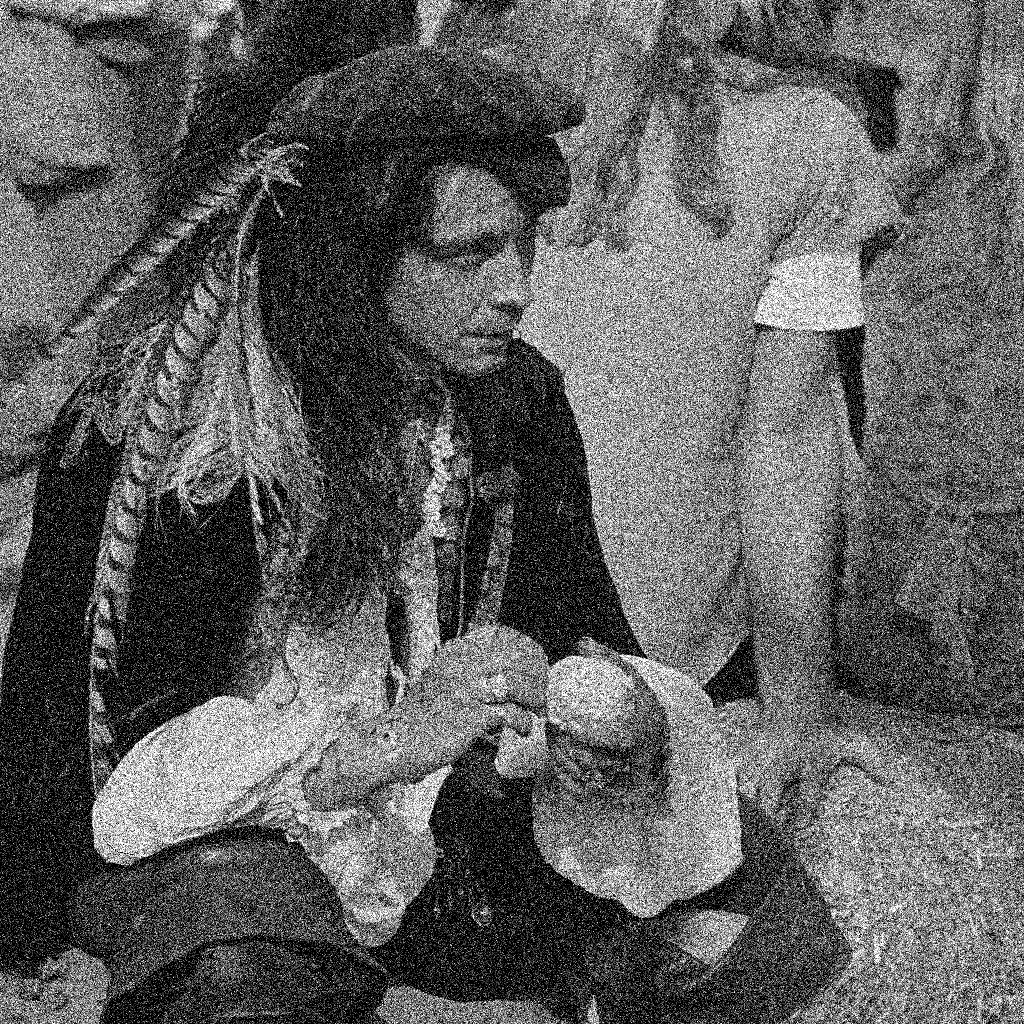}}\
\subfloat[ALG3]
{\includegraphics[width=0.24\textwidth]{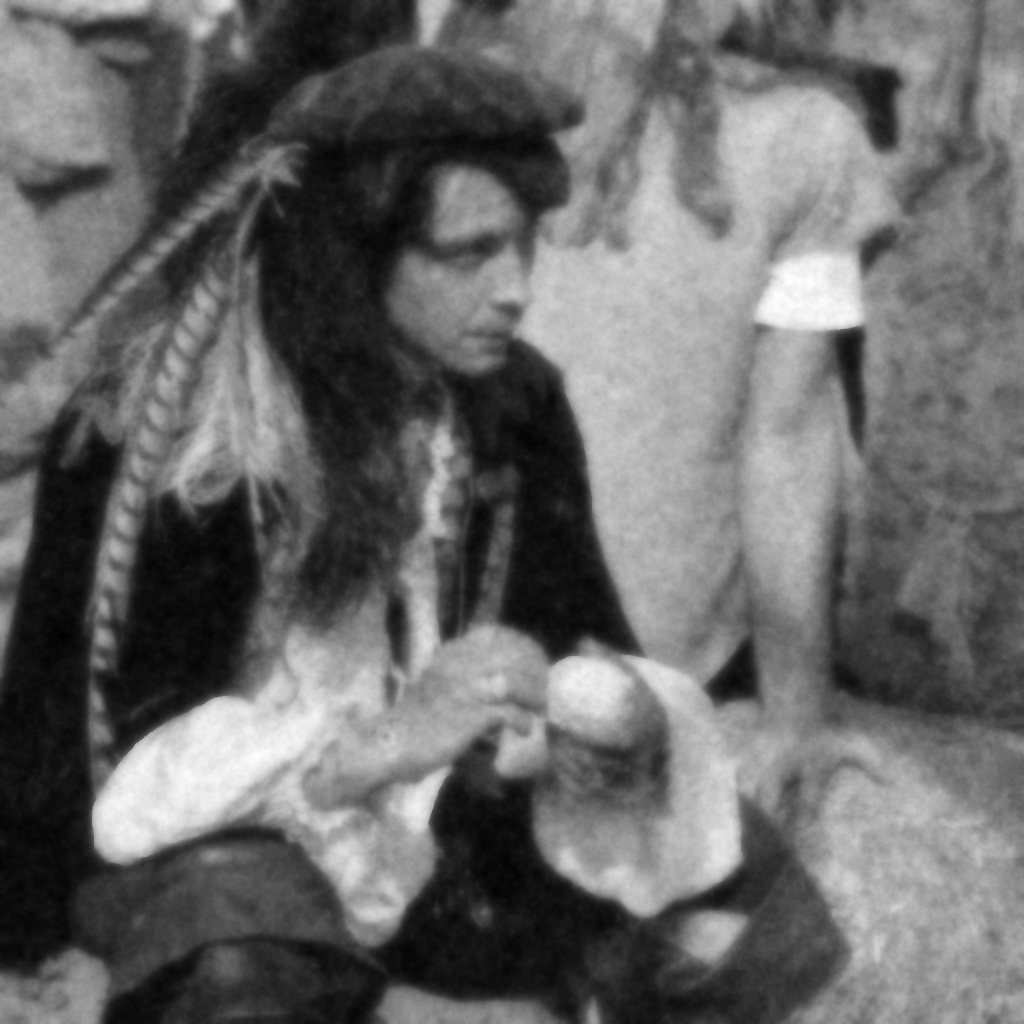}}\
\subfloat[paDR$^{sc}$]
{\includegraphics[width=0.24\textwidth]{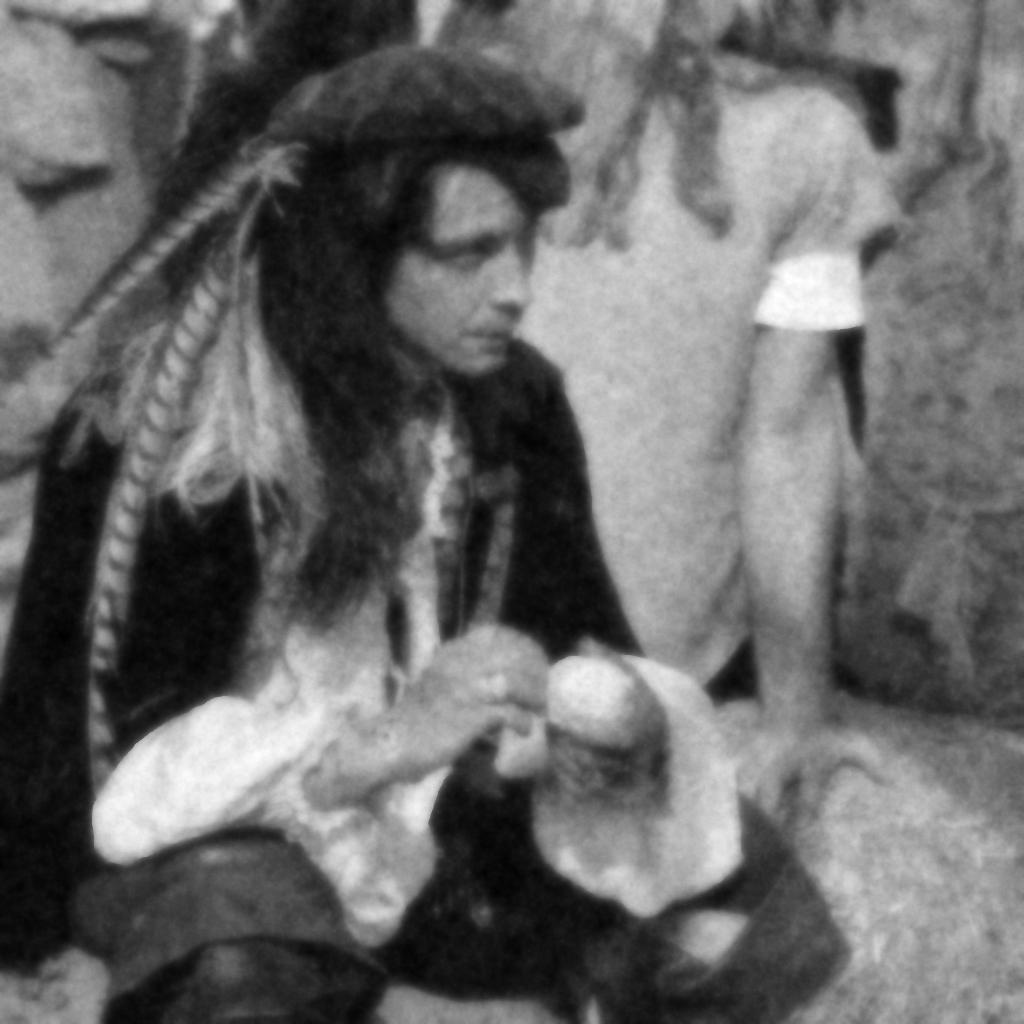}}
\end{center}
\caption{Results for variational $L^2$-Huber-$\TV$ denoising.
  (a) is the original image:
  \emph{Man} (1024 $\times$ 1024 pixels, gray).
  (b) shows a noise-perturbed version of (a) (additive Gaussian noise,
  standard deviation $0.25$), (c) and (d) are the denoised images
  with $\alpha = 1.0$ and
  $\lambda = 0.05$ obtained with algorithm ALG3 and
  \ref{eq:pre:acc-douglas-rachford-sc_}, respectively.
  The iteration was stopped as soon as the primal-dual
  gap normalized by the number of pixels was less than $10^{-14}$.}
\label{fig:man_denoise}
\end{figure}
The table (a) in Figure~\ref{fig:huber} indicates that
\ref{eq:pre:acc-douglas-rachford-sc_} is competitive in comparison
to a well-established fast algorithm suitable for the same class of
saddle-point problems with speed benefits again coming from the
preconditioner. Figure~\ref{fig:huber} (b) moreover shows that the
observed geometric reduction factor for the primal-dual gap is
lower for \ref{eq:acc-douglas-rachford-sc_} and
\ref{eq:pre:acc-douglas-rachford-sc_} than for
ALG3, with the former being roughly equal.
\begin{figure}
  \centering   \subfloat[Comparison of iteration numbers and CPU time cost.]{  \begin{minipage}{0.58\linewidth}    \vspace*{1em}    \begin{tabular}{l@{\,}r@{\,}l@{}r@{\,}l@{}r@{\,}l@{}r@{\,}l}       \toprule
      \multicolumn{9}{c}{$\alpha = 0.05$, $\lambda = 1.0$} \\
      \cmidrule{2-9}    & \multicolumn{2}{c}{\!$\varepsilon = 10^{-8}$\!}
   & \multicolumn{2}{c}{\!$\varepsilon = 10^{-10}$\!}
   & \multicolumn{2}{c}{\!$\varepsilon = 10^{-12}$\!}
   & \multicolumn{2}{c}{\!$\varepsilon = 10^{-14}$\!\!\!}\\
      \cmidrule{1-9}       \!\!ALG3 &82&(0.48s) & 111&(0.51s) & 140&(0.65s)  & 169&(0.78s)\\       \!\!aDR$^{sc}$  &43&(2.00s)& 56&(2.81s) & 71&(3.30s)  & 94&(4.36s)\\
      \!\!paDR$^{sc}$ &44&(0.35s)& 60&(0.47s) & 76&(0.60s)  & 92&(0.72s)\\
      \bottomrule     \end{tabular}
                                                                                        \vspace*{1em}  \end{minipage}}
    \subfloat[Convergence with respect to iteration number.]{
  \begin{minipage}{0.41\linewidth}
    \includegraphics[width=\textwidth,page=3]{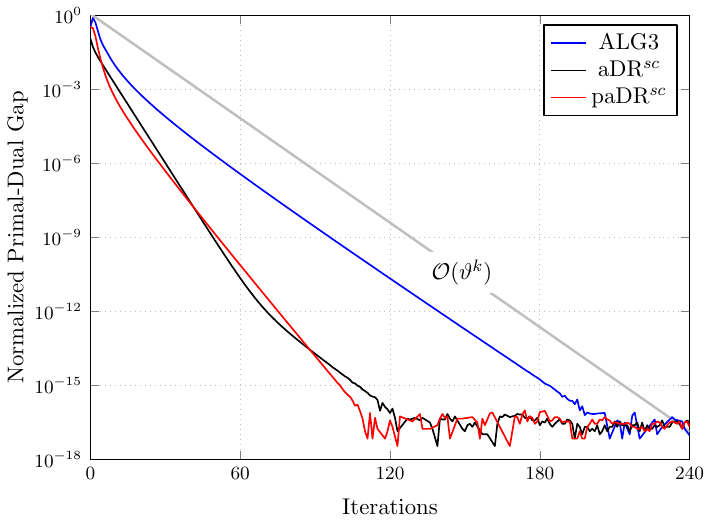}    \vspace*{-0.5em}  \end{minipage}}
    \caption{Convergence results for the $L^2$-Huber-$\TV$ denoising problem.
                          }   \label{fig:huber}
\end{figure}

\section{Summary and conclusions}
\label{sec:summary}

We introduced and studied
novel accelerated algorithms of Douglas--Rachford type
for the solution of structured convex-concave saddle-point problems.
They base on splitting the linear operator and the subgradient operators
in the optimality condition, leading to linear solution steps as well as
proximal mappings in the respective iterations. The acceleration strategies
extensively make use of this specific splitting and rely on strong convexity
assumptions, leading to the same optimal rates that have previously been
reported in the literature, see Table~\ref{tab:summary} for a detailed
overview. All accelerated algorithms may flexibly be preconditioned such
that the corresponding linear iteration step becomes fast and easy to
compute. Numerical experiments indicate that by using
suitable preconditioners such as the symmetric Red-Black
Gauss--Seidel iteration, the proposed algorithms are competitive with
respect to the state of the art for first-order proximal methods. In
particular, considerable speed improvements may be achieved for non-smooth
variational
image-denoising problems for images in the megapixel regime.
Future directions of research may include an extension of the framework to
overrelaxed or inertial variants, the development of adaptive strategies
for step-size parameters update as well as an integration of
instationary preconditioners such as the conjugate gradient (CG) method
into the iteration.

\begin{table}
  \centering
  \begin{tabular}{lllll}
    \toprule
    Method & Prerequisites & Quantity & Type & Rate \\
    \midrule
    Basic~\eqref{eq:douglas-rachford_}/    Preconditioned~\eqref{eq:precon-douglas-rachford_}
           &--- & $(x^k,y^k)$ & weak \\
                & & $\gap_{X_0 \times Y_0}$ & non-ergodic & $\mo(1)$ \\
           & & $\gap_{X_0 \times Y_0}$ & ergodic & $\mO(1/k)$ \\[\medskipamount]
    Accelerated~\eqref{eq:accelerated-douglas-rachford_}/\eqref{eq:pre:accelerated-douglas-rachford}
           & $\mF$ strongly & $\norm{x^k - x^*}^2$
                                      & strong & $\mO(1/k^2)$ \\
           & convex & $y^k$ & weak \\
           & & $\gap_{X_0 \times Y_0}$ & non-ergodic & $\mo(1)$ \\
           & & $\err^p_{Y_0}$ & non-ergodic & $\mo(1/k)$ \\
           & & $\gap_{X_0 \times Y_0}$ & ergodic & $\mO(1/k^2)$ \\
           & & $\err^d$ & ergodic & $\mO(1/k^2)$ \\[\medskipamount]
    Accelerated~\eqref{eq:acc-douglas-rachford-sc_}/\eqref{eq:pre:acc-douglas-rachford-sc_}
           & $\mF$, $\mG$ strongly & $\norm{x^k - x^*}^2$
                                      & strong & $\mo(\vartheta^{k})$ \\
           & convex &  $\norm{y^k - y^*}^2$
                                      & strong & $\mo(\vartheta^{k})$ \\
           & & $\gap$ & non-ergodic & $\mo(\vartheta^{k/2})$ \\
           & & $\gap$ & ergodic & $\mO(\vartheta^k)$ \\
    \bottomrule
  \end{tabular}
  \caption{Summary of the convergence properties of the
    discussed Douglas--Rachford iterations.     Each
    convergence property
    for the non-ergodic sequences also holds for the ergodic sequences
    except for~\eqref{eq:acc-douglas-rachford-sc_}
    and~\eqref{eq:pre:acc-douglas-rachford-sc_} where only
    $\norm{x_{\erg}^k - x^*}^2 = \mO(\vartheta^k)$ and
    $\norm{y_{\erg}^k - y^*}^2 = \mO(\vartheta^k)$ can be obtained.
          }
  \label{tab:summary}
\end{table}

\bibliographystyle{plain}
\bibliography{paper}

\end{document}